\documentclass[reqno,a4paper,twoside]{article}
\usepackage{arxiv}

\usepackage{amscd}

\usepackage[dvipsnames]{xcolor}
\usepackage[colorlinks=true,
linkcolor=Blue,
citecolor=Blue,
linktoc=page]{hyperref}              %

\usepackage[style=alphabetic,
hyperref=true,
backend=biber,
useprefix=true,
giveninits=true]{biblatex}           %
\usepackage{url}                    %
\usepackage{tabu}                   %
\usepackage{bbm, amssymb}           %
\usepackage{eucal}                  %
\usepackage{mathrsfs}               %
\usepackage{upgreek}                %
\usepackage{colonequals}            %
\usepackage{mathtools}              %
\usepackage{caption}
\usepackage{subcaption}
\usepackage{enumitem}

\usepackage{amsthm}                 %
\usepackage[capitalize,noabbrev]{cleveref}              %
\crefformat{section}{#2\S#1#3}

\newtheorem{thm}{Theorem}[section]
\newtheorem{lem}[thm]{Lemma}
\newtheorem{prop}[thm]{Proposition}
\newtheorem{cor}[thm]{Corollary}
\newtheorem{introthm}{Theorem}
\newtheorem{introprop}[introthm]{Proposition}
\newtheorem{introcor}[introthm]{Corollary}

\theoremstyle{definition}
\newtheorem{defn}[thm]{Definition}

\theoremstyle{remark}
\newtheorem{rem}[thm]{Remark}
\newtheorem{ex}[thm]{Example}

\let\skippedproof\proof%
\def\proof{\skippedproof\unskip}

\usepackage{tikz}
\usetikzlibrary{arrows}
\usetikzlibrary{cd}
\usetikzlibrary{decorations.pathreplacing,calligraphy}
\tikzset{>=stealth}
\tikzstyle{uncontracted}=[circle, draw=black, inner sep=1.5]
\tikzstyle{contracted}=[circle, draw=black, fill=black, inner sep=1.5]

\newcommand{\cC}{\CMcal{C}}

\newcommand{\cE}{\CMcal{E}}
\newcommand{\cF}{\CMcal{F}}
\newcommand{\cG}{\CMcal{G}}
\newcommand{\cH}{\CMcal{H}}

\newcommand{\cM}{\CMcal{M}}

\newcommand{\cW}{\CMcal{W}}

\newcommand{\bC}{\mathbf{C}}

\newcommand{\bR}{\mathbf{R}}

\newcommand{\bbP}{\mathbb{P}}

\newcommand{\fA}{\mathcal{A}}

\newcommand{\fM}{\mathcal{M}}

\newcommand{\A}{\mathbb{A}}
\newcommand{\C}{\mathbb{C}}

\newcommand{\Hom}{\mathrm{Hom}}
\newcommand{\HH}{\mathrm{HH}}

\renewcommand{\d}{\mathrm{d}}
\renewcommand{\det}{\mathrm{det}}
\newcommand{\id}{\mathrm{id}}

\newcommand{\<}{\langle}
\renewcommand{\>}{\rangle}

\DeclareMathOperator{\supp}{\mathsf{supp}}

\DeclareMathOperator{\tr}{tr}

\renewcommand{\O}{\CMcal{O}}
\newcommand{\Ext}{\mathrm{Ext}}
\newcommand{\HC}{\mathrm{HC}}

\DeclareMathOperator{\colim}{colim}
\DeclareMathOperator{\Mod}{\mathsf{Mod}}

\DeclareMathOperator{\coh}{\mathsf{coh}}

\newcommand{\N}{\mathbb{N}}

\newcommand{\Z}{\mathbb{Z}}

\DeclareMathOperator{\pc}{\mathsf{PC}}
\newcommand{\R}{\mathbb{R}}
\newcommand{\Mat}{\mathrm{Mat}}

\newcommand{\nrm}[1]{{\left\|#1\right\|}}
\newcommand{\norm}{{\nrm{\cdot}}}

\newcommand{\cont}{\mathrm{cont}}

\newcommand{\one}{\mathbbm{1}}

\newcommand{\bb}{\mathbf{b}}

\newcommand{\cyc}{\mathrm{cyc}}

\newcommand{\op}{\mathrm{op}}
\renewcommand{\P}{\mathbb{P}}

\DeclareMathOperator{\D}{\mathsf{D}}

\newcommand{\hh}{\mathrm{H}}

\newcommand{\sHom}{\cH\kern-2pt\operatorname{om}}

\newcommand{\an}{\mathrm{an}}

\newcommand{\AinfAlg}{{\mathsf{Alg}^\infty}}
\newcommand{\AinfAlgun}{\mathsf{Alg}^{\infty,\un}}
\newcommand{\AinfAlgan}{{\mathsf{Alg}^{\infty,\an}}}

\newcommand{\infbimod}{\operatorname{\mathsf{Mod}}^\infty}

\newcommand{\NMod}{\operatorname{\mathsf{NMod}}}
\newcommand{\elle}{{\ell^e}}

\newcommand{\grMod}{\operatorname{\mathsf{grMod}}}
\newcommand{\grNMod}{\operatorname{\mathsf{grNMod}}}
\newcommand{\bimodl}{\grMod \ell^e}

\newcommand{\catp}{\mathsf{Catp}}

\newcommand{\un}{\mathrm{un}}

\newcommand{\barB}{\mathsf{B}}
\newcommand{\barC}{\mathsf{C}}

\newcommand{\frh}{\mathfrak{h}}
\newcommand{\frg}{\mathfrak{g}}
\newcommand{\sA}{\mathscr{A}}
\newcommand{\Dol}{D}

\newcommand{\perf}{\mathrm{perf}}
\newcommand{\cl}{\mathrm{cl}}

\newcommand{\per}{\mathrm{pert}}

\newcommand{\barT}{\mathsf{T}}

\newcommand{\ev}{\mathsf{ev}}

\title[Cyclic $A_\infty$-Algebras and Calabi--Yau Structures in the Analytic Setting]{Cyclic $A_\infty$-Algebras and Calabi--Yau Structures\\ in the Analytic Setting}
\author[O. van Garderen]{Okke van Garderen}

\begin{document}

\maketitle

\begin{abstract}
  This paper considers $A_\infty$-algebras whose higher products satisfy an analytic bound with respect to a fixed norm.
  We define a notion of right Calabi--Yau structures on such $A_\infty$-algebras and show that these give rise to cyclic minimal models satisfying the same analytic bound.
  This strengthens a theorem of Kontsevich--Soibelman \cite{KS08}, and yields a flexible method for obtaining analytic potentials of Hua-Keller \cite{HK19}.

  We apply these results to the endomorphism DGAs of polystable sheaves considered by Toda \cite{Tod18},
for which we construct a family of such right CY structures obtained from analytic germs of holomorphic volume forms on a projective variety.
As a result, we find a canonical cyclic analytic $A_\infty$-structure on the Ext-algebra of a polystable sheaf, which depends only on the analytic-local geometry of its support. 
  This shows that the results of \cite{Tod18} can be extended to the quasi-projective setting, and yields a new method for comparing cyclic $A_\infty$-structures of sheaves on different Calabi--Yau varieties.
\end{abstract}

\tableofcontents

\section{Introduction}
\subsection{Motivation}

In their seminal paper, Donaldson and Thomas \cite{DT98} proposed a definition of gauge-theoretic invariants for a compact Calabi--Yau threefold $X$.
Writing $\upnu_X$ for the volume form on $X$, these invariants would count critical points of a holomorphic Chern-Simons functional
\[
  \mathsf{CS} \colon \alpha \mapsto \int_X \tr(\Dol \alpha \wedge \alpha + \alpha \wedge \alpha \wedge \alpha) \wedge \upnu_X
\]
on spaces of connections on vector bundles on $X$, taken up to the action of a suitable gauge group.
Defining such invariants directly in this infinite dimensional setup proved difficult, which led Thomas \cite{Tho00} to pursue an algebraic setup involving moduli spaces of finite type.
In the study of Donaldson--Thomas type invariants that followed \cite{KS08,Beh09,JS12} the moduli space of interest is the moduli stack of semistable sheaves $\cM_X$, and the invariants are defined by expressing $\cM_X$ locally as a critical locus of a finite type potential \cite{Joy15}.

The link between the gauge-theoretic and the algebraic setups was explored by Toda \cite{Tod18}, who shows that the moduli space $\cM_X$ is locally controlled by a gauge problem associated to a polystable sheaf $\cF \in\coh X$.
Concretely, one takes a locally free resolution $\cE^\bullet\to \cF$ and considers the DG algebra
\[
  \frg_\cE = \frg_\cE^\bullet \colonequals \left(\Upgamma(X, \sA^{0,\bullet}(\cE nd(\cE))), \Dol, \wedge  \right),
\]
of Dolbeault forms, whose Maurer-Cartan locus represents the critical locus of the Chern-Simons functional.
The gauge problem is made finite dimensional by transferring the DG algebra structure to a minimal model, given by
an $A_\infty$-structure on the Ext-algebra
\[
  \cH_\cE = \left(\Ext^\bullet_X(\cF,\cF), \upmu = (\upmu_n)_{n\in \N}\right).
\]
Toda \cite{Tod18} shows that the Maurer-Cartan locus of $\cH_\cE$ is an analytic space, whose quotient by the gauge group presents an analytic neighbourhood of $\cF$ in $\cM_X$.
Under the assumption that $X$ is 3 Calabi--Yau, it is moreover a critical locus $\{\d \cW_\cF = 0\}$ of the associated analytic potential
\[
  \cW_\cF(\alpha) = \sum_n \upsigma(\alpha)(\upmu_n(\alpha,\ldots,\alpha)),
\]
where $\upsigma\colon \Ext^\bullet(\cF,\cF) \xrightarrow{\ \sim\ } \Ext^{\dim X-\bullet}(\cF,\cF)^\vee$ is the Serre duality pairing.
The potential is the desired replacement for the Chern-Simons functional in this finite dimensional model, and is the correct function for defining refined versions of the DT invariants \cite{KS08,MMNS12,Tod21}.

These results suggest that it should be possible to study the enumerative geometry of $X$ by directly manipulating the $A_\infty$-algebra $\cH_\cE$ using algebraic methods.
However, for this to work there are two additional structures which need to be considered:
\begin{itemize}
\item an \emph{analytic structure} on $\frg_\cE$ and $\cH_\cE$, in the form of a norm $\norm$ for which (as shown by Tu \cite{Tu14}) the higher products satisfy a geometric series bound 
  \begin{equation}\label{eq:convmu}
    \nrm{\upmu_n} < C^n\quad \forall n\geq 1.
  \end{equation}
  
\item a \emph{cyclic structure} on $\cH_\cE$, in the form of the isomorphism $\upsigma \colon \cH_\cE \to \cH_\cE^*$ which (as shown by Polishchuk \cite{Pol01}) is compatible with the natural $A_\infty$-bimodule structures on $\cH_\cE$ and $\cH_E^*$.
\end{itemize}
Neither the analytic nor the cyclic structure is compatible with arbitrary $A_\infty$-algebra morphisms.
Instead one should consider morphisms of \emph{cyclic $A_\infty$-algebras} (defined in \cite{Kaj07,KS09}), and morphisms of \emph{analytic $A_\infty$-algebras} (appearing implicitly in \cite{Tu14,Tod18}) respectively.

The aim of this paper is to develop the theory of $A_\infty$-algebras carrying both an analytic and cyclic structure.
The main goal is to find a general method for constructing analytic cyclic minimal models for a given analytic $A_\infty$-algebra, with the motivating example being the analytic DG algebra $\frg_\cE$ and its analytic cyclic minimal model $(\cH_\cE,\upsigma)$.
To do this we study the behaviour of \emph{right Calabi--Yau structures} on such analytic $A_\infty$-algebras.

\subsection{Main result}
In what follows, $A$ denotes an $A_\infty$-algebra defined over a separable $\C$-algebra $\ell$, which includes $A_\infty$-structures on vector spaces and quivers.

Recall that a right $d$-CY structure on $A$ is a cocycle in the negative cyclic complex of $A$ satisfying a certain nondegeneracy condition \cite{BD19}.
In this paper, we take the Connes complex $\barC_\lambda^\bullet(A)$ as our model for negative cyclic cohomology $\HC_\lambda^\bullet(A)$, and an $d$-CY structure will be a cocycle in $\barC_\lambda^{-d}(A)$.
Given an analytic structure $\norm \colon A \to \R$ as above we define the subcomplex
\[
  \barC_\lambda^{\an,\bullet}(A) \colonequals \{\ \upphi \in \barC_\lambda^\bullet(A) \mid \upphi \text{ satisfies a geometric series bound in } \norm\ \},
\]
and call a right CY structure analytic if it lies in $\barC_\lambda^{\an,\bullet}(A)$.
We show that the subcomplex $\barC_\lambda^{\an,\bullet}(A)$ admits pullbacks with respect to $A_\infty$-morphisms satisfying a bound analogous to \labelcref{eq:convmu}, and is therefore a good cohomology theory on a category of analytic $A_\infty$-algebras.
Working inside this category, we find the following main theorem.

\begin{introthm}[\cref{mainthm}]\label{thm:D}
  Let $A$ be an analytic $A_\infty$-algebra admitting an analytic minimal model $\hh(A)$ which is finite dimensional and unital.
  Then every analytic right $d$-CY structure $\upphi \in \barC_\lambda^{\an,-d}(A)$ defines a $d$-cyclic analytic minimal model
  \[
    (\hh(A)^\upphi,\upsigma^\upphi)
  \]
  which depends only on the class $[\upphi]_\an \in \HC_\lambda^{\an,-n}(A)$ up to an isomorphism of cyclic analytic $A_\infty$-algebras.
  If $B$ is a second such $A_\infty$-algebra, then every analytic quasi-isomorphism $f\colon B \to A$ induces a cyclic analytic $A_\infty$-isomorphism
  \[
    (\hh(B)^{f^*\upphi},\upsigma^{f^*\upphi})\ \cong_{\cyc,\an}\ (\hh(A)^\upphi,\upsigma^\upphi).
  \]
\end{introthm}

The proof of our main theorem follows the work of Amorim--Tu \cite{AT22}, proving a type of ``analytic Darboux theorem'' for a version of the homotopy-invariant version of cyclic structures \cite{Cho08}, which are related to the noncommutative symplectic structures of \cite{KS09}.

Cyclic $A_\infty$-algebras are especially interesting in the 3-CY setting, because the cyclic $A_\infty$-structure is captured by a potential: given such a 3-cyclic $A_\infty$-algebra $(A,\upsigma)$ there is an associated noncommutative potential $W \in \widehat\barT V \colonequals \prod_{n\geq 0} V^{\otimes n}$ in the completed tensor algebra over $V = (A^1)^\vee$.
If $A$ is moreover equipped with an analytic structure, then this potential lies in an analytic subring
\[
  \widetilde\barT V \colonequals \left\{\ (v_n)_{n\in\N} \in \widehat\barT V \quad\middle|\quad \exists C \text{ such that } \nrm{v_n} < C^n \text{ for all } n>0\ \right\},
\]
and is therefore an analytic potential in the framework of Hua--Keller \cite{HK19}.
Analytic morphisms induce algebra morphisms between these analytic subrings, so in the 3-CY setting our theorem can be interpreted as follows.

\begin{introcor}[{\cref{potentialsmaincor}}]
  In the situation of \cref{thm:D} suppose that $\upphi$ is a right 3-CY structure, and write $V_A = \hh^1(A)^\vee$.
  Then there is a canonical analytic potential
  \[
    W^\upphi \in \widetilde \barT V_A,
  \]
  which is well-defined up to an automorphism of $\widetilde \barT V_A$.
  Moreover, if $f\colon B\to A$ is a quasi-isomorphism and $V_B = \hh^1(B)^\vee$ then there is an induced isomorphism $g\colon \widetilde\barT V_A \to \widetilde\barT V_B$ such that
  \[
    g(W^\upphi) = W^{f^*\upphi}.
  \]
\end{introcor}

The main theorem and its corollary provide a bridge between infinite dimensional settings, involving e.g. normed DG algebras, with finite dimensional settings, such as those considered in \cite{HK19}.
In particular, it applies to the gauge-theoretic setup mentioned above.

\subsection{Analytic CY structures in complex geometry}

We apply \cref{thm:D} to describe the cyclic analytic minimal models of Toda's DGA $\frg_\cE$ for a perfect complex $\cE$ of vector bundles on a smooth projective variety $X$ of dimension $d$.
If $X$ is Calabi--Yau with holomorphic volume form ${\upnu_X}\in \hh^0(X,\Omega^d_X)$, then there is an induced right CY structure: the bounded linear functional
\begin{equation}\label{eq:volumetrace}
  \frg_\cE \to \C,\quad \alpha \mapsto \int_X \tr(\alpha)\wedge {\upnu_X}
\end{equation}
defines a cocycle $\upphi^{\upnu_X} \in \barC_\lambda^{\an,-d}(\frg_\cE)$ which is the nondegenerate in the appropriate sense.
The associated cyclic analytic minimal model $(\cH_\cE^{\upnu_X},\upsigma^{\upnu_X})$ recovers the one found in \cite{Pol01,Tu14,Tod18}.

There are in general many other choices of analytic right CY structures, each of which yields a cyclic analytic minimal model of $\frg_\cE$.
We identify a family of such CY structures corresponding to \emph{holomorphic volume germs} along the support $Z = \supp \cE$, by which we mean the equivalence class of a differential form
\[
  \upnu \in (\Omega_X^d)_Z \colonequals \operatorname{\underset{U\supset Z}{\colim}} \Upgamma(U,\Omega_X^d)
\]
on analytic open neighbourhoods $U\supset Z$, which is nonvanishing when $U$ is sufficiently small.
Such a germ determines a bounded linear functional on the DG subalgebra
\[
  \frg_{\cE|_U,c} \colonequals \left\{ \upxi \in \frg_\cE \mid \supp \upxi \subset U \right\} \subset \frg_{\cE,c},
\]
of Dolbeault forms with compact support in any sufficiently small neighbourhood $U\supset Z$.
This functional induces an analytic right CY structure on $\frg_{\cE|_U,c}$, which we transfer to a right CY structure on $\frg_\cE$ using an explicit $A_\infty$-quasi-isomorphism.
This leads to the following theorem.

\begin{introthm}[{\cref{cycminmodcan}}]\label{thm:A}
  Let $X$ be a smooth projective vartiety of dimension $d$, and $\cE \in \D^\perf(X)$ a complex with support $Z\subset X$.
  Then every volume germ $\upnu \in (\Omega^d_X)_Z$ determines a class $[\upphi^\upnu]_\an \in \HC_\lambda^{\an,-d}(\frg_\cE)$ of an analytic right $d$-CY structure, and hence a $d$-cyclic analytic minimal model
  \[
    (\cH_\cE^\upnu,\upsigma^\upnu) = (\Ext_X^\bullet(\cE,\cE),\upmu^\nu,\upsigma^\nu),
  \]
  which is uniquely defined up to an isomorphism of cyclic analytic $A_\infty$-algebras.
  Moreover, $\cH_\cE^\upnu$ is isomorphic to $\cH_\cE$ as an ordinary (non-cyclic) analytical $A_\infty$-algebra.
\end{introthm}

The above theorem applies to the setting of \cite{Tod18} when $\cE$ is a resolution of a polystable sheaf $\cF \in \coh X$ whose support does not meet the canonical divisor $K_X$.
Some examples considered in the enumerative geometry literature include point sheaves \cite{BBS13}, and sheaves on contractible curves \cite{Sze08,Katz08}.
If $X$ is a threefold, then there is a noncommutative analytic potential
\[
  W^\upnu_\cE \in \widetilde\barT \Ext^1(\cE,\cE)^\vee,
\]
whose abelianisation $\cW^\upnu_\cE = (W^\upnu_\cE)^{\mathrm{ab}}$ is an ordinary analytic function whose critical locus describes the moduli space of semistable sheaves around $\cF$.
Hence, the gauge-theoretic setup can be generalised from the projective CY-3 case in \cite{Tod18} to open analytic neighbourhoods of projective CY-3 folds equipped with a choice of volume form.

In our setup, both the $A_\infty$-algebra structure on the Ext-algebra of $\cE$ and the volume germ $\upnu$ are determined in terms of analytic-local geometry around the support of $\cE$ in $X$.
One would therefore expect that the cyclic $A_\infty$-algebra $(\cH_\cE^\upnu,\upsigma^\upnu)$ is not sensitive to the global geometry of $X$.
To show this we consider embeddings of such neighbourhoods into other projective varieties.

\begin{introthm}[{\cref{embeddingminmod}}]\label{thm:B}
  Let $X,X'$ be smooth projective varieties of dimension $d$, and let $Y$ be an open submanifold of $X$ with an open embedding into  $X'$, as in the following diagram:
  \[
    \begin{tikzcd}
      X & Y\ar[l,hook',"i",swap] \ar[r,"f"] & X'.
    \end{tikzcd}
  \]
  Let $\cE'$ be a perfect complex on $X'$ with  cohomological support $f(Z) \subset f(Y)$ for some compact $Z\subset Y$, and let $\upnu \in (\Omega^d_{X'})_Z$ be a volume germ.
  Then for any $\cE \in \D^\perf(X)$ such that $\cE|_Y \simeq f^*\cE'$ there exists an analytic quasi-isomorphism
  \[
    \frg_{\cE} \simeq_\an \frg_{\cE'}
  \]
  which identifies the classes $[\upphi^\upnu]$ and $[\upphi^{f^*\upphi}]$ via pullback.
  In particular, there is an analytic cyclic $A_\infty$-isomorphism between the $d$-cyclic analytic minimal models
  \[
    (\cH_\cE^{f^*\upnu},\upsigma^{f^*\upnu}) \cong_{\an,\cyc} (\cH_{\cE'}^\upnu,\upsigma^\upnu).
  \]
\end{introthm}

\begin{introcor}[{\cref{potentialsembedcor}}]
  If $X,X'$ are threefolds, then the noncommutative potentials $W_{\cE'}^\upnu$ and $W_{\cE}^{f^*\upnu}$ are related by an analytic change of variables $\widetilde\barT\Ext^1(\cE,\cE)^\vee \xrightarrow{\sim} \widetilde\barT\Ext^1(\cE',\cE')$.
\end{introcor}

One implication of \cref{thm:B} is that there is a canonical cyclic $A_\infty$-structure on the Ext-algebra of a compactly supported sheaf $\cF$ on a quasi-projective Calabi--Yau threefold $Y$ equipped with a fixed volume form.
Indeed, applying \cref{thm:A} to a resolution of $\cF$ on any compactification $X = \overline Y$, it follows from \cref{thm:B} that the cyclic $A_\infty$-algebra obtained does not depend on these choices.

For the reader who is exclusively interesting in projective Calabi--Yau geometry, we stress that \cref{thm:B} also allows one to compare the cyclic analytic $A_\infty$-structure (hence the enumerative geometry) of different projective CY varieties on a common analytic neighbourhood.
As a motivating example, we compute the cyclic analytic $A_\infty$-structure on the Ext-algebra of an arbitrary  point sheaf.

\begin{introprop}[{\cref{cyclicstructonpoint}}]\label{prop:C}
  Let $p\in X$ be a point on a smooth projective variety $X$ of dimension $d$, and $\cF \to \O_p$ a locally free resolution for the associated point sheaf.
  Then for every volume germ $\upnu \in (\Omega^d_Y)_p$ there is an isomorphism of cyclic analytic $A_\infty$-algebras
  \[
    \cH^\upnu_\cF \cong_{\an,\cyc} \left(\ \bigwedge^\bullet T_{\kern-.5pt o\kern.5pt}\A^d,\ \upsigma^\lambda\colon \xi \mapsto \lambda(\xi \wedge -)\ \right),
  \]
  where the right hand side denotes the graded algebra of polyvectors at the origin $o\in \A^d$, with cyclic structure induced by some linear form $\lambda\colon \bigwedge^d T_{\kern-.5pt o\kern.5pt}\A^d \xrightarrow{\sim} \C$.
\end{introprop}

\subsection{Structure of the paper}

In \cref{sec:prelims} we recall the definition of (analytic) $A_\infty$-algebras, and prove some new results about inverses of analytic morphisms that are used in later sections.

In \cref{sec:bimodules} we define analytic $A_\infty$-bimodules and analytic Hochschild cohomology, and we prove some new results about the invertibility of analytic $A_\infty$-bimodule maps.

In \cref{sec:cyclic} we define analytic (homotopy) cyclic structures, used to prove \cref{thm:D}.

In \cref{sec:geometry} we apply these results to the DG algebras of Dolbeault differential forms, which is used to prove \cref{thm:A}, \cref{thm:B}, and \cref{prop:C}.
This section is largely self-contained.

\subsection*{Acknowledgements}
The author would like to thank Ben Davison and Louis Ioos for interesting discussions, Sarah Scherotzke for support.
He also thanks an anonymous referee for suggesting important improvements to an earlier version of this paper.

\section{Analytic \texorpdfstring{$A_\infty$-algebras}{A-infinity algebras}}\label{sec:prelims}

In this section we recall the definition of $A_\infty$-algebras and morphisms, and introduce the relevant category of analytic $A_\infty$-algebras based on the convergence conditions used in \cite{Tu14}.

In what follows let $\ell$ be a separable $\C$-algebra and $\bimodl$ the graded $\C$-linear category of graded bimodules $V = V^\bullet$ which are nonzero in finitely many degrees.
Given $V,W\in\bimodl$,
\[
  \hom_{\elle}^\bullet(V,W) \colonequals \bigoplus_{n\in\Z} \Hom_{\elle}(V^n,V^{n+\bullet})
\]
denotes the graded vector space of $\elle$-linear maps, and $\Hom_\elle(V,W)$ the subspace of graded morphisms.
We write $|f|$ for the degree of a homogeneous element $f\in\hom_\elle^\bullet(V,W)$.
The category $\bimodl$ carries the usual graded tensor product $\otimes \colonequals \otimes_\ell$ and shift functor $[1]\colon \bimodl \to \bimodl$.
We write $s\in \hom^{-1}_\elle(V,V[1])$ for the shift map, acting as the identity $V^i \to V^i = (V[1])^{i-1}$ on each graded component.

\subsection{\texorpdfstring{$A_\infty$-algebras}{A-infinity algebras} and morphisms}

$A_\infty$-algebras are commonly defined as sequences of multilinear maps either of the form $V^{\otimes n} \to W$ with varying degrees or of the form $V[1]^{\otimes n} \to W[1]$ with all maps being homogeneous.
In this section we adopt the latter convention, adopting the notation
\[
  \fM^\bullet(V,W) \colonequals \prod_{n=0}^\infty \hom_\elle^\bullet(V[1]^{\otimes n}, W[1]),
\]
for the graded vector space of these sequences $f = (f_n)_{n\in\N}$, and we write $\overline\fM(V,W) \subset \fM(V,W)$ for the subspace of sequences with $f_0 = 0$.
The exists various compositions between sequences in $\fM(V,W)$, which are most cleanly described using the bar construction
\[
  \barB V \colonequals \left(\bigoplus_{n\in\N} V[1]^{\otimes n}, \Delta\right)
\]
which is a cofree coaugmented conilpotent coalgebra over $\ell$ with the decomposition coproduct $\Delta$.
The cofreeness implies that every sequence $f \in \fM(V,V) \simeq \hom_\elle(\barB V,V[1])$ induces a unique coderivation $\widetilde f \colon \barB V \to \barB V$ via the formula (see Tradler \cite{Tra08})
\[
  \widetilde f \colonequals \sum_{k\in\N} \sum_{i,j\geq 0} \id_{V[1]}^{\otimes i}  \otimes f_k \otimes \id_{V[1]}^{\otimes j},
\]
and likewise every coderivation defines a sequence in $\fM(V,V)$ by composing with the obvious projection $\uppi \colon \barB V \to V[1]$.
In terms of the bar construction, $A_\infty$-algebras are defined as follows.
\begin{defn}
  An $A_\infty$-algebra is a pair $(A,\upmu)$ of $A\in\bimodl$ and a map $\upmu \in \overline\fM^1(A, A)$ satisfying the condition ${\widetilde \upmu}^2 = 0$, or equivalently $\upmu \circ \widetilde \upmu = 0$.
\end{defn}
The $A_\infty$-condition implies that $\upmu_1^2 = 0$, so that $\upmu_1$ makes $A[1]$ into a cochain complex.
We say an $A_\infty$-algebra is \emph{compact} if the cohomology with respect to $\upmu_1$ is finite dimensional over $\C$.

\begin{ex}\label{exDGalg}
  If $(A,\d, \cdot)$ is a DG algebra, then the associated $A_\infty$-structure on is given by the compositions $\upmu_1 = -s\circ \d \circ s^{-1}$ and $\upmu_2 = -s^{-1} \circ (-\cdot-) \circ (s^{-1}\otimes s^{-1})$.
\end{ex}
 
Given $V,W\in\bimodl$, every sequence $f \in \fM(V,W) \simeq \hom_\elle(\barB V,W[1])$ induces a  cohomomorphism $\widehat f \colon \barB V \to \barB W$ between the associated coalgebras, via the formula
\[
  \widehat f \colonequals \sum_{k\in\N} \sum_{n_1,\ldots,n_k\in\N} f_{n_1} \otimes \cdots \otimes f_{n_k},
\]
and every cohomomorphism between bar-coalgebras can be described in this way.
These lifts are used to define (pre-)morphisms of $A_\infty$-algebras.
\begin{defn}\label{def:premorb}
  A pre-morphism between $A_\infty$-algebras $A$ and $B$ is an element $f\in \overline\fM^0(A,B)$.
  A pre-morphism is a \emph{morphism} if it satisfies $\widetilde \upmu_B \circ \widehat f = \widehat f \circ \widetilde \upmu_A$, or equivalently $\upmu_B \circ \widehat f = f\circ \widetilde \upmu_A$.
\end{defn}
The morphism condition implies that $f_1\colon A[1] \to B[1]$ is a morphism of cochain complexes.
A morphism $f\in \Hom_\AinfAlg(A,B)$ is called a \emph{quasi-isomorphism} if $f_1$ is a quasi-isomorphism.

By lifting pre-morphisms to cohomomorphisms, one obtains an associative composition between \hbox{(pre-)mor}\-phisms which we will write as
\[
  \diamond \colon \fM^0(B,C) \times \fM^0(A,B) \to \fM^0(A,C),\quad g \diamond f \colonequals g \circ \widehat f.
\]
With this composition $A_\infty$-algebras form a category $\AinfAlg$ with morphism spaces $\Hom_\AinfAlg(A,B)$.
We remark that the composition of two quasi-isomorphisms is again a quasi-isomorphism.

\begin{defn}
A \emph{unit} for an $A_\infty$-algebra $A = (A,\upmu)$ is an injective map $\one \colon \ell[1] \to A[1]$ such that
  \[
    \begin{aligned}
      \upmu_2(\one(l),a) = la,\quad \upmu_2(a,\one(l)) = (-1)^{|a|+1} al \quad &\text{for all } l\in \ell[1],\  a\in A[1]\\
      \upmu_n(\ldots,\one(l),\ldots) = 0 \quad &\text{for all } n\neq 2,
    \end{aligned}
  \]
  and a triple $(A,\upmu,\one)$ is called a \emph{unital $A_\infty$-algebra}.
\end{defn}

If $A$ and $B$ are unital $A_\infty$-algebras, with units $\one_A$ and $\one_B$, then one can consider the subset $\Hom_\AinfAlg^\un(A,B) \subset \Hom_\AinfAlg(A,B)$ of morphisms $f$ satisfying
\[
  \begin{aligned}
    f_1 \circ \one_A &= \one_B,\\
    f_{i+j+1} \circ (\id^{\otimes i} \otimes \one_A \otimes \id^{\otimes j}) &= 0\quad \text{for all } i+j > 0,
  \end{aligned}
\]
which are called \emph{unital} $A_\infty$-morphisms.
The composition of two unital morphisms is again unital, so unital $A_\infty$-algebras again form a category $\AinfAlgun$.
Finally, we consider a weakening of the unital condition in cohomology.

\begin{defn}
  A \emph{weak unit} or \emph{cohomological unit}, is a map $\one\colon \ell[1] \to A[1]$ for which the image is $\upmu_1$-closed and which satisfies the following relations in $\upmu_1$-cohomology:
  \[
    \upmu_2([\one(l)],[a]) = [la],\quad \upmu_2([a],[\one(l)]) = (-1)^{|a|+1}[al].
  \]
\end{defn}

One can likewise define a category of cohomologically unital $A_\infty$-algebras, with morphisms consisting of $A_\infty$-morphisms $f$ for which $[f_1(\one(l))] = [\one(l)]$.

\subsection{Homotopies}\label{ssec:homotopies}

$A_\infty$-algebras are a model for homotopical algebra and it is therefore natural to consider morphisms up to homotopy.
There are various definitions, ranging from algebraic \cite{LefevreHasegawa03,K01} to more continuous variants \cite{FOOO09}.
These models lead to equivalent theories of homotopy, but differ in flexibility and expressivity.
Below we use a variation of \cite{FOOO09}.

Let $\pc^\infty([0,1]) \subset L^\infty([0,1])$ denote the ring of piece-wise smooth functions, i.e. bounded real functions on $[0,1]$ that are smooth outside of a discrete set, identified if they agree almost everywhere.
There is an associated DG algebra of piecewise differential forms
\[
  \Omega^\bullet_{[0,1]} \colonequals \{ x + y\d t \mid x,y  \in \pc^\infty([0,1]),\ x\text{ is continuous}\},
\]
with the usual wedge product and differential $\d \colon x + y \d t \mapsto \tfrac{\partial x}{\partial t} \d t$.
Then we consider for any $A_\infty$-algebra $(A,\upmu)$ the tensor product $A_\infty$-algebra
\[
  \big(\Omega^\bullet_{[0,1]}\otimes A,\ \upmu^\otimes\big).
\]
where the $A_\infty$-algebra structure is defined as in \cite{L11} by
\begin{align}  
  \upmu_1^\otimes &= \d \otimes \id +  \id \otimes \upmu_1, \label{eq:difftensor}\\
  \upmu_k^\otimes &= (-\cdots-) \otimes \upmu_k, \label{eq:highertensor}
\end{align}
where $(-\cdots-)$ denotes the multiplication on $k$ elements in $\Omega^\bullet_{[0,1]}$, and maps are applied along the isomorphism $(\Omega_{[0,1]}^\bullet\otimes A[1])^{\otimes k} \cong (\Omega_{[0,1]}^\bullet)^{\otimes k} \otimes A[1]^{\otimes k}$ observing the usual Koszul sign conventions.
The tensor product fits into a diagram of three quasi-isomorphisms
\[
  A \xrightarrow{\ a \mapsto 1\otimes a\ } \Omega^\bullet_{[0,1]} \otimes A \xrightarrow{(\ev_0, \ev_1)} A \times A
\]
where $\ev_t((x+y\d t)\otimes a) = x(t) \cdot a$.
These maps make $\Omega^\bullet_{[0,1]} \otimes A$ into a path-space object as in \cite[Definition 4.2.1]{FOOO09},
which leads to the following notion of a homotopy.

\begin{defn}
  A \emph{homotopy} between two morphisms $f,g\in \Hom_\AinfAlg(A,B)$ is an $A_\infty$-morphism $H\in \Hom_\AinfAlg(A,\Omega^\bullet_{[0,1]}\otimes B)$ such that $\ev_0\circ H = f$ and $\ev_1 \circ H = g$.
\end{defn}

To make the condition more explicit we can write $H \in \overline\fM^0(A,\Omega^\bullet_{[0,1]}\otimes B)$ as
$H = H_x + H_y \d t$ for some coefficient functions $t\mapsto H_x^t\in \overline\fM^0(A,B)$ and $t\mapsto H_y^t \in \overline\fM^{-1}(A,B)$.
Then as observed in \cite[(4.2.41)]{FOOO09} the map $H$ is a homotopy if and only if the following two conditions hold:
\begin{align}
  H_x^t &\in \Hom_\AinfAlg(A,B) & \text{for all } t\in[0,1], \label{eq:homotopycondone}\\
  \frac{\partial}{\partial t} H_x^t &= \upmu_B \circ (\widehat H_x^t \otimes H_y^t \otimes \widehat H_x^t) + H_y^t \circ \widetilde \upmu_A
  & \text{for almost all } t\in [0,1], \label{eq:homotopycondtwo}
\end{align}
It follows from a general argument \cite[Proposition 4.2.37]{FOOO09} homotopy defines an equivalence relation between $A_\infty$-morphisms, which is compatible with the composition.
We give a direct proof below, which will also apply in the analytic setting.

\begin{lem}\label{lem:homequiv}
  Homotopy defines an equivalence relation $\sim$ on $\Hom_\AinfAlg(A,B)$.
\end{lem}
\begin{proof}
  It is clear that $f\sim f$ for any $f\in \Hom_\AinfAlg(A,B)$ via the constant homotopy $H^t = f$.
  Likewise if $f\sim g$ via a homotopy $H = H_x + H_y \d t$ then $G^t = H^{-t}_x - H_y^{-t} \d t$ defines a homotopy $g\sim f$.
  Now if $f\sim g$ and $g\sim h$ via homotopies $H = H_x + H_y\d t$ and $G= G_x + G_y\d t$, then we consider the pre-morphism $G*H \in \overline\fM^0(A,\Omega^\bullet_{[0,1]}\otimes B)$ with
  \[
    (G*H)^t = \begin{cases}
      H_x^{2t} + 2 H_y^{2t}\d t  & t\leq 1/2\\
      G_x^{2t-1} + 2 G_y^{2t-1}  & t>1/2
    \end{cases}
  \]
  We remark that this is well-defined: the coefficient functions are piecewise smooth, and $(G*H)_x$ is continuous since $H_x^1 = g = G_x^0$ by assumption.
  By construction $(G*H)_x^t \in \Hom_\AinfAlg(A,B)$ for all $t\in [0,1]$, showing \eqref{eq:homotopycondone}.
  The identity \eqref{eq:homotopycondtwo} holds for almost all $t \in [0,1/2)$ since
  \[
    \begin{aligned}
      \frac{\partial}{\partial t} (G*H)_x^t
      &= 2\cdot (\upmu \circ (\widehat H_x^{2t} \otimes H_y^{2t} \otimes \widehat H_x^{2t}) + H_y^{2t} \circ \widetilde \upmu)
      \\&= \upmu \circ (\smash{\widehat{(G*H)}}_x^t \otimes (G*H)_y^t \otimes \smash{\widehat{(G*H)}}_x^t) + (G*H)_y^t \circ \widetilde \upmu,
    \end{aligned}
  \]
  and a similar computation shows the identity holds for almost all $t\in (1/2,1]$.
  Hence $G*H$ is a homotopy between $f = \ev_0 \circ (G*H)$ and $h = \ev_1\circ (G*H)$.
\end{proof}

\begin{lem}\label{lem:diamondhom}
  If $f\sim f'$ and $g\sim g'$ then $f\diamond g \sim f'\diamond g'$.
\end{lem}
\begin{proof}
  Let $H$ and $G$ be the homotopies for $f\sim f'$ and $g\sim g'$.
  Then $H\diamond g$ is clearly a homotopy between $f\diamond g \sim f'\diamond g$.
  Now $f'$ lifts to an $A_\infty$-morphism $\id \otimes f'$ between the path objects, providing a homotopy $(\id\otimes f')\diamond G$ between $f'\diamond g$ and $f'\diamond g'$.
\end{proof}

We say two morphisms $f,g$ in $\AinfAlg$ are \emph{homotopy inverse} if $f\diamond g$ and $g\diamond f$ are homotopic to the identity.
In this case $f$ and $g$ are said to be \emph{homotopy equivalences} and their source/target $A_\infty$-algebras are homotopy equivalent.

\subsection{Minimal models}
\label{sec:minmods}
  
Recall that an $A_\infty$-algebra $(A,\upmu)$ is called \emph{minimal} if $\upmu_1 = 0$.
It is often convenient to replace an $A_\infty$-algebra by a homotopy equivalent minimal one: a minimal model.
\begin{defn}
  A \emph{minimal model} for $A$ is a minimal $\hh(A) \in \AinfAlg$ with a quasi-isomorphisms $P = P_A \in \Hom_\AinfAlg(A,\hh(A))$ and $I = I_A \in\Hom_\AinfAlg(\hh(A),A)$ such that:
  \begin{equation}\label{eq:minmodord}
      P\diamond I = \id_{\hh(A)}\quad\text{ and }\quad   I\diamond P \sim \id_A.
  \end{equation}
\end{defn}
It is well-known that every $A_\infty$-algebra admits a minimal model, which can be constructed explicitly using Kadeishvili's homotopy transfer formula \cite{Kad80} applied to the cohomology algebra of $A$.

Given $A_\infty$-algebras $A,B\in\AinfAlg$ with a fixed choice of minimal models $\hh(A)$ and $\hh(B)$, we obtain for every morphism $f\in\Hom_\AinfAlg(A,B)$ an induced map
\[
  \hh(f) \colonequals P_B \diamond f \diamond I_A \in \Hom_\AinfAlg(\hh(A),\hh(B)).
\]
It is well-known that if $f$ is a quasi-isomorphism, then $\hh(f)$ is an isomorphism.
The inverse at the level of minimal models can be constructed using a ``tree-formula'' that we recall below.

\begin{defn}
  $\O(n)$ denotes the set of \emph{planar rooted trees on $n$ leaves}, in which every internal node has valency $\geq 3$, and $\O(n,d) \subset \O(n)$ denotes the subset of trees with $d$ internal nodes.
\end{defn}

Note that there is a unique tree in $\O(1) = \O(1,0)$ and that for $n>1$ each tree $T\in \O(n)$ is uniquely determined by the ordered list of rooted subtrees $T_1,\ldots,T_k$ starting in the internal node of $T$ connected to the root.

For any pre-morphism $f\in \overline\fM^0(A,B)$ and map $g_1 \colon B[1]\to A[1]$ we define a map for every $T\in \O(n)$ as follows:
for the unique tree $T\in \O(1)$ we set $g_T = g_1$ and for general $T\in\O(n)$
\[
  g_T = -g_1 \circ f_k \circ (g_{T_1} \otimes \cdots \otimes g_{T_k}).
\]
where $T_1,\ldots,T_k$ are the subtrees starting in the internal node in $T$ connected to the root.
The following result is classical, and can be easily verified using the recursive formula.
\begin{prop}\label{inversemap}
  Let $f\in\overline\fM^0(A,B)$ be a pre-morphism such that $f_1\colon A[1]\to B[1]$ admits an inverse $g_1\colon B[1]\to A[1]$.
  Then the pre-morphism $g\in \overline\fM(B,A)$ defined by
  \[
    g_n = \sum_{T\in \O(n)} g_T,
  \]
  is an inverse to $f$. If $f$ is moreover an $A_\infty$-morphism then $g$ is again an $A_\infty$-morphism.
\end{prop}
\begin{lem}\label{quasiisominmodinv}
  Every quasi-isomorphism induces an isomorphism between minimal models.
\end{lem}
\begin{proof}
  Fix minimal models $\hh(A)$ and $\hh(B)$ with maps $I_A,I_B,P_A,P_B$ as in \eqref{eq:minmodord}.
  Then given any quasi-isomorphism $f\in \Hom_\AinfAlg(A,B)$ the induced map $\hh(f) = P_B \diamond f \diamond I_A$ is again a quasi-isomorphism, because $P_B$ and $I_A$ are quasi-isomorphisms.
  But then $\hh(f)_1$ is an isomorphism by minimality and the inverse $\hh(f)^{-1}$ exists by \cref{inversemap}.
\end{proof}
\begin{cor}\label{uniqminmod}
  Minimal models are unique up to $A_\infty$-isomorphism.
\end{cor}
\begin{proof}
  This follows directly from \cref{quasiisominmodinv} when taking $\hh(A)$ and $\hh(B)$ to be two different minimal models of $A=B$ with morphism $f=\id_A$.
\end{proof}

Given a quasi-isomorphism $f$ the composition $g = I_A \diamond \hh(f)^{-1}\diamond P_B$ is a homotopy inverse to $f$.
To see this, note that there are identities
\begin{align}  
  g\diamond f =
  I_A \diamond \hh(f)^{-1}\diamond P_B \diamond f
  \,\sim\,
  I_A \diamond \hh(f)^{-1}\diamond \underbrace{P_B \diamond f \diamond I_A}_{\hh(f)} \diamond P_A
  \,\sim\,
  I_A \diamond P_A \sim \id_A, \label{eq:homotinvgf}\\
  f\diamond g =
  f \diamond I_A \diamond \hh(f)^{-1}\diamond P_B
  \,\sim\,
  I_B \diamond \underbrace{P_B \diamond f \diamond I_A}_{\hh(f)} \diamond \hh(f)^{-1} \diamond P_B
  \,\sim\,
  I_B \diamond P_B \sim \id_B, \label{eq:homotinvfg}
\end{align}
induced by the homotopies $I_A\circ P_A \sim \id_A$ and $I_B\circ P_B \sim \id_B$.
Hence, every quasi-isomorphism is also a homotopy equivalence.

\subsection{Analytic algebras and morphisms}\label{sec:norms}

We now consider the category $\grNMod \elle$ of \emph{graded normed $\elle$-bimodules}, by which we mean pairs $(V,\norm_V)$ of a bimodule $V\in\bimodl$ and a norm $\norm_V \colon V\to \R$ on the underlying ungraded vector space so that for any $l,r\in \ell$ the multiplication $v \mapsto l v r$ is bounded in $\norm_\ell$.
Between any $V,W \in \grNMod \elle$ there is a graded normed vector space
\[
  \hom_\elle^\cont(V,W) \colonequals \{ f\in \hom_\elle(V,W) \mid \nrm{f}_\op < \infty \}
\]
of bounded morphisms with respect to the operator norm $\nrm{f}_\op \colonequals \sup_{\|v\|_V = 1} \|f(v)\|_W$.
The shift naturally extends to a functor $[1]\colon \grNMod \elle \to \grNMod\elle$, and for any pair $V,W\in\grNMod \elle$ the tensor product $V\otimes W$ is again normed via the projective tensor norm
\[
  \|z\|_\pi \colonequals \inf\{\textstyle \sum_i \|v_i\|_V \|w_i\|_W \mid \textstyle z = \sum_i v \otimes_\C w \},
\]
where the infimum is over all presentations of an element $z\in V\otimes W$ as a sum of pure tensors in the tensor product $V\otimes_\C W$ over $\C$.
If $f_i\colon V_i \to W_i$ are bounded bimodule morphisms between $V_1,V_2,W_1,W_2 \in \grNMod\elle$, the tensor product $f_1\otimes f_2$ has operator norm
\[
  \nrm{f_1\otimes f_2}_\op = \nrm{f_1}_\op \nrm{f_2}_\op < \infty,
\]
and is therefore a bounded bimodule map $V_1\otimes V_2 \to W_1\otimes W_2$.
Henceforth, we will drop the subscripts on the above norms to reduce clutter.

Now given a normed $\ell$-bimodule $V\in\grNMod \elle$ and $r \in \R$ with $r > 0$ we consider the following normed version of the bar construction:
\[
  \barB(V,r) \colonequals (\barB V, \norm_r),\quad \nrm{\textstyle\sum v_n}_r \colonequals \sup_{n\in\N} \nrm{v_n}r^{-n},
\]
noting that the supremum in the definition of $\norm_r$ is finite for any finite sum in $\barB V$.
Bounded maps out of these normed $\ell$-bimodules are characterised by the following lemma.

\begin{lem}\label{analboundedness}
  Let $V,W\in\grNMod \elle$ and let $f \in \fM(V,W)$. Then the following are equivalent:
  \begin{enumerate}
  \item there exists $C>0$ such that $\nrm{f_n} < C^n$ for all $n\geq 1$,
  \item there exists $r>0$ such that $f\colon \barB(V,r) \to W[1]$ is bounded,
  \end{enumerate}
  if additionally $f_0 = 0$ then this is also equivalent to:
  \begin{enumerate}[resume]
  \item for every $r'>0$ there exists an $r>0$ such that $\widehat f \colon \barB(V,r) \to \barB(W,r')$ is bounded,
  \end{enumerate}
  if additionally $V=W$ then this is also equivalent to:
  \begin{enumerate}[resume]
  \item for every $r'>0$ there exists an $r>0$ such that $\widetilde f \colon \barB(V,r) \to \barB(V,r')$ is bounded.
  \end{enumerate}
\end{lem}
\begin{proof}
  (1$\implies$2) Without loss of generality we may assume that $f_0 = 0$.
  Let $C>0$ be such that $\nrm{f_n} < C^n$ for all $n\geq 1$ and fix some $r < C^{-1}$.
  Then the map $f \colon \barB(V,r) \to W[1]$ satisfies for every $v = \sum_{n\geq 1} v_n \in \barB(V,r)$
  \[
    \nrm{f(v)} \leq \sum_{n\geq 1} \nrm{f_n(v_n)} < \sum_{n\geq 1} C^n \nrm{v_n} \leq \sum_{n\in\N} (Cr)^n \nrm{v}_r = \frac{\nrm{v}_r}{1-Cr},
  \]
  where the final equality for the geometric series holds because $Cr < 1$.
  It follows that $f$ has operator norm bounded by $(1-Cr)^{-1} < \infty$.

  (2$\implies$1) If $f \colon \barB(V,r) \to W[1]$ is bounded in the operator norm by some constant $K>1$, then for each $n\geq 1$ and $v_n \in V[1]^{\otimes n}$ the map $f_n$ satisfies
  \[
    \nrm{f_n(v_n)} = \nrm{f(v_n)} \leq K \nrm{v_n} r^{-n} \leq (K/r)^n \nrm{v_n}.
  \]
  Hence, picking any $C> K/r$ it follows that $\nrm{f_n} < C^n$ for all $n\geq 1$.

  (1$\implies$3) Let $C>0$ be such that $\nrm{f_n} < C^n$ for all $n\geq 1$, fix $r'>0$ and let $K = \min\{r',1\}$.
  Then assuming $f_0=0$ and expanding the formula for $\widehat f$, we find that for every $v = \sum_{n\in\N} v_n \in \barB V$
  \[
    \begin{aligned}
      \nrm{\widehat f(v)}_{r'} &\leq \sum_{n\geq 1} \sum_{n_1+\cdots+n_k = n} \nrm{f_{n_1}\otimes \cdots \otimes f_{n_k}(v_n)} (r')^{-k}
      \\&\leq \sum_{n\in\N} \sum_{n_1+\cdots+n_k = n} C^n K^{-k} \nrm{v_n}
      \\&\leq \sum_{n\in\N} (\text{\# partitions of n})\cdot \left(\frac CK\right)^n \nrm{v_n}.
    \end{aligned}
  \]
  The number of partitions of $n$ is exponentially bounded by a function $q^n$ for some $q>0$.
  Picking $r< \frac{K}{qC}$ the geometric series in $\frac{rqC}{K}$ again converges, which shows that
  \[
    \nrm{\widehat f(v)}_{r'} \leq \sum_{n\in\N} \left(\frac{rqC}{K}\right)^n \nrm{v}_r = \frac{K}{K-rqC}  \nrm{v}_r.
  \]
  As $q$ and $r$ can be chosen independently of $v$, it follows that $\widehat f \colon \barB(V,r) \to \barB(V,r')$ is bounded.

  (1$\implies$ 4) Let $C>1$ be such that $\nrm{f_n}<C^n$ for all $n\in\N$, fix $r'>0$ and let $K=\min\{r',1\}$.
  Then assuming $f_0 = 0$ and expanding the formula for $\widetilde f$ yields for every $v=\sum_{n\in\N} v_n \in\barB V$
  \[
    \begin{aligned}
      \nrm{\widetilde f(v)}_{r'} &\leq \sum_{n\in\N} \sum_{i+k+j=n} \nrm{(\id^{\otimes i} \otimes f_k \otimes \id^{\otimes j})(v_n)} (r')^{-(i+j+1)}
      \\&\leq \sum_{n\in\N} \sum_{i+k+j=n} C^k K^{-(i+j+1)}\nrm{v_n}
      \\&\leq \sum_{n\in\N} n^3 \left(\frac CK\right)^n\nrm{v_n}.
    \end{aligned}
  \]
  Picking $q>0$ such that $n^3 < q^n$ for all $n$ and an $r<\frac{K}{qC}$, it again follows that
  \[
    \nrm{\widetilde f(v)}_{r'} \leq \sum_{n\in\N} \left(\frac{rqC}{K}\right)^n \nrm{v}_r = \frac{K}{K - rqC}\nrm{v}_r,
  \]
  hence $\widetilde f$ is bounded when considered as a linear map $\barB(V,r) \to \barB(V,r')$.
  
  (3,4$\implies$2) For any $r'$, the projection $\uppi\colon \barB(V,r') \to V[1]$ is obviously bounded, so if a map $\widehat f$ or $\widetilde f$ is bounded it follows that $f = \uppi \circ \widehat f = \uppi \circ \widetilde f$ is again bounded.
\end{proof}

We define the space of \emph{analytic sequences of multilinear maps} between $V$ and $W$ as
\[
  \fA(V,W) = \bigcup_{r>0} \hom_\elle^\cont(\barB(V,r), W[1]) \subset \fM(V,W),
\]
and note that its elements are sequences $f = (f_n)_{n\in\N}$ satisfying one of the equivalent conditions in the above lemma.
As before, we write $\overline\fA(V,W) \subset \fA(V,W)$ for the sequences with $f_0 = 0$.

\begin{defn}
  An \emph{analytic $A_\infty$-algebra} is a pair $(A,\upmu)$ of $A\in \grNMod\elle$ and $\upmu\in \overline\fA^1(A,A)$ which satisfies ${\widetilde \upmu}^2 = 0$.
\end{defn}

\begin{ex}
  Let $(A,\d)$ be a DG algebra, and suppose $\norm\colon A \to \R$ is any norm on the underlying vector space.
  Then the differential and multiplication define an analytic $A_\infty$-algebra structure if and only if both are bounded.
  In particular, this includes the case of Banach DG algebras.
\end{ex}

\begin{ex}
  Let $(A,\upmu)$ be a finite dimensional $A_\infty$-algebra such that $\upmu_n = 0$ for $n\gg 0$.
  Then $(A,\upmu)$ is analytic for any choice of norm on $A$.
\end{ex}

\begin{rem}
It follows from \cref{analboundedness} that the definition given above is equivalent to the $A_\infty$-algebras satisfying a convergence condition used in \cite{Tu14,Tod18}.
\end{rem}

Given two analytic $A_\infty$-algebras $A$ and $B$, we say a pre-morphism $f$ is \emph{analytic} if it lies in the subspace $f\in \fA^0(A,B)$.
\Cref{analboundedness} implies that the composition of analytic pre-morphisms is well-defined.

\begin{lem}
  For any three analytic $A_\infty$-algebras $A,B,C$, the composition restricts to a map
  \[
    \diamond \colon \fA^0(B,C) \times \fA^0(A,B) \to \fA^0(A,C).
  \]
\end{lem}
\begin{proof}
  Given $f\in \fA^0(A,B)$ and $g\in \fA^0(B,C)$, it follows from \cref{analboundedness} that there exist $r_1,r_2>0$ such that the maps
  \[
    g \colon \barB(B,r_1) \to C[1],\quad \widehat f \colon \barB(A,r_2) \to \barB(B,r_1)
  \]
  are bounded. Then the composition $g \diamond f = g \circ \widehat f$ is also bounded.
\end{proof}

We conclude that analytic $A_\infty$-algebras again form a category $\AinfAlgan$, with morphism spaces
\[
  \Hom_\AinfAlg^\an(A,B) \colonequals \overline\fA(A,B) \cap \Hom_\AinfAlg(A,B),
\]
with an obvious forgetful functor $\AinfAlgan \to \AinfAlg$.

\subsection{Analytic homotopies}

To generalise the definition of homotopies to analytic $A_\infty$-algebras we introduce a norm on the path object.
The algebra $\pc^\infty([0,1])$ is bounded for the essential supremum-norm $\norm_{\infty}$ on $L^\infty([0,1])$ and likewise $\Omega^\bullet_{[0,1]}$ is bounded in the Sobolev norm
\[
  \nrm{x+y\d t} = \nrm{x}_\infty + \nrm{\tfrac{\partial x}{\partial t}}_\infty + \nrm{y}_\infty.
\]
For $A\in \AinfAlg^\an$ there is an induced tensor norm on $\Omega^\bullet_{[0,1]}\otimes A$ and one checks easily that this is an analytic $A_\infty$-algebra using \eqref{eq:difftensor} and \eqref{eq:highertensor}.
The following is then a natural generalisation of homotopy between analytic $A_\infty$-morphisms. 

\begin{defn}
  For $A,B\in\AinfAlgan$ a homotopy $H\in \Hom_\AinfAlg(A,\Omega^\bullet_{[0,1]}\otimes B)$ is \emph{analytic} if it lies in the subspace $\Hom_\AinfAlg^\an(A,\Omega^\bullet_{[0,1]}\otimes B)$.
\end{defn}

We will use an explicit sufficient condition for a homotopy to be analytic in terms of the coefficient functions.
Suppose $H = H_x + H_y \d t$ then for all $k\geq 1$ we have a bound
\[
  \nrm{H_k} \leq \nrm{H_{x,k}}_\infty  + \nrm{\tfrac{\partial}{\partial t} H_{x,k}}_\infty + \nrm{H_{y,k}}_\infty,
\]
where $\norm_\infty$ denotes the essential supremum of the operator norms over the interval,
and bounds
\[
  \nrm{H_{x,k}}_\infty \leq \nrm{H_k},\quad
  \nrm{\tfrac{\partial}{\partial t}H_{x,k}}_\infty \leq \nrm{H_k},\quad
  \nrm{H_{y,k}}_\infty \leq \nrm{H_k}.
\]
Hence $H$ is analytic if and only if there is a $C>0$ such that for each $k\geq 1$ the functions $H_{x,k}$, $\tfrac{\partial}{\partial t} H_{x,k}$, and $H_{y,k}$ are bounded almost everywhere by $C^k$.
We use this to prove the following.

\begin{lem}
  Analytic homotopy defines an equivalence relation $\sim_\an$ on $\Hom_\AinfAlg^\an(A,B)$.
\end{lem}
\begin{proof}
  Given $f\in \Hom_\AinfAlg^\an(A,B)$ the constant homotopy $H = 1\otimes f$ has norms
  $\nrm{(1\otimes f)_k} = \nrm{f_k}$ and is therefore analytic.
  Hence $f\sim_\an f$ always holds.
  Likewise, if $f\sim_\an g$ with analytic homotopy $H = H_x + H_y \d t$, then the coefficient functions of the homotopy $G^t = H_x^{-t} - H_y^{-t} \d t$ have norms
  \[
    \nrm{G_{x,k}}_\infty = \nrm{H_{x,k}}_\infty,\quad
    \nrm{\tfrac{\partial}{\partial t} G_{x,k}}_\infty
    = \nrm{\tfrac{\partial}{\partial t} H_{x,k}}_\infty,\quad
    \nrm{G_{y,k}}_\infty = \nrm{H_{y,k}}_\infty.
  \]
  Hence $G$ is again analytic and provides an analytic homotopy $g\sim_\an f$.
  Finally, suppose $f\sim_\an g$ and $g\sim_\an h$ via analytic homotopies $H$ and $G$, then we have
  \[
    \begin{gathered}
      \nrm{(G*H)_{x,k}}_\infty = \max\{\nrm{H_{x,k}}_\infty,\nrm{G_{x,k}}_\infty\},\quad
      \nrm{(G*H)_{x,k}}_\infty = 2 \max\{\nrm{\tfrac{\partial}{\partial t}H_{x,k}}_\infty,\nrm{\tfrac{\partial}{\partial t}G_{x,k}}_\infty\},\\
      \nrm{(G*H)_{x,k}}_\infty = 2 \max\{\nrm{H_{y,k}}_\infty,\nrm{G_{y,k}}_\infty\},\quad
    \end{gathered}
  \]
  Since the coefficient functions of $H$ and $G$ are essentially bounded by $C^k$ for some $C>0$, the coefficient functions of $G*H$ are essentially bounded by $(2C)^k$.
  Hence, $G*H$ is an analytic homotopy $f\sim_\an h$.
\end{proof}

The relation $\sim_\an$ is again compatible with $\diamond$, because the compositions used in \cref{lem:diamondhom} are analytic.
Hence we obtain an equivalence relation on morphisms in $\AinfAlgan$ which is compatible with composition.
In particular, we can speak of \emph{analytic homotopy inverse} maps and \emph{analytic homotopy equivalences} as before.

\subsection{Analytic minimal models}

We define an analogue of a minimal model in the analytic setting by requiring all the structures involved to be analytic.
\begin{defn}\label{minimalmodel}
    An \emph{analytic minimal model} for $A\in\AinfAlgan$ is a minimal $\hh(A) \in \AinfAlgan$ with quasi-isomorphisms $P \in \Hom_\AinfAlg^\an(A,\hh(A))$ and $I \in\Hom_\AinfAlg^\an(\hh(A),A)$ such that:
  \begin{equation}\label{eq:minmodordan}
      P\diamond I = \id_{\hh(A)}\quad\text{ and }\quad   I\diamond P \sim_\an \id_A.
  \end{equation}
\end{defn}

Contrary to the non-analytic case, \emph{not every analytic $A_\infty$-algebra has an analytic minimal model}.
Indeed, it is not even possible in general to find chain maps $I_1$ and $P_1$ providing a bounded splitting of its cohomology.
Nonetheless, the construction in \cite{Tu14} shows that the homotopy transfer formula can be used to construct analytic minimal models in geometric settings which we will see in \cref{sec:geometry}.

We conclude by generalising some results about (homotopy) inverses from \cref{sec:minmods}.
\begin{lem}\label{analinverse}
  Let $f \in \overline\fA^0(A,B)$ be an \emph{analytic} pre-morphism such that $f_1\colon A[1]\to B[1]$ admits a bounded inverse $g_1$. Then the inverse $g$ of \cref{inversemap} is again analytic.
\end{lem}
\begin{proof}
  By assumption there exists a common constant $C>1$ such that $\nrm{g_1} < C$ and $\nrm{f_n} < C^n$ for all $n\in\geq 1$.
  We now claim that for any tree $T \in \O(n,d)$ the norm of $g_T$ is bounded by $\nrm{g_T} < C^{2n+2d-1}$.
  This follows by induction: for the base case $T\in \O(1) = \O(1,0)$, as
  \[
    \nrm{g_T} = \nrm{g_1} < C = C^{2+0-1},
  \]
  while for $n>1$ any tree $T\in \O(n)$ determined by subtrees $T_i \in \O(n_i,d_i)$ such that $\sum n_i = n$ and $\sum d_i = d-1$, we have
  \[
    \nrm{g_T} \leq \nrm{g_1}\nrm{f_k}\nrm{g_{T_1}\otimes \cdots\otimes g_{T_k}} 
    < C \cdot C^k \cdot C^{2\sum n_i + 2\sum d_i - k} = C^{2n + 2d -1},
  \]
  assuming the claim holds for all $n_i< n$ and $d_i < d$.
  Now we note that a tree $T\in\O(n,d)$ has $n+d$ total vertices, so the cardinality of $\O(n,d)$ is bounded by the Catalan number $\cC_{n+d}$.
  Because internal nodes have valency $\geq 3$ it follows that any tree in $\O(n)$ has at most $d = n-1$ internal nodes, so that the total cardinality is bounded by the sum $\cC_1 + \cdots + \cC_{2n-1}$ of Catalan numbers.
  It is well-known that the function $n \mapsto \cC_1 + \cdots + \cC_{2n-1}$ can be bounded by $q^n$ for all $n$, for some sufficiently large constant $q\gg 0$.
  We therefore find the following bound on the norm of $g_n$:
  \[
      \nrm{g_n} \leq \sum_{d=0}^{n-1} \sum_{T\in \O(n,d)} \|g_T\|
      < \sum_{d=0}^{n-1} \sum_{T\in \O(n,d)} C^{2n+2d-1}
      \leq |\O(n)| \cdot C^{4n} < (qC^4)^n.
  \]
  Because the constant $qC^4$ is independent of $n$, it follows that $g$ is analytic.
\end{proof}
Note that if the $A_\infty$-algebras $A$ and $B$ are finite-dimensional, then the continuity of $g_1$ is automatic.
The above now implies that the morphism between minimal models induced by an analytic quasi-isomorphism is invertible in the analytic category.
\begin{lem}\label{quasitoiso}
  Suppose $A,B\in \AinfAlgan$ are compact and admit analytic minimal models.
  Then every analytic quasi-isomorphism induces an analytic isomorphism between the minimal models.
\end{lem}
\begin{proof}
  By assumption there are minimal models $\hh(A)$, $\hh(B)$ such that the maps $I_A,I_B,P_A,P_B$ are analytic, and hence the morphism $\hh(f) \colonequals P_B \diamond f \diamond I_A$ is again analytic.
  By assumption $\hh(A)$ and $\hh(B)$ are finite dimensional, so the inverse of $\hh(f)_1$ is bounded and therefore the inverse $\hh(f)^{-1}$ is analytic by \cref{analinverse}.
\end{proof}
\begin{cor}\label{uniqanalminmod}
  Analytic minimal models of compact analytic $A_\infty$-algebras are unique up to analytic isomorphism.
\end{cor}
If $f$ is an analytic quasi-isomorphism between compact analytic $A_\infty$-algebras admitting minimal models, then $g = I_A \diamond \hh(f)^{-1} \diamond P_B$ is an analytic homotopy inverse, since the homotopies \eqref{eq:homotinvgf} and \eqref{eq:homotinvfg} are again analytic.
Hence for such analytic $A_\infty$-algebras, analytic quasi-isomorphisms and analytic homotopy equivalences are the same.

Finally we show that the existence of minimal models only depends on the analytic homotopy equivalence class of an $A_\infty$-algebra.

\begin{lem}\label{existsminmod}
  Suppose $A,B\in\AinfAlgan$ analytic $A_\infty$-algebras which are analytically homotopy equivalent.
  Then $A$ admits an analytic minimal model if and only if $B$ does.
\end{lem}
\begin{proof}
  By assumption there are maps $f \in \Hom_\AinfAlg^\an(A,B)$ and $g\in\Hom_\AinfAlg^\an(B,A)$ which are homotopy inverse.
  Then the composition $g\diamond f \in \Hom_{\AinfAlg}^\an(A,A)$ is a homotopy equivalence, hence a quasi-isomorphism.
  If $A$ admits a minimal model, it follows from \cref{quasitoiso} that $g\diamond f$ admits an analytic homotopy-inverse $T = I_A \diamond \hh(g\diamond f)^{-1} \diamond P_A$.
  We see that
  \[
    (P_A\diamond T\diamond g) \diamond (f\diamond I_A)  =
    P_A \diamond I_A \diamond \hh(g\diamond f)^{-1} \diamond \hh(g\diamond f) = \id_{\hh(A)},
  \]
  so that $P_A\diamond T\diamond g$ is left-inverse to $f\diamond I_A$.
  To show it is homotopy right-inverse, first note that $g\diamond f \sim_\an \id_A$ implies that $T$ is homotopic to the identity via
  \[
    T = I_A \diamond \hh(g\diamond f)^{-1} \diamond P_A \sim_\an
    I_A \diamond \hh(g\diamond f)^{-1} \diamond P_A \diamond g\diamond f \diamond I_A \diamond P_A \sim_\an
    = I_A \diamond P_A \sim_\an \id_A.
  \]
  Since $f\diamond g\sim_\an \id_B$ also holds, the claim then follows from the chain of homotopies
  \[
    (f\diamond I_A) \diamond (P_A\diamond T\diamond g) \sim_\an f \diamond T \diamond g \sim_\an f\diamond g \sim_\an \id_B.
  \]
  Hence the maps $I_B = f\diamond I_A$ and $P_B = P_A\diamond T\diamond g$ make $\hh(A)$ into a minimal model for $B$.
\end{proof}

\section{Analytic \texorpdfstring{$A_\infty$-}{A-infinity }bimodules}\label{sec:bimodules}

In this section we define $A_\infty$-bimodules and morphisms over an analytic $A_\infty$-algebra by imposing an analogous condition on sequences of multilinear maps. 
We start by recalling the definition of such bimodules using a bar construction, for which we again follow Tradler \cite{Tra08}.

\subsection{\texorpdfstring{$A_\infty$-}{A-infinity }bimodules}

Given an $A_\infty$-algebra $A \in \AinfAlg$, $A$-bimodules and the morphisms between them can be defined via double sequences of morphisms in the graded vector space
\[
  \fM_A(M,N) \colonequals \prod_{i,j=0}^\infty \hom_\elle(A[1]^{\otimes i} \otimes M[1] \otimes  A[1]^{\otimes j}. N[1]).
\]
We will use the notation $\uprho = (\uprho_{i,j})_{i,j\in\N}$ for homogeneous elements of $\fM_A(M,N)$, and when applying such a map to elements we will underline the element in $M[1]$ for clarity.
As before, this space of sequences can be identified with the morphisms out of a bar construction
\[
  \barB_A M \colonequals \barB A \otimes M[1] \otimes \barB A = \bigoplus_{i,j=0}^\infty A[1]^{\otimes i} \otimes M[1]\otimes A[1]^{\otimes j},
\]
which is naturally a cofree cobimodule over the coalgebra $\barB A$.
In \cite{Tra08} it is shown that one can lift any map $\uprho \in \fM_A(M,N)$ to a cohomomorphism $\widehat \uprho\colon \barB_AM \to \barB_AN$, and to a coderivation $\widetilde \uprho \colon \barB_AM \to \barB_A M$ if $M=N$.
These maps are explicitly given by the following formulas:
\[
  \begin{aligned}
    \widehat \uprho &\colonequals \sum_{i,j\geq 0} \id^{\otimes i} \otimes \underline{\uprho} \otimes \id^{\otimes j}
    \\
    \widetilde \uprho &\colonequals
                        \sum_{i,j} \id^{\otimes i} \otimes  \underline\uprho \otimes \id^{\otimes j} + 
                        \sum_{i,j,k\geq 0} \id^{\otimes i} \otimes \upmu_A \otimes \id^{\otimes k} \otimes \underline\id_{M[1]} \otimes \id^{\otimes j} \\&\quad+
                        \sum_{i,j,k\geq 0}\id^{\otimes i} \otimes \underline \id_{M[1]}  \otimes \id^{\otimes k} \otimes \upmu_A \otimes \id^{\otimes j},
  \end{aligned}
\]
where the part of the maps landing in $M[1]$ is again underlined.
These lifts are used to define $A_\infty$-bimodule structures and their morphisms, as follows.

\begin{defn}
  An $A_\infty$-bimodule over an $A_\infty$-algebra $A$ is pair $(M,\upnu)$ of a bimodule $M \in \bimodl$ and a map $\upnu \in \fM_A^0(M, M)$ satisfying $\widetilde \upnu^2 = 0$, or equivalently $\upnu \circ \widetilde \upnu = 0$.
\end{defn}

The main examples of bimodules which we will consider are the following, proofs that these are bimodules can be found in \cite{Tra08}.

\begin{ex}
  The \emph{diagonal bimodule} $A_\Delta = (A,\upnu_\Delta)$ is defined by the map
  \[
    \upnu_{\Delta,i,j}(a,\underline b,c) = \upmu_{i+j+1}(a,b,c).
  \]
\end{ex}
\begin{ex}
  The \emph{dual bimodule} $A^\vee = (\Hom_\C(A,\C), \upnu_\Delta^\vee)$ is defined by the sequence of maps
  \[
    \upnu_{\Delta,i,j}^\vee(a, \underline b,c)(d) = (-1)^K b(\mu_{i+j+1}(c,d,a)),
  \]
  where $K= |a|(|b| + |c| + |d|) + |b|$ is a Koszul sign.
\end{ex}

Bimodules over an $A_\infty$-algebra $A$ form a DG category with the following morphism complexes.

\begin{defn}
  A \emph{pre-morphism} between $A_\infty$-bimodules $(M,\upnu_M)$ and $(N,\upnu_N)$ is a cochain in
  \[
    \hom_{A-A}^\bullet(M,N) \colonequals \left(\fM_A^\bullet(M, N),\delta\right),\quad \updelta(\uprho) = \upnu_N \circ \widehat\uprho - (-1)^{|\uprho|} \uprho \circ \widetilde\upnu_M.
  \]
  And the degree $0$ cocycles form the subspace of \emph{morphisms} $\Hom_{A-A}(M,N) \subset \hom_{A-A}^0(M,N)$.
\end{defn}

Given $\uprho \in \hom_{A-A}(M,N)$ and $\uptau \in \hom_{A-A}(L,M)$ the composition is defined as $\uprho \diamond \uptau = \uprho \circ \widehat \uptau$.
This composition then yields a well-defined DG category $\infbimod_{A-A}$ of $A_\infty$-bimodules.

Given a morphism $f \in \Hom_{\AinfAlg}(A,B)$ of $A_\infty$-algebras $A,B\in\AinfAlg$ there is a DG functor
\[
  f^\sharp \colon \infbimod_{B-B}\to \infbimod_{A-A},
\]
mapping a $B$-bimodule $(M,\upnu)$ to $(M,f^\sharp\upnu) = (M,\upnu \circ (\widehat f \otimes \underline\id_{M[1]} \otimes \widehat f))$ and mapping a $B$-bimodule pre-morphism $\uprho$ to the map $f^\sharp\uprho$ defined by the composition
\[
  \barB_A M = \barB A \otimes M[1] \otimes \barB A \xrightarrow{\widehat f \otimes \underline\id_{M[1]} \otimes \widehat f} \barB B \otimes M[1] \otimes \barB B \xrightarrow{\uprho} N[1].
\]
For a proof that this defines a DG functor $f^\sharp$ and that this is functorial in $f$, see \cite[§\nobreak\,\nobreak2.8]{Ganatra12}.

For any $f \in \Hom_{\AinfAlg}(A,B)$ there are $A_\infty$-bimodule morphisms $f_\Delta \in \Hom_{A-A}(A_\Delta,f^\sharp B_\Delta)$ and $f_\Delta^\vee\in \Hom_{A-A}(f^\sharp B_\Delta^\vee, A_\Delta^\vee)$ defined in components by
\[
  \begin{aligned}
    (f_\Delta)_{i,j}(a,b,c) &= f_{i+1+j}(a,b,c),\\
    (f_\Delta^\vee)_{i,j}(a,\underline{b},c)(d) &= (-1)^K b (f_{i+1+j} (c, d, a)),
  \end{aligned}
\]
where $a \in A^{\otimes i}$, $c\in A^{\otimes j}$, and $K = |a|(|b|+|c|+|d|)$ is a Koszul sign.
Given any pre-morphism $\uprho \in \hom_{B-B}(B_\Delta,B_\Delta^\vee)$ there is a pre-morphism $f^*\uprho \in \hom_{A-A}(A_\Delta,A_\Delta^\vee)$ given by
\[
  f^*\uprho \colon \barB_A A_\Delta \xrightarrow{\widehat f_\Delta} \barB_A f^\sharp B_\Delta \xrightarrow{\widehat{f^\sharp\uprho}} \barB_A f^\sharp B_{\Delta}^\vee \xrightarrow{{\widehat f}^\vee_\Delta} A_\Delta^\vee.
\]
This defines a chain map $f^*\colon \hom_{B-B}(B_\Delta,B_\Delta^\vee) \to \hom_{A-A}(A_\Delta,A_\Delta^\vee)$ for each $f$, since $f^\sharp$ is a DG functor and $f_\Delta$ and $f_\Delta^\vee$ commute with the differential.
One can check that this is also functorial.

\subsection{Hochschild cohomology}

Given an $A_\infty$-algebra $A$, one can define the Hochschild cohomology with values in any $A_\infty$-bimodule $M$.
The Hochschild complex can again be modeled using the bar construction: following Tradler \cite[Lemma 2.3]{Tra08} each $\upxi \in \fM(A,M)$ defines a map
\[
  \overline\upxi \colonequals \sum_{i,j,n\geq 0} \id^{\otimes i}_{A[1]} \otimes \underline{\xi_n} \otimes \id^{\otimes j}_{A[1]},
\]
which is a coderivation $\overline\upxi\colon \barB A \to \barB_AM$ with values in the comodule $\barB_AM$.
As before, the underline signifies the component mapping to the factor $M$.
Using this lift, the Hochschild cochain complex can be defined as follows.

\begin{defn}
  Let $(A,\upmu)$ be an $A_\infty$-algebra and $(M,\upnu)$ an $A$-bimodule.
  Then the \emph{$M$-valued Hochschild cochain complex} is the complex $\barC^\bullet(A,M) \colonequals \left(\fM(A,M)[-1],\ \bb\right)$ with differential
  \[
    \bb(\xi) \colonequals \upnu \circ \overline\xi - (-1)^{|\xi|} \xi \circ \widetilde\upmu.
  \]
  The cohomology of the complex is denoted $\HH^\bullet(A,M)$, and if $M = A_\Delta$ the complex and its cohomology are denoted as $\barC^\bullet(A) \colonequals \barC^\bullet(A,A_\Delta)$ and $\HH^\bullet(A) \colonequals \HH^\bullet(A,A_\Delta)$ respectively.
\end{defn}

The Hochschild complex is functorial in each factor: any morphisms $f\in\Hom_{\AinfAlg}(A,B)$ or $\uprho\in\Hom_{A-A}(M,N)$ induce respective chain maps
\[
  \begin{aligned}
    \barC^\bullet(B,M) &\to \barC^\bullet(A,f^\sharp M)\quad &\upxi \mapsto \upxi \circ \widehat f,\\
    \barC^\bullet(A,M) &\to \barC^\bullet(A,N)\quad &\upxi \mapsto \uprho \circ \overline\upxi.
  \end{aligned}
\]
In particular, for $M= A_\Delta^\vee$ and $\uprho = f_\Delta^\vee$ the combination of these constructions yields a chain map $f^* \colon \barC^\bullet(B,B_\Delta^\vee) \to \barC^\bullet(A,A_\Delta^\vee)$, sending $\xi \mapsto f_\Delta^\vee \circ \overline{\xi \circ \widehat f}$.
Explicitly, this map is given by
\[
  f^*\xi(a_1,\ldots,a_n)(a_0) = \sum_{\substack{i,j\geq 0\\i+j \leq n}} \pm \xi(\widehat f(a_{i+1},\ldots,a_{n-j}))(f(a_{n-j+1},\ldots,a_n,a_0,a_1,\ldots,a_i)),
\]
where the signs are the natural Koszul signs.
This map is natural with respect to $f$ and induces a map $[f^*] \colon \HH^\bullet(B,B_\Delta^\vee) \to \HH^\bullet(A,A_\Delta^\vee)$, making the Hochschild cohomology with values in the dual bimodule functorial.

A variation of Hochschild cohomology is \emph{negative cyclic cohomology}, which can be realised using Connes' complex.
An element $\upxi = (\xi_n)_{n\in\N} \in \barC^\bullet(A,A^\vee_\Delta)$ is called \emph{cyclic} if for each $n\geq 1$ and all $a_0,\ldots,a_n \in A[1]$ the following relation holds:
\[
  \xi_n(a_1,\ldots,a_n)(a_0) = (-1)^{|a_0|(|a_1|+\cdots+|a_n|)} \xi_n(a_0,\ldots,a_{n-1})(a_n)
\]
One checks that the Hochschild differential preserves cyclic cochains, yielding a subcomplex
\[
  \barC^\bullet_\lambda(A) \colonequals \left(\left\{\ \upxi \in \barC^\bullet(A,A^\vee_\Delta) \mid \upxi \text{ is cyclic }\right\}, \bb \right),
\]
called Connes' complex.
The cohomology of this complex is denoted $\HC_\lambda^\bullet(A)$.
We remark that the pullback along any morphism $f\in\Hom_{\AinfAlg}(A,B)$ preserves cyclic cocycles, and therefore restricts to a map $f^*\colon\barC^\bullet_\lambda(B) \to \barC^\bullet_\lambda(A)$ and an induced map $[f^*] \colon \HC_\lambda^\bullet(B) \to \HC_\lambda^\bullet(A)$.

It is well-known that the Hochschild cohomology and negative cyclic cohomology are invariant under homotopy equivalences.
Below we include a proof using the definition of homotopy in~\cref{ssec:homotopies}, which will also apply in the analytic setting.

\begin{prop}\label{hochschildhomot}
  Let $f,g\in\Hom_\AinfAlg(A,B)$ with $f\sim g$, then $f$ and $g$ induce the same morphisms
  \[
    [f^*] = [g^*] \colon \HH^\bullet(B,B_\Delta^\vee) \to \HH^\bullet(A,A_\Delta^\vee),\quad
    [f^*] = [g^*] \colon \HC^\bullet_\lambda(B) \to \HC^\bullet_\lambda(A).
  \]
\end{prop}
\begin{proof}
  If $f\sim g$ then $f = \ev_0 \circ H$ and $g=\ev_1 \circ H$ for some homotopy $H\in \Hom_\AinfAlg(A,\Omega^\bullet_{[0,1]}\otimes B)$ and the induced maps are therefore given by
  \[
    [f^*] = [H^*][\ev_0^*],\quad [g^*] = [H^*][\ev_1^*].
  \]
  Therefore, it is sufficient to prove that the maps $\ev_0,\ev_1 \in \Hom_\AinfAlg(\Omega^\bullet_{[0,1]} \otimes B, B)$ induce the same map $[\ev_0^*] = [\ev_1]$ on Hochschild/negative cyclic cohomology, which we do below.
  
  Let $\upxi \in Z^k\barC(B,B^\vee_\Delta)$ be any fixed cocycle.
  Then the pullback along $\ev_t$ for any $t\in [0,1]$ is given by a cocycle $\ev_t^*\upxi \in Z^k\barC(\Omega^\bullet_{[0,1]} \otimes B,(\Omega^\bullet_{[0,1]} \otimes B)^\vee_\Delta)$,
  given on pure tensors of the form $\omega_i = (x_i+y_i\d t) \otimes b_i$ by
  \[
    (\ev_t^*\upxi)_n (\omega_1,\ldots,\omega_n)(\omega_{n+1})
    = x_1(t)\cdots x_{n+1}(t) \cdot \xi_n(b_1,\ldots,b_n)(b_{n+1}).
  \]
  We claim that the difference $\ev_1^*\upxi - \ev_0^*\upxi$ is the $\bb$-image $\upzeta\in\barC^{k-1}(\Omega^\bullet_{[0,1]} \otimes B, (\Omega^\bullet_{[0,1]} \otimes B)^\vee_\Delta)$ defined on pure tensors $\omega_i = (x_i+y_i\d t) \otimes b_i$ as
  \begin{equation}\label{eq:zetacobound}
    \begin{aligned}
      \zeta_n(\omega_1,\ldots,\omega_n)(\omega_{n+1})
      &= \sum_{j=1}^n(-1)^{|\upxi|+|b_1|+\cdots+|b_{j-1}|} \int_0^1 x_1\cdots y_j \cdots x_{n+1} \d t \cdot \xi_n(b_1,\ldots,b_n)(b_{n+1})\\
      &\quad+ (-1)^{|\upxi|+1+|b_1|+\cdots+|b_n|} \int_0^1 x_1\cdots x_n y_{n+1} \d t \cdot \xi_n(b_1,\ldots,b_n)(b_{n+1}).
    \end{aligned}
  \end{equation}
  For brevity we decompose the $A_\infty$-structure defined in~\eqref{eq:difftensor} and~\eqref{eq:highertensor} as $\upmu^\otimes = \d + \upmu$, using the abbreviation $\d = \d \otimes \id$ and $\upmu = \sum_{k\geq 1} (-\cdots-) \otimes \upmu_k$.
  Then we have the following:
  \begin{align}
      \bb(\upzeta)
      &=
        \d_\Delta^\vee \circ \upzeta - (-1)^{|\upzeta|} \cdot \upzeta \circ \widetilde{\d} \label{bbtopline}\\
      &\quad+\upmu^\vee_\Delta \circ \overline\upzeta
        -(-1)^{|\upzeta|} \upzeta \circ \widetilde{\upmu_k}.\label{bbbotline}
  \end{align}  
  By inspection, the value of~\eqref{bbtopline} on variables $\omega_i = (x_i+y_i\d t) \otimes b_i$ is 
  \[
    \begin{aligned}
      &(\d^\vee \circ \upzeta - (-1)^{|\upzeta|} \cdot\upzeta \circ \widetilde{\d})(\omega_1,\ldots,\omega_n)(\omega_{n+1})\\
      \\&\quad=
        (-1)^{|\upzeta|+|b_1|+\cdots + |b_n|} \upzeta(x_1\otimes b_1,\ldots,x_n\otimes b_n)(\d x_{n+1}\otimes b_{n+1})
      \\&\quad\quad-
       \sum_{j=1}^n (-1)^{|\upzeta|+|b_1|+\cdots+|b_{j-1}|} \cdot\upzeta_n(x_1\otimes b_1,\ldots, \d x_j\otimes b_j , \ldots, x_n\otimes b_n)(x_{n+1}\otimes b_{n+1})
      \\&\quad=
      \sum_{j=1}^{n+1}
      \int_0^1 \left(x_0\cdots \tfrac{\partial x_j}{\partial t} \cdots x_n\right) \d t \cdot \upxi_n(b_1,\ldots,b_n)(b_{n+1})
      \\&\quad=
      (x_0(1)\cdots \cdots x_n(1) -
      x_0(0)\cdots \cdots x_n(0))\cdot \upxi_n(b_1,\ldots,b_n)(b_0)
      \\&\quad= (\ev_1^*\upxi)_n (\omega_1,\ldots,\omega_n)(\omega_0) - (\ev_0^*\upxi)_n (\omega_1,\ldots,\omega_n)(\omega_0).
    \end{aligned}
  \]
  Conversely, the motivated reader may check that the value of~\eqref{bbbotline} on pure tensors $\omega_i$ is equal to
  \[
    \begin{aligned}
      &(\upmu_\Delta^\vee\circ\overline\upzeta - (-1)^{|\upzeta|} \upzeta \circ \widetilde \upmu)(\omega_1,\ldots,\omega_n)(\omega_{n+1})    
      \\&=\quad\sum_{j=1}^{n} (-1)^{|\upxi|+|b_1|+\cdots+|b_{j-1}|} \cdot \int_0^1 x_1\cdots y_j \cdots x_n \d t \cdot \bb(\upxi)(b_1,\ldots,b_n)(b_{n+1}),
    \end{aligned}
  \]
  which always vanishes because $\upxi$ is assumed to be a cocycle.
  We conclude that
  \[
    [\ev_1^*\upxi] = [\ev_0^*\upxi + \d \upzeta] = [\ev_0^*\upxi].
  \]
  To see that the same is true for the cyclic classes, it suffices to show that $\upzeta$ is cyclic whenever $\upxi$ is cyclic.
  If $\upxi$ is cyclic, then given homogeneous pure tensors $\omega_i$ with $\omega_j = y_j \d t \otimes b_j$ for some fixed $j\leq n$ and $\omega_i = x_i \otimes b_i$ for all $i\neq j$ we have
  \[
    \begin{aligned}
      &\zeta_n(\omega_{n+1},\omega_1,\ldots,\omega_{n-1})(\omega_n)
      \\&\quad=
      (-1)^{|\upxi|+|b_{n+1}| + |b_1|+\cdots+|b_{j-1}|} \int_0^1 x_1\cdots y_j \cdots x_{n+1} \d t \cdot \xi_n(b_{n+1},b_1,\ldots,b_{n-1})(b_n)
      \\&\quad=
      (-1)^{|b_{n+1}| + |b_1|+\cdots+|b_{j-1}| + |b_{n+1}|(|b_1|+\cdots+|b_n|)} \int_0^1 x_1\cdots y_j \cdots x_{n+1} \d t \cdot \xi_n(b_1,\ldots,b_n)(b_{n+1})
      \\&\quad=
      (-1)^{|b_{n+1}|(|b_1|+\cdots+(|b_j|+1)+\cdots+|b_n|)}\upzeta(\omega_1,\ldots,\omega_n)(\omega_{n+1})
      \\&\quad=
      (-1)^{|\omega_{n+1}|(|\omega_1|+\cdots+|\omega_n|)}\upzeta(\omega_1,\ldots,\omega_n)(\omega_{n+1})
    \end{aligned}
  \]
  A similar computation shows that the cyclic identity holds when $\omega_{n+1} = y_{n+1}\d t \otimes b_{n+1}$.
  This shows that $\ev_0$ and $\ev_1$ induce the same classes on Hochschild and negative cyclic cohomology, as claimed.
\end{proof}

\begin{cor}\label{hochschildminim}
  For an $A_\infty$-algebra $A$ with minimal model $\hh(A)$ there are isomorphisms
  \[
    \HH(A,A^\vee_\Delta) \cong \HH(\hh(A),\hh(A)^\vee_\Delta),\quad
    \HC_\lambda(A,A^\vee_\Delta) \cong \HC_\lambda(\hh(A),\hh(A)^\vee_\Delta).
  \]
\end{cor}
\begin{proof}
  The maps induced by the morphisms $I$ and $P$ satisfy $I^*P^* = (P\diamond I)^* = \id$ and the existence of a homotopy $I\diamond P \sim \id_A$ implies $[P^*][I^*] = [(I\diamond P)^*] = [(\id_A)^*] = \id$.
\end{proof}

\subsection{Analytic \texorpdfstring{$A_\infty$-}{A-infinity }bimodules}\label{sec:analbimodules}

Given $A\in\AinfAlgan$ and $M\in\grNMod\elle$ we consider a family of norms as in~\cref{sec:norms}: for each $r>0$ we define
\[
  \barB_A(M,r) \colonequals (\barB_A M, \norm_r),\quad \nrm{m}_r = \sup_{i,j} \nrm{m_{i,j}}r^{-i-j-1},
\]
where $m = \sum_{i,j} m_{i,j} \in \barB_AM$ is a decomposition into elements $m_{i,j} \in A[1]^{\otimes i} \otimes M[1] \otimes A[1]^{\otimes j}$.
Boundedness of linear maps out of these normed bar constructions can again be characterised in various ways, analogously to~\cref{analboundedness}.

\begin{lem}\label{bianalbound}
  For $\uprho\in\fM_A(M,N)$ the following are equivalent:
  \begin{enumerate}
  \item there exists $C>0$ such that $\nrm{\uprho_{i,j}} \leq C^{i+j+1}$ for all $i,j\in\N$
  \item there exists $r>0$ such that $\uprho\colon \barB_A(M,r) \to N[1]$ is bounded,
  \item for every $r'>0$ there exists $r>0$ such that $\widehat \uprho \colon \barB_A(M,r) \to \barB_A(N,r')$ is  bounded
  \end{enumerate}
  if moreover $M=N$, then this is also equivalent to:
  \begin{enumerate}[resume]
  \item for every $r'>0$ there exists $r>0$ such that $\widetilde \uprho \colon \barB_A(M,r) \to \barB_A(M,r')$ is bounded.
  \end{enumerate}
\end{lem}
\begin{proof}
  (1$\implies$2) Suppose there exists $C>0$ such that $\nrm{\uprho_{i,j}} \leq C^{i+j+1}$ for all $i,j\in\N$ and fix $0<r< C^{-1}$.
  Then an element $m = \sum_{i,j} m_{i,j} \in \barB_A(M,r)$ of norm $\nrm{m}_r = 1$ satisfies $\nrm{m_{i,j}} \leq r^{i+j+1}$ for all $i,j\in\N$ and therefore
  \[
    \nrm{\uprho(m)} \leq \sum_{n\in\N} \sum_{i+j=n} \nrm{\uprho_{i,j}(m_{i,j})} \leq \sum_{n\in\N} \sum_{i+j=n} C^{i+j+1} r^{i+j+1} = \sum_{n\in \N} n (Cr)^{n+1}.
  \]
  Because the power series $\sum_n n z^{n+1}$ has radius of convergence $1$, it converges at $z = Cr < 1$.
  In particular, the norm $\uprho$ is bounded.
  
  (2$\implies$1) Let $K>1$ be a constant bounding the norm of $\uprho\colon \barB_A(M,r) \to N[1]$ and pick $C \geq \frac{K}{r}$, then for all $i,j\in\N$ and $m_{i,j}\in A[1]^{\otimes i} \otimes M[1] \otimes A[1]^{\otimes j}$ of norm $\nrm{m_{i,j}} = 1$ the map $\uprho_{i,j}$ satisfies
  \[
    \nrm{\uprho_{i,j}(m_{i,j})} = \nrm{\uprho(m_{i,j})} \leq K \nrm{m_{i,j}} r^{-i-j-1} \leq \left(\frac{K}{r}\right)^{i+j+1} \leq C^{i+j+1},
  \]
  which shows that in particular that $\nrm{\uprho_{i,j}} \leq C^{i+j+1}$.

  (1$\implies$3) Suppose there exists $C>1$ such that $\nrm{\uprho_{i,j}} \leq C^{i+j+1}$ for all $i,j\in\N$, and fix $r'>0$.
  Then we choose a constant $q>0$ such that $ij < q^{i+j+1}$ for all $i,j\in\N$, a constant $K=\min\{r',1\}$, and a radius $r<\frac{K}{qC}$.
  Then for any element $m = \sum m_{i,j} \in \barB_A M$ of norm $\nrm{m}_r = 1$ one has $\nrm{m_{i,j}} \leq r^{i+j+1}$ and hence
  \[
    \begin{aligned}
      \nrm{\widehat\uprho(m)}_{r'}
      &\leq \sum_{i,j\geq 0} \sum_{\substack{0\leq i\leq n\\0\leq j\leq m}} \nrm{(\id^{\otimes i} \otimes \uprho_{n-i,m-j} \otimes \id^{\otimes j})(m_{i,j})} (r')^{-i-j-1}
      \\&\leq \sum_{i,j\geq 0} \sum_{\substack{0\leq i\leq n\\0\leq j\leq m}} C^{i+j-i-j+1} K^{-i-j-1} r^{i+j+1}
      \\&\leq \sum_{i,j\geq 0} nm \left(\frac{rC}{K}\right)^{i+j+1}
      \\&\leq \sum_{N\geq 1} (N-1) \left(\frac{rq C}{K}\right)^N
    \end{aligned}
  \]  
  Then the norm is bounded by the value of the power series $\sum_N (N-1) z^N$ at $z=\frac{rqC}{K} < 1$.

  (1$\implies$4) This follows in an analogous way to previous implication. We leave the details up to the reader.

  (3,4 $\implies$ 2) The projection $\uppi \colon \barB_A(N,r') \to N[1]$ is clearly bounded for any $r'>0$, so boundedness of $\widehat \uprho\colon \barB_A(M,r) \to \barB_A(N,r')$ implies that $\uprho = \uppi\circ \widehat \uprho$ is bounded.
  Likewise, $\uprho = \uppi \circ \widetilde \uprho$ is bounded when $\widetilde \uprho\colon \barB_A(M,r) \to \barB_A(M,r')$ is bounded.
\end{proof}

Hence, as in the algebra case we find that the union over the spaces $\hom_\elle^\cont(\barB_A(M,r),N[1])$ is given by a graded vector space of double sequences
\[
\fA_A(V,W) \colonequals \left\{\; f \in \fM_A(V,W) \;\middle|\; \exists C>0 \text{ such that }  \|f_{i,j}\| \leq C^{i+j+1}\text{ for all } i,j\in\N \;\right\},
\]
whose elements we again call \emph{analytic}.
In particular, we obtain a definition of analytic $A_\infty$-bimodules over an analytic $A_\infty$-algebra.

\newcommand{\nn}{\mathbf{n}}
\begin{defn}
  Let $(A,\upmu)$ be an analytic $A_\infty$-algebra, then an \emph{analytic} $A$-bimodule is a pair $(M,\upnu)$ of $M\in \grNMod(\elle)$ and $\upnu \in \fA_A^1(M,M)$ satisfying $\widetilde\upnu^2 = 0$.
\end{defn}

\begin{ex}
  The diagonal bimodule $A_\Delta$ is analytic for any analytic $A_\infty$ algebra when endowed with the induced norm, as $\nrm{\upnu_{\Delta,i,j}} = \nrm{\upmu_{i+j+1}}$ satisfies the appropriate bound.
\end{ex}
\begin{ex}
  The dual bimodule $A^\vee$ is not analytic in general, as the dual space is not endowed with a natural norm.
  However, one can consider the \emph{continuous dual bimodule}
  \[
    A_\Delta' \colonequals \left(\hom^\cont_\C(A,\C), \upnu'_\Delta\right),
  \]
  where the components of $\upnu'_\Delta$ are defined by the same formula as the components of $\upnu_\Delta^\vee$.
  It is again straightforward to check that $\nrm{\upnu^\vee_{\Updelta,i,j}} = \nrm{\upmu_{i+j+1}}$ is bounded by a geometric series.
  If $A$ is finite dimensional, then $A'_\Delta$ is equal to $A^\vee_\Delta$ in $\Mod_{A-A}^\infty$, but in general it is only a submodule.
\end{ex}

As before we call a (pre-)morphism $\uprho \in \hom_{A-A}^{\bullet}(M,N)$ between $A$-bimodules $(M,\upnu_M)$ and $(N,\upnu_N)$ analytic if it lies in the subspace $\fA_A(M,N) \subset \fM_A(M,N)$.
It follows from~\cref{bianalbound} that
\[
  \updelta = \upnu_N \circ \widehat{(-)} - (-1)^{|\uprho|} (-)\circ \widetilde\upnu_M
\]
is a bounded operator on this space, and hence gives rise to a complex of analytic pre-morphisms
\[
  \hom_{A-A}^{\bullet,\an}(M,N) \colonequals \left(\fA_A(M,N), \delta\right),
\]
which is a subcomplex of $\hom_{A-A}^\bullet(M,N)$.
As before we denote the subspace of analytic morphisms by $\Hom_{A-A}^\an(M,N) \subset \Hom_{A-A}(M,N)$.
It follows from~\cref{bianalbound} that the composition of analytic bimodule pre-morphisms is again analytic, and therefore defines morphism complexes for a DG category $\NMod^{\infty,\an}_{A-A}$ of analytic bimodules over $A$.

Given an analytic $A_\infty$-morphism $f$ one can again consider the functor $f^*$, and the pullback map $f^\sharp$ on the Hom-spaces between the diagonal bimodule and its continuous dual.

\begin{lem}\label{analpullback}
  Let $f\in\Hom^\an_{\AinfAlg}(A,B)$ be an analytic morphism, then there is a well-defined functor
  \[
    f^\sharp \colon \NMod^{\infty,\an}_{B-B} \to \NMod^{\infty,\an}_{A-A},
  \]
  which induces a pullback $f^*\colon \hom_{B-B}^\an(B_\Delta,B_\Delta') \to \hom_{A-A}^\an(A_\Delta,A_\Delta')$.
\end{lem}
\begin{proof}
  Given $f \in\Hom^\an_{\AinfAlg}(A,B)$, for any $N,M\in\grNMod \elle$ and $\uprho\in\fA_B(M,N)$, there exists a common constant $C>0$ such that $\nrm{f_n} < C^n$ for all $n\in\N$ and $\nrm{\uprho_{i,j}} < C^{i+j+1}$ for all $i,j\in\N$.
  Then the construction $f^\sharp\uprho$ has components satisfying
  \[
    \begin{aligned}
      \nrm{(f^\sharp\uprho)_{i,j}}
      &\leq \sum_{\substack{i_1+\cdots+i_n = i\\ j_1+\cdots+j_n = j}} \nrm{\uprho_{n,m}} \nrm{f_{i_1}} \cdots \nrm{f_{i_n}} \nrm{f_{j_1}} \cdots\nrm{f_{j_m}}
      \\&< \sum_{\substack{i_1+\cdots+i_n = i\\ j_1+\cdots+j_n = j}} C^{n+m+i+j+1}
      \\&\leq (\text{\# of partitions of $i$}) \cdot (\text{\# of partitions of $j$})) \cdot C^{2(i+j)+1}.
    \end{aligned}
  \]
  Picking some constant $q>1$ such that $q^n$ bounds the number of partitions of $n$, it follows that $\nrm{(f^\sharp\uprho)_{i,j}} < q^{i+j}C^{2(i+j)+1} \leq (qC^2)^{i+j+1}$, which shows that $f^\sharp\uprho \in \fA_A(M,N)$.
  If follows that $f^\sharp$ maps every analytic bimodule structure $\upnu$ to an analytic bimodule structure $f^\sharp\upnu$, and an analytic (pre-)morphism $\uprho \in \hom^\an_{A-A}(M,N)$ between analytic bimodules to an analytic bimodule map $f^\sharp\uprho$.
  Hence $f^\sharp$ is a well-defined functor between categories of analytic bimodules.

  It follows that $f^\sharp B_\Delta$ is an analytic bimodule, and it follows directly from the definition that the map $f_\Delta\colon A \to f^\sharp B_\Delta$ is analytic.
  The bimodule $f^\sharp B_\Delta'$ is likewise analytic, and it follows easily that the formula for $f^\vee_\Delta$ defines an analytic map $f_\Delta' \colon f^\sharp B_\Delta' \to A_\Delta'$.  
  Hence, if $\uprho \in \hom_{B-B}^\an(B_\Delta,B_\Delta')$ is an analytic bimodule map, the composition $f^*\uprho = f_\Delta' \diamond f^\sharp \uprho \diamond f_\Delta$ is again analytic.
\end{proof}

\subsection{Analytic Hochschild cohomology}\label{analhoch}

Given an analytic $A_\infty$-algebra $(A,\upmu)$ and an analytic $A$-bimodule $(M,\upnu) \in \NMod^{\infty,\an}_{A-A}$, we call a Hochschild cochain $\upxi \in \barC^\bullet(A,M) = (\fM(A,M)[-1],\bb)$ analytic if it lies in the subspace $\fA(A,M)[-1] \subset \fM(A,M)[-1]$.
This analytic condition can again be characterised in multiple ways.

\begin{lem}\label{analhochbound}
  Let $\upxi \in \barC^\bullet(A,M)$, then the following are equivalent:
  \begin{enumerate}
  \item $\upxi$ is analytic
  \item for every $r'>0$ there exists $r>0$ such that $\overline\upxi\colon \barB(A,r) \to \barB_A(M,r')$ is bounded
  \end{enumerate}
\end{lem}
\begin{proof}
  The implication (2) $\implies$ (1) is obvious, hence we only check the other implication.
  Given $\upxi\in\barC^{\bullet,\an}(A,M)$ there exists $C>1$ such that $\nrm{\xi_n} < C^n$ for all $n\geq 1$.
  Fix $r'>0$, and let $r < \min\{r',1/C\}$.
  Then for every element $\sum_{n\in\N} a_n \in \barB(A,r)$ of norm $\nrm{a}_r \leq 1$ 
  \[
    \begin{aligned}
      \nrm{\overline\upxi(a)}_{r'}
      &\leq \sum_{i,j,k\geq 0} \sum_{k\geq 0} \nrm{(\id^i\otimes\underline{\xi_k}\otimes \id^{\otimes j})(a_{i+k+j})} (r')^{-i-j-1}
      \\&\leq \sum_{i,j\geq 0} \nrm{\upxi_0(1)} \nrm{a_{i+j}} (r')^{-i-j-1}
      + \sum_{i,j\geq 0}\sum_{k\geq 1} C^{i+k+j} \nrm{a_{i+k+j}} \cdot (r')^{-i-j-1}
      \\&\leq \nrm{\upxi_0(1)}(r')^{-1} \sum_{n\geq 0} (n+1) (r/r')^n 
      + (r')^{-1}\sum_{i,j\geq 0}\sum_{k\geq 1} (Cr)^{k} \cdot (r/r')^{i+j}
      \\&\leq
      \frac{\nrm{\upxi_0(1)}}{r'(1-(r/r'))^2} 
      + \frac{1}{r'(1-Cr)(1-(r/r'))^2}
      < \infty.
    \end{aligned}
  \]
  It follows that the operator $\overline \upxi$ is bounded.
\end{proof}

As a corollary, it follows that the differential maps an analytic cochain $\upxi \in \barC^{\bullet,\an}(A,M)$ to a cochain
\[
  \bb(\upxi) = \upnu \circ \overline\upxi - (-1)^{|\upxi|} \upxi \circ \widetilde\upmu
\]
which is bounded as a map $\barB(A,r) \to M$ for some $r$, hence analytic by~\cref{analboundedness}.
It follows that the analytic cochains form a well-defined subcomplex which we denote by
\[
  \barC^{\bullet,\an}(A,M) \colonequals (\fA(A,M)[-1],\bb).
\]
For $M=A_\Delta$ we again abbreviate $\barC^{\bullet,\an}(A) \colonequals \barC^{\bullet,\an}(A,A_\Delta)$.
Similarly, we define the analytic version of Connes' complex as the intersection
\[
  \barC^{\bullet,\an}_\lambda(A) \colonequals \barC^\bullet_\lambda(A) \cap \barC^{\bullet,\an}(A,A_\Delta'),
\]
which will play the role of the negative cyclic cohomology in the analytic setting.

The analytic Hochschild complex is again functorial in each argument with respect to analytic maps: given $f\in\Hom_{\AinfAlg}^\an(A,B)$ it follows directly from~\cref{analboundedness} that there is a well-defined map
\[
  -\circ \widehat f \colon \barC^{\bullet,\an}(B,M) \to \barC^{\bullet,\an}(A,f^\sharp M),
\]
and likewise for any $\uprho\in\Hom_{A-A}^\an(M,N)$ it follows from~\cref{analhochbound} that
\[
  \uprho \circ \overline{-} \colon \barC^{\bullet,\an}(A,M) \to \barC^{\bullet,\an}(A,N),
\]
is well-defined.
In particular, composition with $\widehat f$ and $f'_\Delta$ induces maps $f^*\colon \barC^{\bullet,\an}(B,B_\Delta') \to \barC^{\bullet,\an}(A,A_\Delta')$ and $f^*\colon \barC^{\bullet,\an}_\lambda(B) \to\barC^{\bullet,\an}_\lambda(A)$ for any
analytic $A_\infty$-morphism $f\in\Hom_{\AinfAlg}^\an(A,B)$.
As before, we have a compatibility with homotopies.

\begin{lem}\label{analhochschildhomot}
  Let $f,g\in\Hom_\AinfAlg^\an(A,B)$ with $f\sim_\an g$, then $f$ and $g$ induce the same morphism on analytic Hochschild cohomology and analytic cyclic cohomology.
\end{lem}
\begin{proof}
  If $f\sim_\an g$ then $f = \ev_0 \circ H$ and $g=\ev_1\circ H$ for an analytic map $H\in \Hom_\AinfAlg^\an(A,\Omega^\bullet_{[0,1]}\otimes B)$.
  Hence, as in~\cref{hochschildhomot} it suffices to show that the analytic maps $\ev_0$ and $\ev_1$ induce the same map on the analytic versions of Hochschild/negative cyclic cohomology.
  For this it suffices to show that if $\upxi \in Z^k\barC^{\bullet,\an}(B,B'_\Delta)$ is an analytic cocycle then the cochain~\eqref{eq:zetacobound} is analytic.
  However, this is straightforward: for any fixed $n$ and elements $\omega_i = \sum_{k_i} (x_i^{k_i} + y_i^{k_i}\d t) \otimes b^{k_i}_i$  we have
  \[
    \begin{aligned}
      \nrm{\zeta_n(\omega_1,\ldots,\omega_n)(\omega_{n+1})}
      &\leq
        \sum_{k_1,\ldots,k_{n+1}}\sum_{j=1}^n \int_0^1 x_1^{k_1}\cdots y_j^{k_j}\cdots x_{n+1}^{k_{n+1}}\d t \cdot \nrm{\xi_n(b_1,\ldots,b_n)(b_{n+1})}
      \\&\leq
      \nrm{\xi_n} \cdot \sum_{k_1,\ldots,k_{n+1}} \sum_{j=1}^n \nrm{x_1^{k_1}}_\infty\cdots \nrm{y_j^{k_j}}_\infty\cdots \nrm{x_{n+1}^{k_{n+1}}}_\infty \cdot \nrm{b_1}\cdots \nrm{b_{n+1}}.
      \\&\leq \nrm{\xi_n} \cdot \sum_{k_1} \nrm{(x_1^{k_1} + y_1^{k_1}\d t) \otimes b^{k_1}_1} \cdots \sum_{k_{n+1}}  \nrm{(x_{n+1}^{k_{n+1}} + y_{n+1}^{k_{n+1}}\d t) \otimes b^{k_{n+1}}_{n+1}}.
    \end{aligned}
  \]
  Taking the infimum over all such decompositions of $\omega_i$ into pure tensors, we see that $\nrm{\zeta_n} \leq \nrm{\xi_n}$, and it follows that $\upzeta$ is again an analytic cochain.
  Hence $[\ev_1^*\upxi] = [\ev_0^*\upxi]$ also holds in the analytic versions of the Hochschild and negative cyclic cohomology.
\end{proof}

\begin{cor}
  For every analytic $A_\infty$-algebra $A$ admitting an analytic minimal model $\hh(A)$ there are isomorphisms $\HH^\an(A,A^\vee_\Delta) \cong \HH^\an(\hh(A),\hh(A)^\vee_\Delta)$ and
  $\HC^\an_\lambda(A) \cong \HC^\an_\lambda(\hh(A))$.
\end{cor}

\subsection{Analytic inverses for bimodule morphisms}

In the remaineder of this section we show the analogue of~\cref{analinverse} for bimodule morphisms: given $\uprho \in \Hom_{A-A}^\an(M,N)$ such that $\uprho_{0,0}$ is invertible in $\grNMod \elle$, we show that it admits an analytic inverse $\uptau \in \Hom_{A-A}^\an(N,M)$.
To check the growth condition on $\uptau$, we define it explicitly using a tree formula using a special kind of planar tree.

\begin{defn}
  A \emph{caterpillar} is a rooted planar tree $T$ with one of the leaves marked, such that all internal nodes have valency $\geq 3$ and lie on a \emph{central path} between the root and the marked leaf.
  The set of caterpillars for which $n_1$ unmarked leaves lie to the left of this central path and $n_2$ unmarked leaves lie to the right of the central path is denoted $\catp(n_1,n_2)$.
  The subset of caterpillars with $d$ internal nodes is denoted $\catp(n_1,n_2,d)$.
\end{defn}
\begin{figure}[h]
  \centering
  \begin{subfigure}[b]{.49\textwidth}
    \centering
    \begin{tikzpicture}[scale=2]
      \begin{scope}[inner sep=1pt]
        \node[inner sep=0pt] (A) at (0,0) {$\scriptstyle\times$};
        \node[fill=black,circle] (B) at (0,1) {};
        \node[fill=black,circle] (B1) at (.25,1.5) {};
        \node[fill=black,circle] (B2) at (.5,1.5) {};
        \node[fill=black,circle] (B3) at (-.25,1.5) {};
        \node[fill=black,circle] (B4) at (-.5,1.5) {};
        \node[fill=black,circle] (D) at (0,2) {};
        \node[fill=black,circle] (D1) at (.5,2.5) {};
        \node[fill=black,circle] (D2) at (.25,2.5) {};
        \node[fill=black,circle] (D3) at (-.5,2.5) {};
        \node[inner sep=0pt] (E) at (0,3) {$\scriptstyle\times$};
      \end{scope}
      \draw (A) -- (B) -- (D) -- (E);
      \draw (B) -- (B1);
      \draw (B) -- (B2);
      \draw (B) -- (B3);
      \draw (B) -- (B4);
      \draw (D) -- (D1);
      \draw (D) -- (D2);
      \draw (D) -- (D3);
    \end{tikzpicture}
  \end{subfigure}
  \begin{subfigure}[b]{.49\textwidth}
    \centering
    \begin{tikzpicture}[scale=2]
      \begin{scope}[inner sep=1pt]
        \node[inner sep=0pt] (A) at (0,0) {};
        \node[draw, inner sep=2pt] (B) at (0,1) {$\scriptstyle\uprho_{2,2}$};
        \node (B1) at (1,3) {};
        \node (B2) at (1.25,3) {};
        \node (B3) at (-1,3) {};
        \node (B4) at (-1.25,3) {};
        \node[draw,  inner sep=2pt] (D) at (0,2) {$\scriptstyle\uprho_{2,1}$};
        \node (D1) at (.75,3) {};
        \node (D2) at (.5,3) {};
        \node (D3) at (-.5,3) {};
        \node (E) at (0,3) {};
      \end{scope}
      \draw (A) --node[midway,draw=black,fill=white,inner sep=1.5pt] {$\scriptstyle\uptau_{0,0}$} (B) -- node[midway,draw=black,fill=white,inner sep=1.5pt] {$\scriptstyle\uptau_{0,0}$} (D) -- node[midway,draw=black,fill=white,inner sep=1.5pt] {$\scriptstyle\uptau_{0,0}$} (E);
      \draw (B) to[out=10,in=-90]  (B1);
      \draw (B) to[out=0,in=-90] (B2);
      \draw (B) to[out=170,in=-90] (B3);
      \draw (B) to[out=180,in=-90] (B4);
      \draw (D) to[out=0,in=-90] (D1);
      \draw (D) to[out=10,in=-90] (D2);
      \draw (D) to[out=180,in=-90] (D3);
    \end{tikzpicture}
  \end{subfigure}
  \caption{
    Left: an example of a tree $T \in \catp(3,4,2)$ with the root at the bottom. Right: the string diagram for the corresponding map $\uptau_T \colon A[1]^{\otimes 3} \otimes M[1] \otimes A[1]^{\otimes 4} \to N[1]$.
  }
  \label{fig:catptree}
\end{figure}
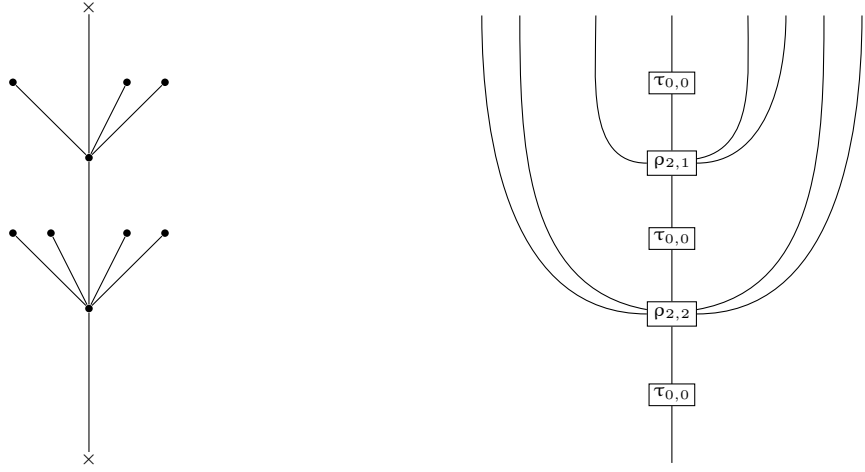

A typical example of a caterpillar tree is given in~\cref{fig:catptree}.
There is a unique tree $T\in \catp(0,0)$ which has $d=0$ internal nodes.
For $n>1$ one has $d>1$ and every tree $T\in \catp(n_1,n_2,d)$ can be uniquely decomposed into an internal node and a subtree in two ways:
\[
  T \quad=\quad
  \begin{tikzpicture}[baseline=(current bounding box.center)]
    \node[inner sep=0pt,outer sep=0pt] (A) at (0,0) {$\scriptstyle\times$};
    \node[inner sep=1pt, fill=black,circle] (B) at (1,0) {};
    \node[inner sep=1pt, fill=black,circle] (C1) at (2,.8) {};
    \node (C2) at (2,.6) {$\vdots$};
    \node[inner sep=1pt, fill=black,circle] (C3) at (2,.2) {};
    \node[inner sep=1pt, fill=black,circle] (C4) at (2,-.8) {};
    \node (C5) at (2,-.4) {$\vdots$};
    \node[inner sep=1pt, fill=black,circle] (C6) at (2,-.2) {};
    \node (D) at (3.5,0) {$T_r$};
    \draw (A) -- (B) -- (D);
    \draw (B) -- (C1);
    \draw (B) -- (C3);
    \draw (B) -- (C4);
    \draw (B) -- (C6);
    \draw [decorate, decoration = {calligraphic brace,mirror}, thick] (2.3,.1) to[edge label'=$\scriptstyle l_1$] (2.3,.9);
    \draw [decorate, decoration = {calligraphic brace}, thick] (2.3,-.1) to[edge label=$\scriptstyle l_2$]  (2.3,-.9); 
  \end{tikzpicture}
  \quad=\quad
  \begin{tikzpicture}[baseline=(current bounding box.center)]
    \node (A) at (0,0) {$T_l$};
    \node[inner sep=1pt, fill=black,circle] (B) at (1,0) {};
    \node[inner sep=1pt, fill=black,circle] (C1) at (2,.8) {};
    \node (C2) at (2,.6) {$\vdots$};
    \node[inner sep=1pt, fill=black,circle] (C3) at (2,.2) {};
    \node[inner sep=1pt, fill=black,circle] (C4) at (2,-.8) {};
    \node (C5) at (2,-.4) {$\vdots$};
    \node[inner sep=1pt, fill=black,circle] (C6) at (2,-.2) {};
    \node[inner sep=0pt,outer sep=0pt] (D) at (3.5,0) {$\scriptstyle\times$};
    \draw (A) -- (B) -- (D);
    \draw (B) -- (C1);
    \draw (B) -- (C3);
    \draw (B) -- (C4);
    \draw (B) -- (C6);
    \draw [decorate, decoration = {calligraphic brace,mirror}, thick] (2.3,.1) to[edge label'=$\scriptstyle r_1$] (2.3,.9);
    \draw [decorate, decoration = {calligraphic brace}, thick] (2.3,-.1) to[edge label=$\scriptstyle r_2$]  (2.3,-.9); 
  \end{tikzpicture}
\]
so $T$ is unique determined by a pair of numbers $(l_1,l_2)$ and a tree $T_r \in \catp(n_1-l_1,n_2-l_2,d-1)$, and also by a pair of numbers $(r_1,r_2)$ and a tree as $T_l \in \catp(n_1-r_1,n_2-r_2,d-1)$.
Hence we have bijections
\[
  \begin{aligned}
    \catp(n_1,n_2) &\cong \{(l_1,l_2,T_r) \mid l_1+l_2>0,\ T_r \in \catp(n_1-l_1,n_2-l_2)\}
    \\&\cong  \{(T_l,r_1,r_2) \mid r_1+r_2>0,\ T_l \in \catp(n_1-r_1,n_2-r_2)\}
  \end{aligned}
\]
Given $\uprho\in\hom_{A-A}(M,N)$ and $\uptau_{0,0} \in \hom_{\elle}^\cont(N,M)$ we then define for every $T\in\catp(n_1,n_2)$ a bimodule pre-morphism as follows.
For $n_1=n_2 =0$ we set $\uptau_T = \uptau_{0,0}$, and otherwise we define $\uptau_T$ via the following equivalent recursive formulas:
\begin{equation}\label{eq:bitreeinv}
  \begin{aligned}
    \uptau_T &= - \uptau_{0,0} \circ \uprho_{l_1,l_2} \circ (\id^{\otimes l_1} \otimes \underline{\uptau_{T_r}}\otimes \id^{\otimes l_2})\\
           &= - \uptau_{T_l} \circ (\id^{\otimes (n_1-r_1)} \otimes \underline{(\uprho_{r_1,r_2}\circ \uptau_{0,0})} \otimes \id^{\otimes (n_2-r_2)}).    
  \end{aligned}
\end{equation}
See again figure~\cref{fig:catptree} for an example of the map $\uptau_T$.
If $\uptau_{0,0}$ is an inverse for $\uprho_{0,0}$ then the sum of these maps determines an inverse bimodule map.

\begin{lem}
  Let $\uprho \in \hom_{A-A}(M,N)$ be a pre-morphism such that $\uprho_{0,0}$ admits a left-/right inverse $\uptau_{0,0}$.
  Then $\uprho$ admits a left-/right inverse $\uptau \in \hom_{A-A}(N,M)$ with components
  \[
    \uptau_{i,j} = \sum_{T\in \catp(i,j)} \uptau_T,
  \]
  where the maps $\uptau_T$ are defined as above.
\end{lem}
\begin{proof}
  Suppose $\uptau_{0,0}$ is a right-inverse for $\uprho_{0,0}$, then we claim that $\uptau$ is a right-inverse for $\uprho$.
  The base case $i=j = 0$ is trivial as
  \[
    (\uprho\diamond\uptau)_{0,0} = \uprho_{0,0} \circ \sum_{T\in\catp(0,0)} \uptau_T = \uprho_{0,0} \circ \uptau_{0,0} = \id_{M[1]}.
  \]
  For $i+j>0$ it follows from the recursive definition of the maps $\uptau_T$ that
  \[
    \begin{aligned}
      (\uprho\diamond\uptau)_{i,j} &= \uprho_{0,0} \circ \uptau_{i,j} + \sum_{l_1+l_2>0}\sum_{T_r\in\catp(i-l_1,j-l_2)} \uprho_{l_1,l_2} \circ (\id^{\otimes l_1} \otimes \uprho_{T_r} \otimes \id^{\otimes l_2})
      \\&= \uprho_{0,0} \circ \uptau_{i,j} + \sum_{l_1+l_2>0}\sum_{T_r\in\catp(i-l_1,j-l_2)} \uprho_{0,0} \circ \uptau_{0,0} \circ \uprho_{l_1,l_2} \circ (\id^{\otimes l_1} \otimes \uprho_{T_r} \otimes \id^{\otimes l_2})
      \\&= \uprho_{0,0} \circ \uptau_{i,j} - \sum_{T \in \catp(i,j)} \uprho_{0,0} \circ \uptau_{T} = 0,
    \end{aligned}    
  \]
  where the final line uses the bijection between trees in $\catp(i,j)$ and triples $(l_1,l_2,T_r)$ discussed above.
  It follows that $\uprho\diamond \uptau = \id$, so $\uptau$ is a right inverse.
  The case where $\uptau_{0,0}$ is a left-inverse is similar, using the other decomposition into triples $(T_l,r_1,r_2)$.
  In particular, $\uptau$ is a two-sided inverse if $\uptau_{0,0}$ is.
\end{proof}

We remark that if $\uprho$ is a morphism and $\uptau=\uprho^{-1}$ is a two-sided inverse, then the latter is again a morphism by the equality
\[
  0 = \uptau \diamond \updelta(\id) = \uptau \diamond \updelta(\uprho)\diamond \uptau + \updelta(\uptau) = \updelta(\uptau).
\]
We claim that the map $\uptau$ is moreover analytic when $\uprho$ is analytic.

\begin{prop}\label{analbiinverse}
  Let $\uprho \in \hom_{A-A}^\an(M,N)$ be an analytic pre-morphism, and suppose $\uprho_{0,0}$ admits a continuous left-/right-inverse $\uptau_{0,0}$.
  Then $\uprho$ admits an analytic left-/right-inverse.
\end{prop}
\begin{proof}
  Because $\uprho$ is analytic and $\uptau_{0,0}$ is continuous, there exists $C>0$ such that $\nrm{\uptau_{0,0}} < C$ and $\nrm{\uprho_{i,j}} < C^{i+j+1}$ for all $i,j\in\N$.
  We claim that for $T\in\catp(n_1,n_2,d)$ the map $\uptau_T$ satisfies $\nrm{\uptau_T} < C^{n_1+n_2+2d+1}$.
  For $n_1=n_2 = 0$ and $d=0$ one has
  \[
    \nrm{\uptau_T} = \nrm{\uptau_{0,0}} < C = C^{0+0+0+1}.
  \]
  Now let $n_1,n_2$ and $d>0$ be arbitrary, then assuming the bound holds for all $d' < d$ it follows by the recursive formula that
  there exists $(l_1,l_2)$ and a tree $T_r\in \catp(n_1-l_1,n_2-l_2,d-1)$ such that
  \[
    \nrm{\uptau_T} = \nrm{\uptau_{0,0}} \nrm{\uprho_{l_1,l_2}} \nrm{\uptau_{T_r}} < C \cdot C^{l_1+l_2+1} \cdot C^{(n_1-l_1)+(n_1-l_2)+2(d-1)+1} = C^{n_1+n_2+2d+2}.
  \]
  As the number of internal nodes satisfies $d\leq i+j$, it follows that $\nrm{\uptau_T} < C^{3(i+j)+3}$ for $i+j> 0$.
  Hence, taking the sum over all trees in $\catp(i,j)$ one finds that
  \[
    \nrm{\uptau_{i,j}} \leq \sum_{T\in\catp(i,j)} \nrm{\uptau_T} < \sum_{T\in\catp(i,j)} C^{3(i+j)+3} = |\catp(i,j)| (C^3)^{i+j+1}.
  \]
  Each $T\in \catp(i,j)$ corresponds to a unique planar tree with less than $3(i+j)+3$ total nodes, so the cardinality of the set $\catp(i,j)$ is clearly bounded by the Catalan number $\cC_{3(i+j)+3}$.
  Taking again a constant $q>1$ such that $\cC_n < q^n$ for $n\gg 0$, it follows that
  $\nrm{\uptau_{i,j}} < (qC)^{i+j+1}$ for $i+j\gg 0$.
  Hence $\uptau$ is analytic.  
\end{proof}

\section{Cyclic structures and analytic Calabi-Yau structures}\label{sec:cyclic}

The notion of a cyclic structure on an $A_\infty$-algebra was introduced in \cite{Kaj07,KS09} as a non-degenerate skew-symmetric pairing $\<-,-\> \colon A[1] \otimes A[1] \to \ell[2-d]$ satisfying a condition
\[
  \<\upmu_n(a_1,\ldots,a_n),a_{n+1}\> = (-1)^K \<a_1,\upmu_n(a_2,\ldots,a_{n+1})\>, \quad\forall n\in\N
\]
for some appropriate Koszul sign $K$.
In particular, a cyclic $A_\infty$-algebra is always finite dimensional.
In practice, it is easier to work with an alternative characterisation of the cyclic structures as a bimodule map \cite{Tra08,Cho08} which is as follows.

\begin{defn}\label{def:cyclic}
  A \emph{$d$-cyclic structure} on $A\in\AinfAlg$ is a morphism $\upsigma \in \Hom_{A-A}(A_\Delta,A_\Delta^\vee[-d])$ such that $\upsigma_{0,0}$ is a skew-symmetric isomorphism and $\upsigma_{i,j} = 0$ for $i+j > 0$.
\end{defn}

A pair $(A,\upsigma_A)$ of an $A_\infty$-algebra and a cyclic $A_\infty$-structure is called a cyclic $A_\infty$-algebra, and given two such cyclic $A_\infty$-algebras $(A,\upsigma_A)$ and $(B,\upsigma_B)$ a morphism $f\in\Hom_{\AinfAlg}(A,B)$ is said to be cyclic precisely if $f^* \upsigma_B = \upsigma_A$.

A case of special interest is when $A$ is $d$-cyclic for $d=3$, and $A^0 \cong \ell$.
In this case, the structure of the cyclic $A_\infty$-algebra $(A,\upsigma_A)$ can be recovered from a \emph{potential}:
setting $V_A = (A^1)^*$ the cyclic structure defines an element
\[
  W_A \colonequals \sum_{n\in\N} \upsigma(\upmu_n(-,\ldots,-))(-) \in \widehat\barT_\ell V_A
\]
in the \emph{completed tensor algebra} $\widehat\barT_\ell V_A \colonequals \prod_{n\geq 0} V_A^{\otimes n}$ on $V_A$, which can be used to recover the $A_\infty$-structure (see e.g. \cite{VdBer15}).
Likewise, given an $\ell$-bimodule $V$ and an element $W = \sum_n W_n \in \widehat\barT_\ell V$ such that each $W_n$ is invariant under cyclic permutation, there is a well-defined cyclic $A_\infty$-algebra structure on
\[
  A_V = \ell \oplus V^*[-1] \oplus V[-2] \oplus \ell^*[-3].
\]
Any (pre-)morphism $f\in \hom_{\AinfAlg}^0(A,B)$ induces a graded algebra morphism $f^* \colon \widehat\barT_\ell V_B \to \widehat\barT_\ell V_A$, and a morphism $f$ is cyclic if and only if $f^*(W_B) = W_A$ by a result of Kajiura \cite{Kaj07}.

In this section we consider cyclic analytic $A_\infty$-algebras, by which we mean a pair $(A,\upsigma_A)$ with $A\in\AinfAlgan$ and $\upsigma_A$ a cyclic structure in the above sense.
In the $d=3$ case the analytic requirement yields an \emph{analytic potential}, as defined in \cite{HK19}.

\begin{prop}\label{analpotential}
  Let $(A,\upsigma)$ be a 3-cyclic analytic $A_\infty$-algebra with $A^0 \cong \ell$ and pick a basis $v_1,\ldots v_m$ for $V_A$, then the potential $W_A$ lies in the analytic subring
  \[
    \widetilde\barT_\ell V_A \colonequals \left\{\sum_{n=0}^\infty \sum_{i_1,\ldots,i_n} c_{i_1,\ldots,i_n} \cdot v_{i_1}\otimes \ldots \otimes v_{i_n} \in \widehat\barT_\ell V_A \ \middle|\ %
      \begin{gathered}
        \text{ there exists } C>0 \text{ such that } \\
        |c_{i_1,\ldots,i_n}| < C^n \text{ for all }  i_1,\ldots,i_n
      \end{gathered}
      \ %
    \right\}.
  \]
\end{prop}
\begin{proof}
  Let $a_1,\ldots,a_m \in A^1$ be a dual basis to $v_1,\ldots,v_m \in V_A = (A^1)^*$.
  Then the coefficient $c_{i_1,\ldots,i_n}$ of $W_A$ in $\widehat\barT_\ell V_A$ is given by $c_{i_1,\ldots,i_n} = \upsigma(\upmu_{n-1}(a_{i_1},\ldots,a_{i_{n-1}}))(a_{i_n})$.
  Because $\upmu$ is analytic by assumption, there exists $C_0$ such that $\nrm{\upsigma} < C_0$ and $\nrm{\upmu_n} < C^n$ for all $n$, and $\nrm{a_i} < C_0$ for each basis vector.
  It follows that there is a bound for all $i_1,\ldots,i_n$:
  \[
    |c_{i_1,\ldots,i_n}| = |\upsigma(\upmu_{n-1}(a_{i_1},\ldots,a_{i_{n-1}}))(a_{i_n})| \leq \nrm{\upsigma}\nrm{\upmu_{n-1}}\nrm{a_{i_1}}\cdots\nrm{a_{i_n}} < C_0 \cdot C_0^{2n-1} \cdot C_0^n = (C_0^2)^n,
  \]
  so setting $C = C_0$, the result follows.
\end{proof}

\begin{rem}
  Hua--Keller \cite{HK19} use the language of quivers to define their potential.
  Here the base is $\ell = \prod_{v\in Q_0} \C v$ is a product over the nodes of a quiver $Q$, and the $\ell$-bimodule $V$ should be seen as the span of the arrows in $Q$.
  This identifies $\widehat\barT_\ell V$ with the completed path algebra $\widehat{\C Q}$,
  and identifies the subring $\widetilde\barT_\ell V_A$ with the analytic path algebra $\widetilde{\C Q} \subset \widehat{\C Q}$ in \cite{HK19}.
\end{rem}

In the rest of this section we will show how to obtain a cyclic structure on analytic minimal models of analytic $A_\infty$-algebras that satisfy a Calabi--Yau property.
This requires us to consider a ``homotopic'' generalisation of cyclic structures, called a \emph{strong homotopy inner product}.

\subsection{Strong homotopy inner products}

If $\upsigma$ is a cyclic structure on an $A_\infty$-algebra $B$ then the pull-back $f^*\upsigma$ along a quasi-morphism $f\in\Hom_{\AinfAlg}(A,B)$ will, in general, not be a cyclic structure on $A$ because both the condition
\[
  (f^*\upsigma)_{i,j} = 0 \text{ if } i+j>0,
\]
as well as the requirement that $(f^*\upsigma)_{0,0}$ is an isomorphism may be violated;
one says that cyclic structures are \emph{strict}.
For this reason it is better to use a homotopy-invariant alternative which relaxes the cyclic conditions.
One such alternative was define by Cho \cite{Cho08}, which we recall here following the setup of \cite{AT22}.

\begin{defn}
  Given $A \in \AinfAlg$, a \emph{closed 2-form} is an element of the subcomplex $\Omega^{2,\cl}(A)\subset \hom_{A-A}^\bullet(A_\Delta,A_\Delta^\vee)[-2]$ of shifts of bimodule morphisms $\uprho$ satisfying:
  \begin{itemize}
  \item \emph{skew-symmetry}: for any $i,j\in\N$ and $a_0\otimes\cdots\otimes a_{i+j+1} \in (A[1])^{\otimes i+j+1}$ the relation
    \[
      \uprho(a_1\otimes\cdots\otimes\underline a_{i+1}\otimes\cdots\otimes a_n)(a_0) =
      (-1)^{K}       \uprho(a_{i+2}\otimes\cdots\otimes\underline a_0\otimes\cdots\otimes a_i)(a_{i+1})
    \]
    holds, where $K$ is obtained from the Koszul sign rule after cyclic permutation
  \item \emph{closedness}: for any $a_0\otimes\cdots\otimes a_n \in (A[1])^{\otimes n}$ and indices $1\leq i < j < k \leq n$ the relation
    \[
      \begin{aligned}
        (-1)^{K_i}\uprho(a_{j+1}\otimes\cdots\otimes\underline a_i\otimes\cdots\otimes a_{j-1})(a_j) &+ (-1)^{K_j}\uprho(a_{k+1}\otimes\cdots\otimes\underline a_j\otimes\cdots\otimes a_{k-1})(a_k) \\&+ (-1)^{K_k}\uprho(a_{i+1}\otimes\cdots\otimes\underline a_k\otimes\cdots\otimes a_{i-1})(a_i) = 0
      \end{aligned}
    \]
    holds, where $K_i,K_j,K_k$ are obtained from the Koszul sign rule after cyclic permutation.
  \end{itemize}
\end{defn}
\begin{defn}
  A \emph{strong homotopy inner product} (SHIP) is a cocycle $\uprho\in\Omega^{2,\cl}(A)$ of some degree $d$ such that the map $\uprho_{0,0} \colon A_\Delta[1] \to A_\Delta^\vee[d-1]$ is a quasi-isomorphism.
\end{defn}

Any cyclic structure defines a SHIP, but the latter class is much better behaved under $A_\infty$-morphisms: given $f\in\Hom_{\AinfAlg}(A,B)$ the pullback $f^*$ on bimodule morphisms restricts to a chain map
\[
  f^* \colon \Omega^{2,\cl}(B) \to \Omega^{2,\cl}(A),
\]
which maps SHIPs to SHIPs when $f$ is a quasi-isomorphism.
Following the philosophy of \cite{KS09}, a SHIP can be viewed as a type of noncommutative shifted-symplectic structure:
one can view the Hochschild cohomologies $\barC^\bullet(A)$ and $\barC^\bullet(A,A^\vee)$ as vector fields and differential 1-forms on a noncommutative space, and any cocycle $\uprho \in \Omega^{2,\cl}(A)$ defines a \emph{contraction map}
\[
  \iota_-\uprho\colon \barC^\bullet(A) = \barC^\bullet(A,A) \to \barC^\bullet(A,A_\Delta^\vee),\quad \upalpha \mapsto \uprho \circ \overline\upalpha,
\]
which is a quasi-isomorphism if $\uprho$ is a SHIP.
We now generalise the notion of strong homotopy inner product to analytic $A_\infty$-algebra, as follows.

For an analytic $A_\infty$-algebra $A\in\AinfAlgan$, we define the complex of \emph{analytic closed 2-forms} 
\[
  \Omega^{2,\cl,\an}(A) \colonequals \Omega^{2,\cl}(A) \cap \fA(A_\Delta,A_\Delta')[-2],
\]
where we view $\fA(A_\Delta,A_\Delta')$ as a subspace of $\fM(A_\Delta,A_\Delta^\vee)$ via the inclusion $A_\Delta' \hookrightarrow A_\Delta^\vee$.
The assignment $A\mapsto \Omega^{2,\cl,\an}(A)$ is again functorial, as \cref{analpullback} guarantees that there is a chain map
\[
  f^* \colon \Omega^{2,\cl,\an}(B) \to \Omega^{2,\cl,\an}(A),
\]
for every analytic $A_\infty$-morphism $f\in\Hom_{\AinfAlg}^\an(A,B)$.
We define the following analogue of a SHIP for analytic $A_\infty$-algebras admitting an analytic minimal model.

\begin{defn}
  Let $A$ be an analytic $A_\infty$-algebra $A$ with minimal model $\hh(A)$.
  Then an \emph{analytic SHIP} is a cocycle $\uprho\in \Omega^{2,\cl,\an}(A)$ such that the map
  \[
    \hh(A)[1] \xrightarrow{\ I\ } A[1] \xrightarrow{\uprho_{0,0}} A^\vee[1] \xrightarrow{\ I^\vee\ } \hh(A)^\vee[1] 
  \]
  admits a continuous inverse.
\end{defn}

As before, an analytic SHIP can be though of as a noncommutative shifted symplectic structure with contraction map
\[
  \iota_-\uprho \colon \barC^{\an,\bullet}(A) \to \barC^{\an,\bullet}(A,A_\Delta'),
\]
which is well-defined by the discussion in \cref{analhoch}.
If $\hh(A)$ is a strong minimal model, one can again show that this map is a quasi-isomorphism.
In what follows we consider finite dimensional minimal $A_\infty$-algebras, for which it suffices that $\uprho_{0,0}$ is invertible and the contraction map is an isomorphism.

\begin{prop}
  Suppose $A$ is a finite dimensional minimal $A_\infty$-algebra and $\uprho\in\Omega^{2,\cl,\an}(A)$ a cocycle with $\uprho_{0,0}$ invertible in $\grMod\elle$.
  Then $\uprho$ is an analytic  SHIP and $\iota_-\uprho$ is an isomorphism.
\end{prop}
\begin{proof}
  Continuity of $\uprho_{0,0}^{-1}$ is automatic if $A$ is finite dimensional, hence $\uprho$ is an analytic SHIP.
  It then follows from \cref{analbiinverse} that $\uprho$ admits an analytic inverse $\uptau \in \Hom_{A-A}^\an(A'_\Delta,A_\Delta)$, and the induced map
  \[
    \barC^{\an,\bullet}(A,A_\Delta')\to \barC^{\an,\bullet}(A) \quad \upxi \mapsto \uptau \circ \overline{\upxi}
  \]
  is then inverse to the contraction map by bi-functoriality of the Hochschild cochain complex.
\end{proof}

\subsection{An analytic Darboux lemma}

Kontsevich--Soibelman \cite{KS09} showed that when interpreted in the language of noncommutative symplectic geometry, any SHIP can be put into a standard \emph{Darboux form}.
In the algebraic setup we use here, this means that any SHIP $\uprho\in \Omega^{2,\cl}(A)$ on a minimal $A_\infty$-algebra $A = (A^\bullet,\upmu)$ can be strictified to a genuine cyclic structure
\[
  \upsigma = f^*\uprho \in \Omega^{2,\cl}(A^\per)
\]
where $A^\per = (A^\bullet,\upmu^\per)$ is a perturbation of the $A_\infty$-structure and $f\in \Hom_{\AinfAlg}(A^\per,A)$ is an $A_\infty$-morphism with first order term $f_1 = \id_A$.
A Darboux lemma of this form was first shown by Cho--Lee \cite{CL11}.
More recently, Amorim--Tu \cite{AT22} show that the morphism $f$ can be obtained as $f=f^1$ from a solution of a differential equation
\begin{equation}\label{crucialdiffeq}
  \left\{
  \begin{aligned}
    \frac{\d}{\d t} f^t &= \upalpha^t \diamond f^t\\
    f^0 &= \id,
  \end{aligned}
  \right.
\end{equation}
where $t\mapsto \upalpha^t$ is certain one-parameter family of Hochschild cocycles interpolating between $\uprho$ and $\upsigma$, chosen such that the solution satisfies $f^1 = f$.
In what follows we give an analytic version of this result using an explicit tree expression for the solution.

Fix a finite dimensional minimal $A \in \AinfAlgan$ and let $\upalpha \in \fM(A,A\otimes \pc^\infty([0,1]))[-1]$ be a continuous family of Hochschild cochains $\upalpha^t \in \barC_\bullet(A)$.
Then we define for every $T\in\O(n)$ a continuous family
\[
  f_T\in \hom_\elle^\cont(A[1]^{\otimes n}, A[1] \otimes \pc^\infty([0,1]))
\]
as follows.
The unique tree $T\in \O(1)$ determines the constant map $f_T^t = \id_{A[1]}$, and for $T\in\O(n)$ with $n>1$ the map $f_T$ is given by the recursive formula:
\[
  f^t_T = \int_0^t \upalpha^\uptau_k \circ (f_{T_1}^\uptau \otimes \cdots \otimes f_{T_k}^\uptau) \d \uptau,
\]
where $(T_1,\ldots,T_k)$ are the subtrees emanating from the first internal node as in \cref{analinverse}.
The following lemma shows that this yields a solution for the differential equation of \cite{AT22}.

\begin{lem}\label{treemorb}
  Suppose $\upalpha^t_0 = \upalpha^t_1 = 0$. Then the family $f \in \overline\fM(A,A\otimes \pc^\infty([0,1]))$ defined by
  \[
    f_n = \sum_{T \in \O(n)} f_T,
  \]
  is a solution to the differential equation \labelcref{crucialdiffeq}.
\end{lem}
\begin{proof}
  For $t=0$ we have $f^0_T = \id_{A[1]}$ for the unique tree $T\in \O(1)$ and $f_T^0 = 0$ for $T\in \O(n)$ with $n>1$, as the latter are given by an empty integral.
  Hence $f$ satisfies the initial condition $f^0 = \id$ in \labelcref{crucialdiffeq} and just remains to check that it satisfies the differential equation
  \begin{equation}\label{eq:actualdiffeq}
    \frac{\d}{\d t} f^t_n = (\upalpha^t\diamond f^t)_n \colonequals \sum_{k=1}^n \sum_{n_1+\cdots+n_k = n} \upalpha_k^t \circ (f_{n_1} \otimes \cdots \otimes f_{n_k}).
  \end{equation}
  Applying the recursive formula for $f_T$ it follows that
  \[
    \begin{aligned}
      \frac{\d}{\d t} f^t_n
      &=  \sum_{T\in\O(n)} \frac{\d}{\d t} f_T^t
      \\&= \sum_{k=2}^{n} \sum_{n_1+\cdots + n_k = n} \sum_{T_i\in\O(n_i)} \upalpha_k^t \circ (f_{T_1}^t \otimes\cdots\otimes f_{T_k}^t)
      \\&=  \sum_{k=2}^n\sum_{n_1+\cdots + n_k = n} \upalpha_k^t \circ\left(
      \left(\sum_{T_1\in\O(n_1)} f_{T_1}^t\right) \otimes\cdots\otimes
      \left(\sum_{T_k\in\O(n_k)} f_{T_k}^t\right)\right)
      \\&= \sum_{k=2}^n\sum_{n_1+\cdots + n_k = n} \upalpha_k^t \circ (f_{n_1} \otimes \cdots \otimes f_{n_k}),
    \end{aligned}    
  \]
  which is equal to \labelcref{eq:actualdiffeq} under the assumption $\upalpha^t_1 = 0$.
  Hence $f$ is a solution to \labelcref{crucialdiffeq} as claimed.
\end{proof}

Let $\pc^\infty([0,1])$ again be endowed with the essential supremum norm.
Then a continuous pre-morphism $f\in\fM(A,A\otimes\pc^\infty([0,1])$ is analytic if and only if there exists $C>0$ such that for all $n\geq 1$
\[
  \nrm{f_n}_\infty < C^n,
\]
where $\norm_\infty$ denotes the supremum of the operator norms $\nrm{f_n^t}$ for $t\in [0,1]$.
We can use this to prove the following.

\begin{lem}\label{analmorb}
  In the situation of \cref{treemorb}, suppose that $\upalpha$ is analytic, then $f$ is also analytic.
\end{lem}
\begin{proof}
  Let $C>0$ be such that $\nrm{\upalpha_n}_\infty < C^n$ for all $n\in\N$.
  We will show that $\nrm{f_T}_\infty \leq C^{n+d-1}$ holds for each $T \in \O(n,d)$ by induction over $n$.
  For the unique tree $T \in \O(1) = \O(1,0)$ then 
  \[
    \nrm{f_T}_\infty = \sup_{t\in [0,1]} \nrm{\id} = 1 = C^{1-0-1},
  \]
  which establishes the base case.
  For the induction step let $n>1$ and assume the statement holds for all $T\in \O(m,d)$ with $m<n$.
  Then $T\in \O(n,d)$ is built out of $l\geq 2$ subtrees $T_1,\ldots,T_l\subset T$ with $T_i \in \O(n_i,d_i)$ satisfying $n_i\neq 0$, $\sum n_i = n$, and $\sum_{i=1}^k d_i = d-1$.
  The because each $n_i$ is less than $n$ it follow by the induction hypothesis that
  \[
    \begin{aligned}
      \nrm{f_T}_\infty
      &= \sup_{t\in[0,1]}\nrm{\int_0^t \upalpha^\uptau_k \circ (f_{T_1}^\uptau \otimes \cdots \otimes f_{T_l}^\uptau) \d \uptau}
      \\&\leq \sup_{\uptau \in [0,1]} \nrm{\upalpha^\uptau_k} \nrm{f_{T_1}^\uptau} \cdots \nrm{ f_{T_k}^\uptau }
      \\&\leq \nrm{\upalpha_k}_\infty \nrm{f_{T_1}}_\infty \cdots \nrm{f_{T_k}}_\infty
      \\&\leq C^k\cdot C^{n_1+d_1-1} \cdots C^{n_k+d_k-1} = C^{n+d-1}
    \end{aligned}
  \]
  which verifies the induction.
  Because the valency of every internal node of a tree in $\O(n)$ is at least $3$, one easily verifies that the maximal number of internal nodes is $d\leq n$.
  Hence, the sum over all trees is bounded by
  \[
    \nrm{f_n}_\infty \leq \sum_{T\in \O(n)} \nrm{f_T}_\infty \leq |\O(n)| \cdot C^{2n-1}
  \]
  where $|\O(n)|$ denotes the cardinality of $\O(n)$, which is bounded by the Catalan number $\cC_{2n}$.
  These Catalan numbers are bounded by $q^n$ for some fixed $q>0$, so we obtain a bound
  \[
    \nrm{f_n}_\infty < \cC_{2n}\cdot C^{2n} < q^nC^{2n} = (qC^2)^n.
  \]
  which shows that $f$ is uniformly analytic.  
\end{proof}

Using the above and the results of the previous section, we arrive at an analytic generalisation of the Darboux lemma appearing in \cite{KS08,CL11,AT22}.

\begin{lem}[Analytic Darboux]\label{Darboux}
  Let $A \in \AinfAlgan$ be minimal and finite dimensional and suppose $\uprho \in \Omega^{2,\cl,\an}(A)$ is a SHIP of degree $2-d$.
  Then:
  \begin{enumerate}
  \item there exists an $A_\infty$-algebra $A^\uprho$ and an analytic $A_\infty$-isomorphism $f \in \Hom_{\AinfAlg}^\an(A^\uprho,A)$ such that
    \[
      \upsigma^\uprho \colonequals f^*\uprho
    \]
    is a cyclic structure.
  \item if $\uptau \in \Omega^{2,\cl,\an}(A)$ is another SHIP such that the image $\uptau - \uprho$ along the inclusion $\Omega^{2,\cl,\an}(A) \to \Omega^{2,\cl}(A)$ is exact, then there is a cyclic analytic $A_\infty$-isomorphism
    \[
      (A^\uprho,\upsigma^\uprho) \cong_{\an,\cyc} (A^\uptau,\upsigma^\uptau).
    \]
  \end{enumerate}
  Effectively, every class in $\hh^{2-d}\Omega^{2,\cl}(A)$ admitting an analytic lift defines a cyclic analytic model.
\end{lem}
\begin{proof}
  (1) Let $\upsigma \colonequals \uprho_{0,0}$ and consider for each $t\in [0,1]$ the analytic closed 2-form
  \begin{equation}\label{eq:formfam}
    \uprho^t \colonequals \upsigma + t(\uprho - \upsigma) \in \Omega^{2,\cl,\an}(A),
  \end{equation}
  and use \cref{analsplitepi} to fix an analytic Hochschild cocycle $\upxi \in \barC^{\an,\bullet}(A,A_\Delta')$ with image $S(\upxi) = \uprho - \upsigma$ along the surjective map $S\colon  \to \barC^{\an,\bullet}(A,A_\Delta') \to \Omega^{2,\cl,\an}(A)[1]$.
  It follows from the formula \labelcref{eq:splittingmap} that $\upxi_1$ is given on generators by
  \[
    \upxi_1(x_1)(x_0) = \tfrac12 (\uprho - \upsigma)(\underline x_0)(x_1) = 0,
  \]
  so $\upxi_1 = 0$, and likewise $\upxi_0 = 0$.
  Because $A$ is minimal the component $\uprho^t_{0,0} = \uprho_{0,0}$ is invertible and hence $\uprho^t$ is an isomorphism by \cref{analbiinverse}.
  Hence, for every $t$ the Hochschild cocycle $\upalpha^t \in \barC^{\an,\bullet}(A)$ defined by the equation
  \[
    \upalpha^t \colonequals (\uprho^t)^{-1} \circ \overline\upxi,
  \]
  satisfies the equation $\iota_{\upalpha^t}\uprho^t = \upxi$.
  It is straightforward to verify that the sequence $\upalpha = (\upalpha_n)_{n\in\N}$ of functions $\upalpha_n\colon t\mapsto \upalpha^t_n$ is in $\fM_{[0,1]}(A,A)$.
  Indeed, each component $\uprho^t_n$ varies linearly in $t$, and it is easy to verify from the recursive identity \labelcref{eq:formfam} that the function $t\mapsto (\uprho^t)_n^{-1}$ is a polynomial of degree at most $n$ in $t$.
  Hence, $\upalpha$ is a well-defined family of Hochschild cocycles, which satisfies
  \[
    \upalpha^t_1 = (\uprho^t)^{-1}_{0,0} \circ \upxi_1 = 0
  \]
  because $\upxi_1=0$, and likewise $\upalpha^t_0 = ((\uprho^t)^{-1} \circ \overline\upxi)_0 = 0$ by definition.
  To show that it is analytic we note that for each $t\in[0,1]$ we have 
  \[
    \nrm{\uprho^t_{0,0}} = \nrm{\uprho_{0,0}},\quad \nrm{\uprho^t_{i,j}} = t\nrm{\uprho^t_{i,j}} \leq \nrm{\uprho_{i,j}} \quad (i+j > 0),
  \]
  so that $\nrm{\uprho^t_n} < C_1^n$ for a uniform choice of constant $C_1>0$ independent of $t$.
  It then follows from the proof of \cref{analbiinverse} that there is a uniform constant $C_2>0$ such that  $\nrm{(\uprho^t)^{-1}_n} < C_2^n$ for all $t$, and because $\upxi$ is analytic there is some $C_3>0$ bounding the composition
  \[
    \nrm{\upalpha^t_n} = \nrm{((\uprho^t)^{-1} \circ \overline\upxi)_n} < C_3^n,
  \]
  for all $t$.
  Therefore $\upalpha = (\upalpha_n)_{n\in\N}$ is uniformly analytic, and it follows from \cref{treemorb} and \cref{analmorb} that there exists a uniformly analytic family of pre-morphisms $f = (f_n)_{n\in\N}$ satisfying the differential equation \labelcref{crucialdiffeq}.
  It then follows from \cite[Proposition 2.32]{AT22} that for each $t$ the pullback $(f^t)^* \colon \Omega^{2,\cl,\an}(A) \to \Omega^{2,\cl,\an}(A)$ satisfies
  \[
    (f^t)^*\uprho^t = \upsigma,
  \]
  so in particular $f^1$ pulls back $\uprho^1 = \uprho$ to the constant 2-form $(f^1)^*\uprho = \upsigma$.
  Because $f_1^1 = \id$, the pre-morphism $f^1_1$ admits an analytic inverse by \cref{analinverse}.
  Hence if we define $A^\uprho$ as the $A_\infty$-algebra with product defined by the analytic Hochschild cocycle
  \[
    \upmu^\uprho \colonequals (f^1)^{-1} \circ \overline\upmu_A \circ \widehat{f^1} \in \barC^{\an,\bullet}(A),
  \]
  it follows as in \cite[Proposition 2.32]{AT22} that $f^1\in \Hom_{\AinfAlg}^\an(A^\uprho,A)$ is an isomorphism and $\upsigma^\uprho = \upsigma = (f^1)^*\uprho$ is a cyclic structure as claimed.

  (2) Now let $\uptau$ be another analytic SHIP defining a cyclic analytic $A_\infty$-algebra $(A^\uptau,\upsigma^\uptau)$, and let $g \in \Hom_{\AinfAlg}^\an(A^\uptau,A)$ denote the isomorphism for which $\upsigma^\uptau = g^*\uptau$ is cyclic.
  Suppose that $\uptau - \uprho$ is exact in the complex $\Omega^{2,\cl}(A)$, then as in \cite[Lemma 2.33]{AT22} we note that $\uptau_{0,0} = \uprho_{0,0}$, so that for every $t\in[0,1]$ the analytic closed 2-form
  \[
    \uprho^t \colonequals \uptau + t(\uprho - \uptau) \in \Omega^{2,\cl,\an}(A),
  \]
  has $\uprho^t_{0,0} = \uptau_{0,0} = \uprho_{0,0}$ an isomorphism, and is therefore a SHIP.
  Fixing a cocycle $\upxi \in \barC^{\an,\bullet}(A,A'_\Delta)$ such that $S(\upxi) = \uprho - \uptau$ as before, there again exists a family of Hochschild cocycles $\upalpha$ which satisfy the equation $\iota_{\upalpha^t}\uprho^t = \upxi$ for each $t\in[0,1]$.
  The norms of $\uprho^t_n$ are bounded by $\max\{\nrm{\uptau_n},\nrm{\uprho_n}\}$ for all $t$, so it again follows that $\upalpha$ is uniformly analytic.
  Hence there exists a family of pre-morphisms $h$ with $h^0 = \id$ solving the differential equation
  \[
    \frac{\d}{\d t} h^t = \upalpha^t \diamond h^t.
  \]
  Because $\uprho - \uptau$ is exact in $\Omega^{2,\cl}(A)$ it follows that $\upxi$ is exact in $\barC^\bullet(A,A^\vee)$
  Therefore \cite[Lemma 2.33]{AT22} implies that the analytic pre-morphism $h^1$ is an $A_\infty$-morphism such that $(h^1)^*\uptau = \uprho$, hence lies in  $h^1\in \Hom_{\AinfAlg}^\an(A,A)$.
  Now the composition $g^{-1} \circ h^1 \circ f^1 \in \Hom_{\AinfAlg}^\an(A^\uptau,A^\uprho)$ is an analytic $A_\infty$-isomorphism satisfying
  \[
    (g^{-1}\circ h^1 \circ f^1)^*\upsigma^\uptau = (h^1 \circ f^1)^*\uptau = (f^1)^*\uprho = \upsigma^\uprho.\qedhere
  \]
\end{proof}

\subsection{Analytic Calabi--Yau structures and SHIPs}

Cyclic structures express a type of Calabi--Yau symmetry, which can be formalised using the notion of a \emph{right Calabi--Yau structure} in the terminology of \cite{BD19}.
For ordinary $A_\infty$-algebras, these structures can be defined as follows.

\begin{defn}
  A \emph{right $d$-Calabi--Yau structure} on an $A_\infty$-algebra $A = (A,\upmu) \in\AinfAlg$ is a cocycle $\upphi \in Z^{-d}\barC_\lambda^\bullet(A)$ for which the map $A \to A^\vee[-d] \colon a \mapsto \upphi_0(\upmu_2(a,-))$ is a quasi-isomorphism.
\end{defn}

It was shown by Cho--Lee \cite{CL11} (and \cite{KS08} in the symplectic language) that right CY structures and SHIPs on an $A_\infty$-algebra $A$ are equivalent notions.
More precisely, the two structures are identified up to homotopy by a map $\Omega^{2,\cl}(A)[1] \to \barC_\lambda^\bullet(A)$.
This map can be constructed in two steps.

\begin{prop}[see \cite{CL11}\cite{AT22}] \label{sesforcyclic}
  The map $S\colon \barC^\bullet(A,A_\Delta^\vee) \to \Omega^{2,\cl}(A)[1]$ defined by
  \[
    (S\upxi)(\mathbf{x} \otimes \underline v \otimes \mathbf y)(w)
    = \xi(\mathbf{x} \otimes v \otimes \mathbf{y})(w)
    - \xi(\mathbf{y} \otimes w \otimes\mathbf{x})(v),
  \]
  defines a chain map fitting into a short exact sequence of chain complexes
  \begin{equation}\label{eq:firstSES}
    \begin{tikzcd}
      0 \ar[r] & \barC^\bullet_\lambda(A) \ar[r] & \barC^\bullet(A,A_\Delta^\vee) \ar[r,"S",shift left=1] & \Omega^{2,\cl}(A)[1] \ar[r] & 0
    \end{tikzcd}
  \end{equation}
  where the first map is the inclusion of the cyclic cochains.
  Moreover, this short exact sequence is natural over $A\in \AinfAlg$ along the induced pullbacks on Hochschild cochains and bimodule maps.
\end{prop}

The above shows that the map $S$ induces an isomorphism between $\Omega^{2,\cl}(A)[1]$ and the quotient complex $\barC^\bullet(A,A_\Delta^\vee)/\barC_\lambda^\bullet(A)$.
As observed by Kontsevich--Soibelman, the latter complex can be related to the negative cyclic complex via the following short exact sequence.

\begin{prop}[\cite{KS08}\cite{AT22}]\label{secondsesforcyclic}
  There is a short exact sequence of chain complexes
  \begin{equation}\label{eq:secondses}
    \begin{tikzcd}
      0 \ar[r] & \barC^\bullet(A,A^\vee)/\barC^\bullet_\lambda(A) \ar[r,"\id-t"] & (\barB A/\ell)^\vee \ar[r,"N"] & \barC_\lambda^\bullet(A) \ar[r] & 0.
    \end{tikzcd}
  \end{equation}
  where the maps are the compositions of the norm operator $(N\upxi)_n = (1 + t + \cdots + t^{n})\upxi_n$ and the operator $\id - t$ via the identification
  \[
    (\barB A/\ell)^\vee \cong \prod_{n\geq 0} \hom_\elle(A[1]^{\otimes n+1}, \ell)[1] \cong \prod_{n\geq 0} \hom_\elle(A[1]^{\otimes n}, A^\vee) \cong \fM(A,A^\vee)[-1].
  \]
  The short exact sequence is moreover natural in $A\in\AinfAlg$ via the induced maps.
\end{prop}

It is well known that the normalised bar complex $\barB A/\ell$ is acyclic if $A$ is unital, hence so is its dual.
Choosing a contracting homotopy $s\colon (\barB A/\ell)^\vee[1] \to (\barB A/\ell)^\vee$, one obtains a quasi-isomorphism
\[
  \Omega^{2,\cl}(A)[2] \cong (\barC^\bullet(A,A_\Delta^\vee)/\barC_\lambda^\bullet(A))[1] \xrightarrow{N\circ s \circ (\id -t)} \barC_\lambda^\bullet(A),
\]
after composing with the inverse of $S$.
One can show that this map identifies right CY structures with SHIPs, leading to the following statement for unital $A_\infty$-algebras.
\begin{prop}[{\cite[Corollary 2.24]{AT22}}]
  The map $\Omega^{2,\cl}(A)[2] \xrightarrow{\sim} \barC_\lambda^\bullet(A)$ is a quasi-isomorphism, is natural with respect to $A_\infty$-morphisms, and identifies right CY structures with SHIPs.
\end{prop}

We now want to show a version of the above for analytic $A_\infty$-algebras, involving the analytic SHIPs and analytic right Calabi--Yau structures, by which we mean the following.

\begin{defn}
  Let $A\in\AinfAlgan$, then a right Calabi--Yau structure $\upphi$ is \emph{analytic} if it lies in the subcomplex $\barC_\lambda^{\an,\bullet}(A) \subset \barC_\lambda^{\bullet}(A)$.
\end{defn}

\begin{ex}\label{exboundedlin}
  Let $(A,D,\norm)$ be a normed DG algebra and $\uplambda \colon A^n \to \C$ a $\norm$-bounded linear functional such that
  \[
    \uplambda(Da) = 0,\quad \uplambda(a_1\cdot a_2) = (-1)^{|a_1||a_2|}\uplambda(a_2\cdot a_1).
  \]
  Then the element $\upphi \in \barC_\lambda^{-n,\an}(A)$ defined by $\upphi_0 = \uplambda$ and $\upphi_n = 0$ for $n>0$ is a cocycle, and it is a right $n$-Calabi--Yau structure if the cohomology pairing $([a_1],[a_2]) \mapsto \uplambda([a_1\cdot a_2])$ is nondegenerate.
\end{ex}

We will show that the map $\Omega^{2,\cl}(A)[2] \to \barC_\lambda^\bullet(A)$ restricts to a map on analytic elements by showing that the short exact sequences \labelcref{eq:firstSES} and \labelcref{eq:secondses} restrict to the relevant analytic subspaces.
The analytic version of \cref{sesforcyclic} is as follows.

\begin{lem}\label{analsplitepi}
  Let $A$ be an analytic $A_\infty$-algebra, then the map $S$ fits into a short-exact sequence
  \[
    \begin{tikzcd}
      0 \ar[r] & \barC^{\an,\bullet}_\lambda(A) \ar[r] & \barC^{\an,\bullet}(A,A_\Delta') \ar[r,"S"] & \Omega^{2,\cl,\an}(A)[1] \ar[r] & 0
    \end{tikzcd}
  \]
  which is moreover natural with respect to analytic $A_\infty$-morphisms $f\in\Hom_{\AinfAlg}^\an(A,B)$.
\end{lem}
\begin{proof}
  By construction $\barC^{\an,\bullet}_\lambda(A) = \barC^{\bullet}_\lambda(A) \cap \barC^{\an,\bullet}(A,A_\Delta')$ is a subcomplex of $\barC^{\an,\bullet}(A,A_\Delta')$, and this subcomplex is the kernel of the restriction of $S$ to $\barC^{\an,\bullet}(A,A_\Delta')$ by \cref{sesforcyclic}.
  Hence it suffices to show that $S$ restricts to a surjective map $S\colon\barC^{\an,\bullet}(A,A'_\Delta) \to \Omega^{2,\cl,\an}(A)[1]$.
  
  Suppose $\upphi \in \barC^{\bullet,an}(A,A_\Delta')$, and pick $C>0$ such that $\nrm{\upphi_n} < C^n$ for all $n\geq 1$.
  Then for any element $a \in A[1]^{\otimes i}\otimes A[1] \otimes A[1]^{\otimes j}$ and any decomposition $a=\sum_k \mathbf{x}_k \otimes \underline v_k \otimes \mathbf{y}_k$ 
  \[
    \begin{aligned}
      \nrm{(S\upphi)_{i,j}(a)} &\leq \sum_k \sup_{\nrm{w} = 1} \nrm{(S\upphi)_{i,j}(\mathbf{x}_k\otimes \underline v_k \otimes \mathbf{y}_k)(w)}
      \\&\leq \sum_k \sup_{\nrm{w}=1} \left(\nrm{\upphi_{i+j+1}(\mathbf{x}_k\otimes \underline v_k \otimes \mathbf{y}_k)(w)} + \nrm{\upphi_{i+j+1}(\mathbf{y}_k\otimes \underline w \otimes \mathbf{x}_k)(v_k)}\right)
      \\&\leq 2\nrm{\upphi_{i+j+1}}\cdot \sum_k \nrm{\mathbf{x}_k} \nrm{v_k} \nrm{\mathbf{y}_k}.
    \end{aligned}
  \]
  Taking the infimum over all decompositions of $a$, it follows that $\nrm{(S\upphi)_{i,j}(a)} \leq \nrm{\upphi_{i+j+1}} \cdot \nrm{a}$, so that $\nrm{(S\upphi)_{i,j}} \leq 2\nrm{\upphi_{i+j+1}} < (2C)^{i+j+1}$.
  It follows that by \cref{bianalbound} that $S\upphi$ is analytic.

  To show the surjectivity of $S$ we use the map\footnote{Note that this is not a chain map, hence does not split the exact sequence of chain complexes.} $h\colon \Omega^{2,\cl}(A)[1] \to \barC^\bullet(A,A^\vee)$ defined in \cite[Equation 11]{AT22}, which is defined by the equations
  \begin{equation}\label{eq:splittingmap}
    (h\uprho)(x_1\otimes\cdots\otimes x_n)(x_0) = \sum_{i=1}^n \frac{(-1)^\#}{n+1}\uprho(x_{i+1},\ldots,x_n,\underline x_0,\ldots,x_{i-1})(x_i),
  \end{equation}
  which satisfies $S \circ h = \id$.
  Now suppose $\uprho\in \Omega^{2,\cl,an}(A)$ is an analytic closed 2-form, and let $C>0$ be a constant such that $\nrm{\uprho_{i,j}} < C^{i+j+1}$ for all $i,j\in\N$.
  Then for any element $a \in A[1]^{\otimes n}$ and decomposition $a = \sum_k a_{k,1}\otimes\cdots\otimes a_{k,n}$ there is an inequality
  \[
    \begin{aligned}
      \nrm{(h\uprho)_n(a)} &\leq \sup_{\nrm{a_0}=1} \sum_k \nrm{(h\uprho)_n(a_{k,1}\otimes \cdots \otimes a_{k,n})(a_0)}
      \\&\leq
      \frac{1}{n+1} \cdot \sum_k\sum_{i=1}^n\sup_{\nrm{a_0}=1} \nrm{\uprho_{n-i,i-1}(a_{k,i+1},\ldots,a_{k,n},\underline a_0,a_{k,1},\ldots,a_{k,i-1})(a_{k,i})}
      \\&\leq
      \frac{1}{n+1} \cdot \sum_{i=1}^n\nrm{\uprho_{n-i,i-1}} \cdot \sum_k\nrm{a_{k,1}} \cdots \nrm{a_{k,n}}
      \\&< C^n \cdot \sum_k\nrm{a_{k,1}} \cdots \nrm{a_{k,n}}.
    \end{aligned}
  \]
  Taking the infimum over all decompositions of $a$, it follows that $\nrm{(h\uprho)_n(a)} \leq C^n\nrm{a}$, which shows that $h\uprho$ is analytic.
  Hence $h$ defines a map $\Omega^{2,\cl,\an}(A)[1] \to \barC^\bullet(A,A_\Delta')$ which satisfies $S\circ h = \id$, showing that $S$ is surjective.
  
  Finally, if $f\in\Hom_{\AinfAlg}^\an(A,B)$ is an analytic morphism, it follows from \cref{analpullback} and the subsequent discussion that the pullback morphisms $f^*$ preserve analytic elements, so naturality follows.
\end{proof}
\begin{cor}
  The map $S$ induces an isomorphism $\barC^{\an,\bullet}(A,A_\Delta')/ \barC^{\an,\bullet}_\lambda(A) \xrightarrow{\sim} \Omega^{2,\cl,\an}(A)[1]$.
\end{cor}

For the analytic analogue of \cref{secondsesforcyclic} we note that $(\barB A/\ell)^\vee \cong \fM(A,A^\vee)[-2]$, and define
\[
  (\barB A/\ell)^{\vee,\an} \colonequals (\barB A/\ell)^\vee \cap \fA(A,A')[-2],
\]
as the subspace obtained from $\fA(A,A')[-2]$ via this identification.
The following now holds.

\begin{lem}
  Let $A$ be an analytic $A_\infty$-algebra, then there is a short exact sequence
  \begin{equation}\label{eq:secondanalses}
    \begin{tikzcd}
      0 \ar[r] & \barC^{\an,\bullet}(A,A')/\barC_\lambda^{\an,\bullet}(A) \ar[r,"\id-t"] & (\barB A/\ell)^\an \ar[r,"N"] & \barC_\lambda^{\an,\bullet}(A) \ar[r] & 0.
    \end{tikzcd}
  \end{equation}
  which is natural with respect to analytic $A_\infty$-morphisms.
\end{lem}
\begin{proof}
  The operator $\id-t$ restricts to $\fA(A,A')[-1] \subset \fM(A,A^\vee)[-1]$ and clearly has the set of cyclic analytic elements as its kernel.
  Hence, it immediately follows that it defines an injective chain map
  \[
    \begin{tikzcd}
      \barC^{\an,\bullet}(A,A')/\barC_\lambda^{\an,\bullet}(A) \ar[r,"\id-t"] & (\barB A/\ell)^\vee \cap \fA(A,A')[-1] = (\barB A/\ell)^{\vee,\an}.
    \end{tikzcd}
  \]
  To see that the norm operator likewise preserves $\fA(A,A')$, note that if $\upphi$ satisfies $\nrm{\upphi_n} < C^n$ then
  \[
    \nrm{(N\upphi)_n} \leq \nrm{\upphi_n} + \cdots + \nrm{(t^n\upphi)_n} = (1+n) \nrm{\upphi_n} < (1+n) C^n < (3C)^n,
  \]
  so that $N\upphi$ is again analytic.
  Therefore it follows that $N$ defines a chain map
  \[
    \begin{tikzcd}
      (\barB A/\ell)^{\vee,\an} \ar[r,"N"] & \barC_\lambda^{\an,\bullet}(A).
    \end{tikzcd}
  \]
  It is moreover surjective, as every $\upphi\in \barC_\lambda^{\an,\bullet}(A)$ is the image of $\sum_n \tfrac1{n+1} \upphi_n \in \fA(A,A')[-1]$.
  To see that the sequence \labelcref{eq:secondanalses} is exact at the middle term, one can consider the inverse operator
  \[
    f\colon \ker N \to \barC^{\an,\bullet}(A,A'_\Delta)/\barC^{\an,\bullet}_\lambda(A),\quad
    (f\upxi)_n = \frac1{n+1}(t + 2 t^2+\cdots + n t^n)\upxi_n,
  \]
  which maps an analytic element with $\nrm{\upxi_n} <C^n$ for $n\geq 1$ to an element $f\upxi$ satisfying $\nrm{(f\upxi)_n} \leq \tfrac {n^2}{n+1} \nrm{\upxi_n} < (2C)^n$, which is therefore again analytic.
  It follows that \labelcref{eq:secondanalses} is a short exact sequence as required.
  Naturality again follows from the naturality in \labelcref{eq:secondanalses}
\end{proof}

To construct the map $\Omega^{2,\cl,\an}(A)[2] \to \barC_\lambda^\bullet(A)$ it now suffices to find a contracting homotopy for $(\barB A/\ell)^\vee$ which restricts to a contracting homotopy for $(\barB A/\ell)^{\vee,\an}$.
For this one can take e.g. the dual of the degeneracy map in \cite[Proposition-definition 1.1.12]{Lod92}, which is given by the formula
\[
  s\colon (\barB A/\ell)^\vee \to (\barB A/\ell)^\vee[-1],\quad (sg)_n(a_1,\ldots,a_n) = g_{n+1}(1,a_1,\ldots,a_n),
\]
where $1$ is any choice of unit in $A$.
The following lemma shows that this restricts to analytic elements.

\begin{lem}
  Let $A$ be a unital analytic $A_\infty$-algebra. Then the degeneracy map restricts to a contracting homotopy $s\colon (\barB A/\ell)^{\vee,\an} \to (\barB A/\ell)^{\vee,\an}[-1]$.
\end{lem}
\begin{proof}
  Let $f = (f_n)_{n\in\N} \in \fA(A,A')$ be an analytic sequence, and pick $C>0$ such that $\nrm{f_n} < C^n$ for all $n\geq n$.
  Then the corresponding element $g\in (\barB A/\ell)^\vee$ is the map defined by $g_n(a_1,\ldots,a_n) = f_{n-1}(a_1,\ldots,a_{n-1})(a_n)$, hence
  \[
    (sg)_n(a_1,\ldots,a_n) = g_{n+1}(1,a_0,\ldots,a_n) = f_n(1,a_1,\ldots,a_{n-1})(a_n),
  \]
  and the element $sf\in \fA(A,A')$ corresponding to $sg$ is given by $(sf)_n(a) = f_{n+1}(1\otimes a)$.
  Clearly this now satisfies the bound
  \[
    \nrm{(sf)_n(a)} = \nrm{f_{n+1}(1\otimes a)} \leq \nrm{f_{n+1}}\nrm{1}\nrm{a} < C^{n+1} \cdot \nrm{a} < (C^2)^{n+1} \cdot \nrm{a},
  \]
  hence $sf$ is analytic as claimed.  
\end{proof}

\begin{cor}\label{analCYlifts}
  For a unital analytic $A_\infty$-algebra $A$ there is a natural quasi-isomorphism
  \[
    \Omega^{2,\cl,\an}(A)[2] \cong (\barC^{\an,\bullet}(A,A_\Delta')/\barC_\lambda^{\an,\bullet}(A))[-1] \xrightarrow{N \circ s \circ (\id - t)} \barC_\lambda^{\an,\bullet}(A),
  \]
  which maps analytic SHIPs to analytic right CY structures.
\end{cor}

\begin{ex}\label{exSHIPfromlinform}
  Let $(A,\upmu)$ be the analytic $A_\infty$-algebra obtained from a normed DG algebra with $\upphi$ an analytic right CY structure obtained from a bounded linear functional $\uplambda \colon A^n \to \C$ as in \cref{exboundedlin}.
  Then $\upphi$ is the image of the analytic SHIP $\uprho \in \Omega^{2,\cl,\an}(A)$ defined by $\uprho_0(a)(b) = \uplambda(\upmu_2(a,b))$ and $\uprho_n = 0$ for $n>0$:
  one has $\uprho = S(\upxi)$ for the analytic cocycle $\upxi \in \barC^{\an,\bullet}(A,A'_\Delta)$ with $\upxi_1(a)(b) = \tfrac12 \uplambda(a\cdot b)$ and $\upxi_n = 0$ otherwise, and this cocycle maps to the Connes cocycle $(N\circ s\circ (\id - t) \upxi)$ with
  \[
    (N\circ s\circ (\id - t) \upxi)_0(a) = \upxi_1(1)(a) - (-1)^{|a|} \upxi_1(a)(1) = \tfrac12 \uplambda(\upmu_2(1,a) - (-1)^{|a|}\upmu_2(a,1)) = \uplambda(a)
  \]
  and vanishing higher terms.
  This example appears in the non-analytic case in \cite[Example 2.25]{AT22}.
\end{ex}

As result of the corollary, we find that every analytic SHIP is determined up to homotopy by an analytic right Calabi--Yau structure.
Using the corollary above, we can restate the analytic Darboux lemma in terms of Calabi--Yau structures, which yields our main theorem.

\begin{thm}\label{mainthm}
  Suppose $A \in\AinfAlgan$ is compact and admits a unital analytic minimal model. Then:
  \begin{enumerate}
  \item every analytic right $d$-CY structure $\upphi$ defines a cyclic analytic minimal model $(\hh(A)^\upphi,\upsigma^\upphi)$,
  \item if $\upphi,\uppsi$ are analytic right $d$-CY structures such that $[\uppsi] - [\upphi]$ maps to $0$ in $\HC_\lambda^{-d}(A)$, then there is a cyclic analytic $A_\infty$-isomorphism
    \[
      (\hh(A)^\upphi,\upsigma^\upphi) \cong_{\an,\cyc} (\hh(A)^\uppsi,\upsigma^\uppsi),
    \]
  \item if $B \in\AinfAlgan$ is also compact with a unital analytic minimal model, and $g\colon B \to A$ is an analytic quasi-isomorphism, then there is an induced a cyclic analytic $A_\infty$-isomorphism
    \[
      (\hh(B)^{g^*\upphi},\upsigma^{g^*\upphi}) \cong_{\an,\cyc} (\hh(A)^\upphi,\upsigma^\upphi).
    \]
  \end{enumerate}
  Hence, every nondegenerate class in $\HC_\lambda^{\an,-d}(A)$ defines a canonical cyclic analytic minimal model.
\end{thm}
\begin{proof}
  (1) Let $\hh(A)\in \AinfAlgan$ be a finite dimensional analytic minimal model of $A$ with analytic morphisms $I$ and $P$ as in \cref{minimalmodel}.
  Given an analytic right CY structure $\upphi\in\barC_\lambda^{\an,-d}(A)$ the induced map $A\to A^\vee[-d]$ induces an isomorphism $\hh(A) \to \hh(A)^\vee[-d]$, which coincides with the morphism induced by $I^*\upphi \in \barC_\lambda^{\an,\bullet}(\hh(A))$.
  Hence, $I^*\upphi$ is an analytic right CY structure on $\hh(A)$.
  
  Because $\hh(A)$ is unital, it follows by \cref{analCYlifts} that there exists an analytic strong homotopy inner product $\uprho^\upphi \in Z^{2-d}\Omega^{2,\cl,\an}(A)$ with image in $\barC_\lambda^{\an,-d}(A)$ homotopic to $I^*\upphi$.
  Therefore \cref{Darboux} determines a cyclic $A_\infty$-algebra $(\hh(A)^\upphi,\upsigma^\upphi) = (\hh(A)^{\uprho^\upphi},\upsigma^{\uprho^\upphi})$ and an analytic $A_\infty$-isomorphism $f\in \Hom_{\AinfAlg}^\an(\hh(A)^\upphi,\hh(A))$ such that $f^*\uprho^\upphi = \upsigma^\upphi$.
  Because $\hh(A)$ is finite dimensional, it follows by \cref{analinverse} the map $f$ is invertible in $\AinfAlgan$, and we obtain a diagram
  \[
    \begin{tikzcd}
      \hh(A)^\upphi \ar[r,shift left=3pt,"f"] & \ar[l,shift left=3pt,"f^{-1}"] \hh(A) \ar[r,shift left=3pt,"I"] & \ar[l,shift left=3pt,"P"] A,
    \end{tikzcd}
  \]
  which exhibits $(\hh(A)^\upphi,\upsigma^\upphi)$ as a cyclic analytic minimal model of $A$.

  (2) Let $\uppsi$ is a second right CY structure and pick a lift $\uprho^{\uppsi} \in \Omega^{2,\cl,\an}(\hh(A))$ for $I^*\uppsi$, yielding a cyclic minimal model $(\hh(A)^\uppsi,\upsigma^\uppsi)$ as above.
  Suppose $[\upphi] - [\uppsi]$ maps to $0$ in $\HC_\lambda^{-d}(A)$, then it follows by naturality that $[I^*\upphi] - [I^*\uppsi]$ maps to $0$ in $\HC_\lambda^{-d}(\hh(A))$ and considering the commutative diagram
  \[
    \begin{tikzcd}
      \hh^{2-d}\Omega^{2,\cl,\an}(\hh(A)) \ar[r]\ar[d] & \HC_\lambda^{\an,-d}(\hh(A))\ar[d] \\
      \hh^{2-d}\Omega^{2,\cl}(\hh(A)) \ar[r] & \HC_\lambda^{-d}(\hh(A)) 
    \end{tikzcd}
  \]
  it is also clear that $[\uprho^\upphi] - [\uprho^\uppsi]$ maps to $0$ in $\Omega^{2,\cl}(\hh(A))$.
  Hence it follows from \cref{Darboux} that there is a cyclic analytic $A_\infty$-isomorphism
  \[
    (\hh(A)^\uppsi,\upsigma^\uppsi) \cong_{\an,\cyc} (\hh(A)^\upphi,\upsigma^\upphi)
  \]
  between the minimal models as claimed.

  (3) Let $B\in \AinfAlgan$ admit a finite dimensional analytic unital minimal model $\hh(B)$ and suppose $g\colon B \to A$ is an analytic quasi-isomorphism, then $g^*\upphi \in \barC^{\an,\bullet}_\lambda(B)$ is again right CY and induces a cyclic minimal model $\hh(B)^{g^*\upphi}$.
  The map $g$ induces an analytic $A_\infty$-isomorphism $\hh(g)$ fitting into a commutative diagram
  \[
    \begin{tikzcd}[column sep=large]
      B \ar[r,"g"] & A & \\
      \hh(B) \ar[r,dashed,"\hh(g)"]\ar[u,"I"] &
      \hh(A) \ar[u,"I"]\ar[r,"f^{-1}",shift left=3pt] & \ar[l,"f",shift left=3pt] \hh(A)^{\upphi}
    \end{tikzcd}
  \]
  This implies that the equality $[I^*(g^*\upphi)] = [\hh(g)^*I^*\upphi]$ holds, and therefore $[\hh(g)^*\uprho^\upphi] = [\uprho^{g^*\upphi}]$ for the associated closed 2-forms.
  It therefore follows from \cref{Darboux} that the cyclic minimal model $(H^{g^*\upphi},\upsigma^{g^*\upphi})$ is cyclic-analytic $A_\infty$-isomorphic to the cyclic minimal model corresponding to $\hh(g)^*\uprho^\upphi$.
  The latter is given by $(\hh(A)^\upphi,\upsigma^\upphi)$, as the composition $\hh(g)^{-1} \diamond f$ satisfies
  \[
    (\hh(g)^{-1}\diamond f)^*(\hh(g)^*\uprho^\upphi) = f^*\uprho^\upphi = \upsigma^\upphi.\qedhere
  \]
\end{proof}

\begin{cor}\label{potentialsmaincor}
  If $A\in\AinfAlgan$ is compact with unital analytic minimal model with $\hh^0(A) \cong \ell$, then any analytic 3-CY structure $\upphi$ determines an analytic potential
  \[
    W_A^\upphi \in \widetilde\barT_\ell V
  \]
  in the analytic tensor algebra over $V = \hh^1(A)^*$, which only depends on $[\upphi] \in \HC_\lambda^{-3}(A)$ up to an change of coordinates $\widetilde\barT_\ell V \to \widetilde\barT_\ell V$.
  Moreover, if $g\colon B \to A$ is an analytic quasi-isomorphism for $B$ another such analytic $A_\infty$-algebra, then there is an isomorphism $h\colon \widetilde\barT_\ell W \to \widetilde\barT_\ell V$ from the tensor algebra over $W = \hh^1(B)$ such that
  \[
    h^*(W_B^{g^*\upphi}) = W_A^\upphi,
  \]
\end{cor}
\begin{proof}
  Letting $(\hh(A)^\upphi,\upsigma^\upphi)$ be the cyclic minimal model from \cref{mainthm}(1) associated to a 3-CY structure $\upphi$, it follows from \cref{analpotential} that there is an associated analytic potential $W_A^\upphi \in \widetilde\barT_\ell V$ for $V= ( (\hh(A))^1)^* = \hh^1(A)^*$.
  If $\uppsi$ is another 3-CY structure with $[\uppsi] = [\upphi]$ then it follows by \cref{mainthm}(2) that there is an isomorphism $f\in \Hom_{\AinfAlg}^\an(\hh(A)^\upphi,\hh(A)^\uppsi)$ which is cyclic.
  There is an induced algebra isomorphism $f^*\colon \widehat\barT_\ell V \to \widehat\barT_\ell V$, and it follows by \cite[Proposition 4.16]{Kaj07} that
  \[
    f^*(W_A^\uppsi) = W_A^\upphi.
  \]
  Hence it remains to verify that $f^*$ restricts to a map $\widetilde\barT_\ell V \to \widetilde\barT_\ell V$ of analytic tensor algebras.
  As in \cite[Lemma 3.12]{HK19} it suffices to show that $f^*(v) \in \widetilde\barT_\ell V$ for each $v \in V$.
  Picking a basis $v_1,\ldots,v_m \in V$ and letting $a_1,\ldots,a_m \in \hh^1(A)$ denote the dual basis, the element $f^*(v)$ can be written as the power series
  \[
    f^*(v) = \sum_{n\geq 0} \sum_{i_1,\ldots,i_n} v(f_n(a_{i_1},\ldots,a_{i_n})) \cdot v_{i_1}\otimes\cdots \otimes v_{i_n} \in \widehat\barT_\ell V.
  \]
  Because $f$ is analytic, it follows that there exists $C_0 > 0$ such that $\nrm{f_n} < C_0^n$ for all $n\in\N$, $\nrm{v} < C_0$, and $\nrm{a_i} < C_0$ for each $a_i$.
  Then for each $i_1,\ldots,i_n$ the coefficient satisfies:
  \[
    |v(f_n(a_{i_1},\ldots,a_{i_n}))| \leq \nrm{v}\nrm{f_n}\nrm{a_{i_1}}\cdots\nrm{a_{i_n}} < C_0^{2n+1} \leq (C_0^3)^n, 
  \]
  which shows that $f^*(v) \in \widetilde\barT_\ell V$.
  It follows that $W_A^\upphi$ and $W_A^\uppsi$ are related by an analytic change of coordinates.
  A similar argument now show the second statement.  
\end{proof}

\section{Analytic Calabi--Yau structures in geometry}\label{sec:geometry}

We will now apply our main theorem to a geometric setting: we will consider $A_\infty$-algebras which govern the deformations of sheaves on a smooth projective variety.
To set notation, we will write $X$ for a smooth projective variety of dimension $n$, which we interpret as an analytic space, and we will write $Y$ for arbitrary complex manifold.

The $A_\infty$-algebras arising in this setting are obtained from normed DG algebras, and the corresponding $A_\infty$-structures are therefore obtained as in \cref{exDGalg} using the canonical shift map $s\colon A \to A[1]$.

\subsection{The Dolbeault construction}

Recall that the sheaf of holomorphic functions $\O_Y$ on a complex manifold $Y$ admits a fine resolution via the \emph{Dolbeault complex}
\[
  \sA_Y^{0,0} \xrightarrow{\ \overline\partial\ } \sA_Y^{0,1} \xrightarrow{\ \overline\partial\ } \ldots \xrightarrow{\ \overline\partial\ } \sA_Y^{0,n} \longrightarrow 0,
\]
where $\sA^{p,q}_X$ is the sheaf of smooth $(p,q)$-forms on $Y$.
The sheaves $\sA^{p,q}_X$ are acyclic, hence can be used to compute the sheaf cohomology of $\O_Y$ via $\hh^i(U,\O_Y) \cong \hh^i(\Upgamma(U,\sA_Y^{0,\bullet}),\overline\partial)$ on any open $U\subset Y$.
More generally, one can take any a bounded complex of locally free $\O_Y$-modules $\cE = (\cE^\bullet,\updelta)$ on $Y$, i.e. perfect complex, and obtain an acyclic resolution of $\cE$ 
\[
  \sA_Y^{0,\bullet}(\cE) \colonequals \quad  \ldots \xrightarrow{\Dol} \bigoplus_{q\in\Z} \cE^{-q} \otimes_Y \sA_Y^{0,q} \xrightarrow{\Dol} \bigoplus_{q\in\Z} \cE^{1-q} \otimes_Y \sA_Y^{0,q} \xrightarrow{\Dol} \ldots,
\]
where the differential is given by $\Dol = \overline\partial + \updelta$.
This complex again computes the hypercohomology $\bR\Upgamma^i(U,\cE) \cong \hh^i\Upgamma(U,\sA_Y^{0,\bullet}(\cE))$ on any open $U\subset Y$.

The Dolbeault construction can be used to construct a DG enhancement of the derived category $\D^\perf(Y)$: for any perfect complexes $\cE,\cF$ there is a natural isomorphism
\[
  \hh^i\Upgamma(Y,\sA_Y^{0,\bullet}(\cH om(\cE,\cF))) \cong
  \bR^i\Upgamma(Y,\cH om_Y(\cE,\cF)) \cong
  \Hom_{\D^\perf(Y)}(\cE,\cF[i]),
\]
so the complex of global sections for $\sA_Y^{0,\bullet}(\cH om(\cE,\cF))$ is a morphism complex between $\cE$ and $\cF$.
The natural composition on $\D^\perf(Y)$ is given by a wedge product $\wedge$, which can be defined in terms of local sections $f \otimes \upxi \in  (\cH om^{i_1}(\cF,\cG) \otimes \sA_Y^{0,q})(U)$ and $g \otimes \upzeta \in (\cH om^{i_1}(\cE,\cF)\otimes \sA_Y^{0,q_2})(U)$ as
\[
  (f\otimes \upxi) \wedge (g\otimes \upzeta) \colonequals (-1)^{i_2q_1}(f\circ g) \otimes (\upxi \wedge \upzeta) \in (\cH om^{i_1+i_2}(\cE,\cG) \otimes \sA_Y^{0,q_1+q_2})(U).
\]
This product satisfies the Leibniz rule with respect to $\Dol$, and therefore defines a DG enhancement of $\D^\perf(Y)$ called the Dolbeault enhancement, see e.g. \cite[§\nobreak\,\nobreak2.3 Example 9]{Toe11}.

To each perfect complex $\cE \in \D^\perf(Y)$ is associated an endomorphism DGA: the \emph{Dolbeault DG algebra}
\[
  \frg_\cE \colonequals \left(\Upgamma(X,\sA_Y^{0,q}(\cE nd(\cE))),\Dol,\wedge\right),
\]
with cohomology $\Ext^\bullet(\cE,\cE)$.
If $\cE$ admits a direct sum decomposition $\cE = \cE_1^{\oplus m_1} \oplus \ldots \oplus \cE_k^{\oplus m_k}$ then the Dolbeault DG algebra is naturally defined over the base ring
\[
  \ell = \Mat_\C(m_1) \times \cdots \times \Mat_\C(m_k) \subset \frg_\cE^0.
\]
In what follows we will always assume such a splitting is given and work over the fixed base $\ell$.

\subsection{Analytic minimal models in the compact setting}

Now let $X$ be (the analytic space associated to) a smooth projective variety and fix a hermitian metric $\<-,-\>_X \colon \sA^{0,q}_X \otimes \sA^{0,q}_X \to \sA^{0,0}$.
For a perfect complex $\cE = (\cE^\bullet,\updelta)$ one can choose a compatible hermitian metric
\[
  \<-,-\>_\cE \colon \sA^{0,0}_X(\cE^i) \otimes \sA^{0,0}_X(\cE^i) \to \sA^{0,0}_X
\]
for each locally free sheaf $\cE^i$; we will refer to $\cE$ endowed which such a metric as a \emph{hermitian perfect complex}.
Recall (see e.g. \cite{GH78}) that this data determines an $L^2$-norm $\norm_{L^2} \colon \Upgamma(X,\sA_X^{0,\bullet}(\cE)) \to \R$, which is defined on pure tensors $\upxi = f \otimes \omega$ with $f\in \Upgamma(X,\sA^{0,0}_X(\cE^i))$ and $\omega \in \Upgamma(X,\sA^{0,q})$ by the formula
\[
  \nrm{\upxi}_{L^2}^2 \colonequals \int_X \<f,f\>_\cE \cdot \<\omega,\omega\>_X \,\d vol_X,
\]
where $\d vol_X \in \Upgamma(X,\sA^{0,n}_X(\Omega^n_X))$ is the canonical real volume form on $X$.
The \emph{Sobolev $(l,2)$-norm} for $l\in \N$ is defined in terms of the $L^2$-norm as $\nrm{\upxi}_{l,2} \colonequals \sum_{k\leq l} \nrm{\nabla^k\upxi}_{L^2}$, where the operator $\nabla$ acts on each locally free sheaf $\cE^i$ as the canonical metric connection.

Given hermitian perfect complexes $\cE$, $\cF$, there is an induced hermitian metric on $\cH om(\cE,\cF)$ and hence also induced norms $\norm_{L^2}$ and $\norm_{l,2}$ on the global sections of its Dolbeault construction.
As remarked by Fukaya \cite{Fuk01}, the operators $\Dol$ and $\wedge$ satisfy a bound
\[
  \nrm{\Dol}_{l,2} < C,\quad \nrm{\wedge}_{l,2} < C^2,\quad (C>0)
\]
if one chooses a sufficiently large $l > 2\dim_\C X$.
In particular, the Dolbeault DG algebra $\frg_\cE$ of a hermitian perfect complex is naturally an analytic $A_\infty$-algebra.
It was shown by Tu \cite{Tu14} and Toda \cite{Tod18} that this admits a strong analytic minimal model.

\begin{thm}[{\cite[Appendix A]{Tu14}\cite[Lemma 4.1]{Tod18}}]\label{minmodTT}
  The DG algebra $\frg_\cE$ admits a minimal model with diagram
  \[
    \begin{tikzcd}
      \cH_\cE \ar[r,"I",shift left=3pt] & \ar[l,"P",shift left=3pt] \frg_\cE \ar[loop right,"Q"]
    \end{tikzcd}
  \]
  such that the $A_\infty$-structure $\upmu$ on $\cH_\cE$ and the maps $I$, $P$, and $Q$ satisfy the bounds
  \[
    \nrm{\upmu_k}_{l,2} < C^k,\quad \nrm{I_k}_{l,2} < C^k,\quad \nrm{P_k}_{l,2} < C^k,\quad \nrm{Q_k}_{l,2} < C^k
  \]
  for all $k\geq 1$ and a fixed constant $C>0$ independent of $k$.
\end{thm}

At first order, the minimal model is given by Hodge theory: $\cH_\cE$ is the space of harmonic forms in $\frg_\cE$, with $I_1$ and $P_1$ being the inclusion and projection, and $Q_1$ is defined via the Green operator associated to $\Dol$.
The higher compositions are all defined via the standard homotopy transfer formula \cite{Kad80} given via sums of trees.

If $X$ is a Calabi--Yau variety, then Polishchuk \cite{Pol01} showed that the Serre duality pairing pulls back to a cyclic structure on this analytic minimal model. Using the terminology employed in this paper, this can be phrased as follows.

\begin{thm}[{\cite[Theorem 1.1]{Pol01}}]\label{policyc}
  Suppose $X$ is a projective Calabi--Yau variety with holomorphic volume form $\upnu \in \hh^0(X,\Omega^n_X)$, then the Serre pairing 
  \[
    \uprho(-)(-) = \int_X \upnu \wedge \tr(- \wedge -) 
  \]
  pulls back to a cyclic structure $\upsigma = I^*\uprho$ on $\cH_\cE$.
\end{thm}

The Serre pairing used in this theorem can be viewed as an analytic strong homotopy inner product $\uprho \in \Omega^{2,\cl,\an}(\frg_\cE)$, which corresponds to the right Calabi--Yau structure associated to the bounded linear functional
\[
  \uptau_\upnu = \int_X\upnu\wedge \tr(-) \in \frg_\cE' \in \bC^{-n}_\lambda(\frg_\cE),
\]
as in \cref{exSHIPfromlinform}.
In view of the analytic Darboux theorem, one might suspect that there exists other cyclic analytic minimal models for other choices of analytic right Calabi--Yau structure.

In what follows we show how such CY structures can be obtained from a choice of holomorphic volume form on open subsets $U\subset X$, using a compactly supported version of the Dolbeault construction.

\subsection{The compactly supported setting}

Let $U\subset X$ be an open analytic subvariety of a smooth projective variety $X$, viewed as an analytic manifold.
Given a bounded complex of locally free $\O_U$-modules $\cF$ we let
\[
  \frg_{\cF,c} \colonequals \left(\Upgamma_c(U,\sA_U^{0,\bullet}(\cE nd(\cF))),\ \Dol,\ \wedge\right)
\]
denote the nonunital DG subalgebra of $\frg_\cF$ of compactly supported sections.
Henceforth, we will assume that the cohomology of $\cF$ is supported on a compact subset $Z\subset U$, which implies that
\[
  \hh^\bullet\frg_{\cF,c} = \bR^i\Upgamma_c(U,\sA^{0,\bullet}_U(\cE nd(\cF))) \cong  \bR^i\Upgamma(U,\sA^{0,\bullet}_U(\cE nd(\cF))) = \hh^\bullet\frg_\cF,
\]
so that the inclusion $\frg_{\cF,c} \hookrightarrow \frg_\cF$ is a quasi-isomorphism.
We also fix a hermitian metric on $\cF$, which induces a hermitian pairing
\[
  \<-,-\>_{\cE nd(\cF)} \colon \sA_U^{0,0}(\cE nd(\cF)) \otimes \sA_U^{0,0}(\cE nd(\cF)) \to \sA_U^{0,0}
\]
The metric data again determines a well-defined $L^2$-norm $\norm_{L^2}\colon \frg_{\cF,c} \to \R$ which is defined on pure tensors $\upxi = f \otimes \omega$ with $f\in \Upgamma_c(U,\sA^{0,0}(\cE nd(\cF)))$ and $\omega\in\Upgamma_c(U,\sA^{0,q})$ as
\[
  \nrm{\upxi}_{L^2} = \left(\int_X \<f,f\>_{\cE nd(\cF)} \cdot \<\omega,\omega\>_X \d vol_X \right)^{1/2},
\]
where we note that the compactly supported function $\<f,f\>_{\cE nd(\cF)}$ can be viewed as a function on $X$ via extension by $0$, and similarly $\omega$ extends to a smooth $(0,q)$-form on $X$.
The metric connection then also determines a Sobolev norm $\norm_{l,2} = \sum_{k\leq l} \nrm{\nabla^k-}_{L^2}$ as in the compact setting.
A priori, the operators $\Dol$ and $\wedge$ are \emph{not} guaranteed to be bounded, however the reader can check that such a bound exists if $U$ is chosen to be a sufficiently small neighbourhood of the compact set $Z$.

Now suppose that $U$ admits a nowhere-vanishing holomorphic volume form $\upnu \in \Upgamma(U,\Omega^n_U)$.
Then writing again $\tr\colon \sA_U^{0,n}(\cE nd(\cE)) \to \sA_U^{0,n}$ for the point-wise trace, we obtain a linear functional
\[
  \uptau_\upnu \colon \frg_{\cF,c}^n \to \C,\quad \upzeta \mapsto \int_U \upnu \wedge \tr(\upzeta).
\]
We will view $\uptau_\upnu$ as an element of $\frg_{\cF,c}^\vee \subset \fM(\frg_{\cF,c},\frg_{\cF,c}^\vee)[-1]$ in the obvious way, and claim that this is a negative cyclic cocycle.

\begin{lem}\label{traceisCY}
  The map $\uptau_\upnu$ defines a right $n$-Calabi--Yau structure.
\end{lem}
\begin{proof}
  Unraveling the definition of the differential on $\bC^\bullet(\frg_{\cF,c},\frg_{\cF,c}^\vee) = (\fM(\frg_{\cF,c},\frg_{\cF,c}^\vee)[-1],b)$, we see that $b(\uptau_\upnu) = 0$ if and only if the following two equations hold for all homogeneous $\upzeta,\upxi\in\frg_{\cF,c}$:
  \[
    \begin{aligned}
      \uptau_\upnu(\Dol \upzeta) &= 0,\\
      \uptau_\upnu(\upzeta_1 \wedge \upzeta_2) &= (-1)^{|\upzeta_1||\upzeta_2|} \uptau_v(\upzeta_2\wedge \upzeta_1).
    \end{aligned}
  \]
  To show the first equation, we note that for a local form $\upzeta = f \otimes \omega$ with $f$ a local section of $\cE nd^i(\cF)$ and $\omega$ a local $(0,n-i-1)$-form, we have the local identity
  \[
    \tr(\Dol (f\otimes \omega)) = \tr(f)\overline\partial\omega + \tr([\updelta,f])\omega = \tr(f)\overline\partial\omega = \overline\partial\tr(f\otimes \omega).
  \]
  Hence, the same identity also hold for a general compactly supported section $\upzeta$ over $U$.
  Because $\upnu$ is holomorphic it then follows by Stokes' theorem for compactly supported forms that
  \[
    \uptau_\upnu(\Dol \upzeta) = \int_U \overline\partial(\upnu \wedge \tr(\upzeta))
    = \int_U (\partial + \overline\partial)(\upnu\wedge \tr(\upzeta)) = 0,
  \]
  where we note that $\partial(\upnu\wedge\tr(\upzeta)) = 0$ for degree reasons.
  For the second equation, we can likewise consider local forms $\upzeta_1 = f_1\otimes \omega_1$ and $\upzeta_2 = f_2\otimes \upxi_2$ with $f_k$ a local section of $\cE nd^{i_k}(\cF)$ and $\omega_k$ a local section of $\sA_U^{0,q_k}$, there is an identity
  \[
    \begin{aligned}
      \tr(\upzeta_1\wedge \upzeta_2)
      &= (-1)^{q_1 i_2} \tr(f_1\circ f_2) \cdot \omega_1 \wedge \omega_2
      \\&= (-1)^{q_1(q_2+i_2) + i_1i_2} \tr(f_2\circ f_1) \cdot \omega_2\wedge\omega_1
      = (-1)^{|\upzeta_1||\upzeta_2|}\tr(\upzeta_2\wedge\upzeta_1).
    \end{aligned}
  \]
  This identity then also holds for any compactly supported sections on $U$, so integrating against $\upnu$ yields the section equality.
  It follows that $\uptau_\upnu$ is a cocycle in $\barC^\bullet(\frg_{\cF,c},\frg_{\cF,c}^\vee)$, and it is immediate that it lies in $\barC^\bullet_\lambda(\frg_{\cF,c})$ because the cyclic action is trivial on the summand $\frg_{\cF,c}^\vee \subset \barC^\bullet_\lambda(\frg_{\cF,c})$.

  It remains to show that the map $\upzeta \mapsto \uptau_\upnu(\upzeta \wedge -)$ is a quasi-isomorphism, or equivalently that 
  \begin{equation}\label{pairingtoshow}
    (\upzeta,\upxi) \mapsto \uptau_\upnu(\upzeta\wedge\upxi) = \int_U \upnu \wedge \tr(\upzeta \wedge \upxi)
  \end{equation}
  induces a nondegenerate pairing on cohomology.
  For this, we note that a holomorphic volume form $\upnu \in \Upgamma(U,\Omega^n_U)$ represents an isomorphism $[\upnu] \in \Hom_{\D(U)}(\O_U,\Omega^n_U)$ in the derived category.
  Therefore, the wedge product
  \begin{equation}\label{eq:volumprod}
    \upnu\wedge- \colon \Upgamma_c(U,\sA_U^{0,\bullet}(\cE nd(\cF)))
    \to \Upgamma_c(U,\sA_U^{0,\bullet}(\cE nd(\cF) \otimes \Omega_U^n)),
  \end{equation}
  induces the isomorphisms $\Hom_{\D(U)}(\cE,\cE[i]) \to \Hom_{\D(U)}(\cE,\cE\otimes \Omega^n_U[i])$ in the derived category.
  By inspection, \labelcref{pairingtoshow} is the composition of \labelcref{eq:volumprod} with the Serre duality pairing
  \[
    \Upgamma_c(U,\sA_U^{0,n-i}(\cE nd(\cE))) \otimes \Upgamma_c(U,\sA_U^{0,i}(\cE nd(\cE) \otimes \Omega^n_U)) \xrightarrow{\int_U\tr(-\wedge-)} \C,
  \]
  which induces a nondegenerate pairing $\Hom_{\D(U)}(\cE,\cE[n-i]) \otimes \Hom_{\D(U)}(\cE,\cE \otimes \Omega^n[i]) \to \C$ on cohomology.
  Therefore the pairing induced by \labelcref{pairingtoshow} on cohomology is also nondegenerate, and it follows that $\uptau_\upnu$ is a right Calabi--Yau structure.
\end{proof}

We now want to show that the right Calabi--Yau structure $\uptau_\upnu$ is analytic if $\frg_{\cF,c}$ is an analytic $A_\infty$-algebra.
For this it suffices to show that $\uptau_\upnu$ is bounded with respect to the Sobolev norm.
\begin{lem}\label{volumeCYbounded}
  If $\upnu \in \Upgamma(U,\Omega^n_U)$ is an $L^2$-integrable holomorphic volume form, then $\uptau_\upnu$ is bounded.
\end{lem}
\begin{proof}
  For any $\upzeta \in \Upgamma_c(U,\sA^{0,\bullet}(\cE nd(\cF))$ there is a Cauchy-Schwarz type inequality
  \[
    |\tau_\upnu(\upzeta)| = \left|\int_U \upnu \wedge \tr(\upzeta)\right| \leq \nrm{\upnu}_{L^2} \nrm{\tr(\upzeta)}_{L^2},
  \]
  with respect to the $L^2$-inner product on smooth forms determined by the hermitian metric.
  It is then easy to check that $\nrm{\tr(\upzeta)}_{L^2} \leq \nrm{\upzeta}_{L^2}\leq \nrm{\upzeta}_{l,2}$, so that $\nrm{\uptau_\upnu}_{l,2}$ is bounded by $\nrm{\upnu}_{L^2}$.
\end{proof}

We again remark that one can ensure that any holomorphic volume form $\upnu\in\Upgamma(U,\Omega^n_U)$ is bounded in the $L^2$-norm by replacing $U$ with a sufficiently small neighbourhood of $Z$.
After making this slight modification, we find an analytic right Calabi--Yau structure corresponding to the volume form.

\subsection{Comparing the compact and noncompact settings}
Let $\cE$ be a hermitian perfect complex on a projective variety $X$, with cohomology supported on a subset $Z$, and choose an open neighbourhood $U\supset Z$ with inclusion map $i\colon U \hookrightarrow X$.
Then the extension by zero yields an injective map
\[
  i_!\colon \frg_{\cE|_U,c} = \Upgamma_c(U,\sA^{0,\bullet}_U(\cE nd(\cE|_U))) \longrightarrow \Upgamma(X,\sA^{0,\bullet}_X(\cE nd(\cE))) = \frg_\cE,
\]
which exhibits $\frg_{\cE|_U,c}$ as a DG ideal of $\frg_\cE$.
Equipping $\cE|_U$ with the induced metric, this is an isometry with respect to the Sobolev norm, making $\frg_{\cE|_U,c}$ into an analytic $A_\infty$-subalgebra of $\frg_\cE$.
The assumption on the support moreover guarantees that $i_!$ is a quasi-isomorphism.

In order to pull back the analytic right CY structures found above from $\frg_{\cE|_U,c}$ to $\frg_\cE$ we would like to construct a homotopy inverse for $i_!$ using \cref{existsminmod}.
This requires us to first construct a quasi-isomorphism $\frg_\cE \to \frg_{\cE|_U,c}$, for which we use a choice of weak unit.

\begin{lem}\label{unithompair}
  There exists a pair $(u,h) \in \frg_{\cE|_U,c}^0 \times \frg_{\cE}^{-1}$ such that $\Dol u = 0$ and
  \[
    \Dol h = i_!u - \id_\cE \in \frg_\cE^0.
  \]
  In particular, the map $\one \colon l \mapsto l\cdot u$ is a weak unit for the analytic $A_\infty$-algebra $\frg_{\cE|_U,c}$.
\end{lem}
\begin{proof}
  Because $i_!\colon \frg_{\cE|_U,c} \to\frg_\cE$ is a quasi-isomorphism, it follows that there exists a cocycle $u$ with $[i_!u] = [\id_\cE]$, and hence there exists a coboundary $\Dol h = i_!u-\id_\cE$.
  If $\upxi \in \frg_{\cE|_U,c}$ is $\Dol$-closed 
  \[
    [i_!(u \wedge \upxi)] = [\id_\cE \wedge i_!\upxi] + [\Dol h \wedge i_!\upxi] = [i_!\upxi] \in \hh^\bullet\frg_\cE.
  \]
  It follows that $[u\wedge\upxi] = [\upxi]$ and likewise $[\upxi \wedge u] = [\upxi]$ hold in $\hh^\bullet\frg_{\cE|_U,c}$. The map $\one \colon \ell \to \frg_{\cE|_U,c}$ then satisfies $[\one(l)\wedge \upxi] = [l\cdot u\wedge \upxi]= [l\cdot \upxi]$ and $[\upxi \wedge \one(l)] = [\upxi \cdot l \wedge u] = [\upxi\cdot l]$, making it a weak unit.
\end{proof}

For a unit/coboundary pair $(u,h)$ as above, the wedge product $i_!u\wedge \upxi$ with any section $\upxi\in \frg_{\cE}$ has compact support contained in $U$.
Hence, we obtain a restriction $i^*(i_!u \wedge -) \colon \frg_\cE \to \frg_{\cE|_U,c}$ which is a chain-level inverse to $i_!$.
To extend it to an analytic $A_\infty$-morphism, we need the following lemma.

\begin{lem}\label{qisotheotherway}
  Let $h\in \frg_\cE^{-1}$ and $v = \id_\cE + \Dol h$, then the maps $K_n \colon \frg_\cE[1]^{\otimes n} \to \frg_{\cE}[1]$ given by
  \[
    K_n(s\xi_1,\ldots,s\xi_n) = s(v\wedge\xi_1\wedge h\wedge \cdots \wedge h\wedge \xi_n),
  \]
  for $\xi_i \in \frg_\cE$, define an analytic $A_\infty$-morphism $K \in \Hom_{\AinfAlg}^\an(\frg_\cE,\frg_\cE)$.
\end{lem}
\begin{proof}
  Since the $A_\infty$-algebra structures on $\frg_\cE$ and $\frg_{\cE|_U,c}$ come from DG algebra structures, the $A_\infty$-morphism conditions can be written as
  \begin{equation}\label{Kmorphcond}
    \begin{aligned}
      \mu_1 K_n + \sum_{i=1}^{n-1} \mu_2 (K_i \otimes K_{n-i})
      &= \sum_{i= 0}^{n-1}  K_n (\id^{\otimes i} \otimes \mu_1 \otimes \id^{\otimes n-i-1})
        \\&\quad+ \sum_{i=0}^{n-2} K_{n-1} (\id^{\otimes i} \otimes \mu_2 \otimes \id^{\otimes n-i-2}).
    \end{aligned}
  \end{equation}
  Plugging in the definitions $\mu_1 = -s \Dol  s^{-1}$ and $\mu_2 = -s(-\wedge -)(s^{-1})^{\otimes 2}$, it follows from the graded Leibniz rule for $\Dol$ and $\wedge$ that
  \[
    \begin{aligned}
      \mu_1 K_1(s\xi_1,\ldots,s\xi_n) &= -s\Dol(v\wedge \upxi_1\wedge h \wedge \cdots \wedge \upxi_n)
      \\&= \sum_{i=0}^{n-1} (-1)^{\epsilon_{i+1}} s(v \wedge \cdots \wedge h \wedge \Dol \upxi_{i+1} \wedge h \wedge \cdots \wedge \upxi_n)
      \\&\quad- \sum_{i=1}^{n-1} (-1)^{\epsilon_{i+1}} s(v \wedge \cdots \wedge h \wedge \upxi_i \wedge v \wedge \upxi_{i+1} \wedge h \wedge \cdots \wedge \upxi_n),
      \\&\quad+ \sum_{i=1}^{n-1} (-1)^{\epsilon_{i+1}} s(v \wedge \cdots \wedge h \wedge \upxi_i \wedge \upxi_{i+1} \wedge h \wedge \cdots \wedge \upxi_n).
    \end{aligned}
  \]
  where we used the notation ${\epsilon_i} = |\xi_1|+\cdots+|\xi_i|+i$ for the sign.
  By inspection, each of the terms in these summations is given by one of the terms
  \[
    \begin{aligned}
      \mu_2(K_i \otimes K_{n-i})(s\xi)
      &= (-1)^{\epsilon_{i+1}} s(v\wedge \xi_1\wedge \cdots\wedge \xi_i \wedge v\wedge \xi_{i+1}\wedge\cdots\wedge h\wedge\xi_n).\\
      K_n(\id^{\otimes i}\otimes \mu_1 \otimes \id^{\otimes n-i-1})(s\xi) &=
      (-1)^{{\epsilon_{i+1}}} s(v\wedge \xi_1 \wedge \ldots \wedge h\wedge \Dol \xi_{i+1} \wedge h \wedge \cdots \wedge \xi_n)\\
      K_{n-1}(\id^{\otimes i-1}\otimes \mu_2\otimes \id^{\otimes n-i-1})(s\xi) &=
      (-1)^{{\epsilon_{i+1}}} s(v\wedge \xi_1 \wedge \ldots \wedge h\wedge \xi_i \wedge \xi_{i+1} \wedge h \wedge \cdots \wedge \xi_n),
    \end{aligned}
  \]
  where we've abbreviated $s\xi = s\xi_1 \otimes \cdots \otimes s\xi_n$.
  Comparing coefficients, it therefore follows that $K$ defines an $A_\infty$-morphism, and it remains to show it is analytic.
  For this let $C\geq 1$ be any constant so that $\nrm{\wedge}_{l,2} \leq C$, $\nrm{u}_{l,2} < C$ and $\nrm{h}_{l,2} < C$.
  Then for any sections $\upxi_1,\ldots,\upxi_n$ 
  \[
    \nrm{K_n(s\upxi_1,\ldots,s\upxi_n)}_{l,2} \leq C^{3n-1} \cdot \nrm{\upxi_1}_{l,2}\cdots \nrm{\upxi_n}_{l,2},
  \]
  which implies that $\nrm{K_n} < C^{3n-1} < (C^3)^n$, which shows that $K$ is indeed analytic.
\end{proof}

Given a unit/coboundary pair $(u,h)$ as above, $\id_\cE+\Dol h = i_!u$ has support contained in the open $U\subset X$, so for any $\xi_1,\ldots,\xi_n \in \frg_\cE$ there is a well-defined pullback
\[
  K_{c,n}(s\xi_1,\ldots,s\xi_n) = i^*(i_!u \wedge \xi_1 \wedge \cdots \wedge \xi_n) \in \frg_{\cE|_U,c}.
\]
We therefore obtain a well-defined $A_\infty$-morphism $K_c$ from $\frg_\cE$ to $\frg_{\cE|_U,c}$.

\begin{lem}
  Let $(u,h)$ be a unit/coboundary pair as in \cref{unithompair} and $K$ the morphism associated to $h$.
  Then $K_c$ is an analytic homotopy inverse to $i_!$.
\end{lem}
\begin{proof}
  We will construct an analytic homotopy $K^t + Q^t\d t \in \Hom^\an_\AinfAlg(\frg_\cE,\Omega^\bullet_{[0,1]} \otimes \frg_\cE)$ between $\id_{\frg_\cE}$ and $K$,
  where the coefficient functions are defined as follows:
  \[
    \begin{aligned}
      K^t_n(s\xi_1,\ldots,s\xi_n) &= s(\id_\cE + t\cdot \Dol h) \wedge \xi_1 \wedge (t\cdot h) \wedge \cdots \wedge (t\cdot h) \wedge \xi_n),\\
      Q^t_n(s\xi_1,\ldots,s\xi_n) &= s(h\wedge \xi_1 \wedge (t\cdot h) \wedge \cdots \wedge (t\cdot h) \wedge \xi_n).
    \end{aligned}
  \]
  For each $t\in [0,1]$ the map $K^t$ is the $A_\infty$-morphism of \cref{qisotheotherway} corresponding to the element $t\cdot h\in\frg_\cE^{-1}$, so it suffices to check the condition~\eqref{eq:homotopycondtwo}.
  The two terms on the right hand side of~\eqref{eq:homotopycondtwo} are respectively given by the sum of the terms
  \[
    \begin{aligned}
      \mu_1Q^t_n(s\xi_1,\ldots,s\xi_n)
      &= \sum_{i=1}^n (-1)^{\epsilon_{i-1}} t^{n-1} \cdot s(h\wedge \xi_1 \wedge \cdots \wedge \Dol  h \wedge \xi_i \wedge \cdots \wedge h \wedge \xi_n) \\
      &\quad-\sum_{i=1}^n (-1)^{\epsilon_{i-1}} t^{n-1} \cdot s(h\wedge \xi_1 \wedge \cdots \wedge h \wedge \Dol \xi_i \wedge \cdots \wedge h \wedge \xi_n)
      \\
      \mu_2(K^t_i \otimes Q^t_{n-i})(s\xi_1,\ldots,s\xi_n)
      &= (n-1) t^{n-1} \cdot s(\Dol h\wedge \xi_1 \wedge \cdots \wedge h \wedge \xi_n)
      \\&\quad+ (n-1) t^{n-2} \cdot s(\xi_1 \wedge h \wedge \cdots \wedge h \wedge \xi_n)
      \\
      \mu_2(Q^t_i \otimes K^t_{n-i})(s\xi_1,\ldots,s\xi_n)
      &= -(-1)^{\epsilon_{i+1} }t^{n-1} \cdot s(h\wedge \xi_1 \wedge \cdots \wedge \Dol  h \wedge \xi_{i+1} \wedge \cdots \wedge h \wedge \xi_n)
      \\&\quad-(-1)^{\epsilon_{i+1}} t^{n-2}\cdot s(h\wedge \xi_1 \wedge \cdots \wedge h \wedge \xi_{i} \wedge \xi_{i+1} \wedge h \wedge \cdots \wedge h \wedge \xi_n)
    \end{aligned}
  \]
  and the term
  \[
    \begin{aligned}
      (Q^t_{n-1} \widetilde \upmu)(s\xi_1,\ldots,s\xi_n) &= \sum_{i=1}^n (-1)^{\epsilon_{i-1}}t^{n-1}\cdot s(h\wedge \xi_1 \wedge \cdots \wedge h \wedge \Dol \xi_i \wedge \cdots \wedge h \wedge \xi_n)\\
      &\quad+ \sum_{i=1}^{n-1}(-1)^{\epsilon_{i+1}} t^{n-2} \cdot s(h\wedge \xi_1 \wedge \cdots \wedge h \wedge \xi_i \wedge \xi_{i+1} \wedge h \wedge \cdots \wedge h \wedge \xi_n).
    \end{aligned}
  \]
  The sum of all these terms is equal to
  \[
    nt^{n-1}\cdot s(\Dol  h \wedge \xi_1 \wedge h \wedge \ldots \wedge h \wedge \xi_n) + (n-1)t^{n-2} \cdot s(\xi_1 \wedge h \wedge \cdots \wedge h \wedge \xi_n ) = \frac{\partial}{\partial s} K^t_n(s\xi_1,\ldots,s\xi_n),
  \]
  which shows that $K^t+Q^t \d t$ defines a homotopy from $K^0 = \id$ to $K^1 = K$ on $\frg_\cE$.
  To see that this homotopy is analytic, take again a constant $C\geq 1$ such that $\nrm{\wedge}_{l,2} \leq C$, $\nrm{\Dol }_{l,2}<C$ and $\nrm{h}_{l,2} < C$.
  Then the component functions in the homotopy are bounded by
  \[
    \begin{aligned}
      \nrm{K_n^t}_\infty &\leq \int_0^1 t^{n-1} \nrm{\wedge}_{l,2}^{4n}\cdot \nrm{h}_{l,2}^{2n-2}\left(1+t \nrm{\Dol }_{l,2}^2\cdot \nrm{h}_{l,2}^{2}\right) \d t
      < 2 C^{6n+2} ,
      \\
      \nrm{\tfrac{\partial}{\partial t} K_n^t}_\infty& \leq \int_0^1 t^{n-2}\nrm{\wedge}_{l,2}^{4n} \cdot  \nrm{h}_{l,2}^{2n-2}\left(n t \nrm{\Dol }_{l,2}^2 \cdot \nrm{h}_{l,2}^{2} + (n-1)\right) \d t
                                                < 2n C^{6n+2},
      \\
      \nrm{Q_n^t}_\infty &\leq \int_0^1 t^{n-1} \nrm{\wedge}_{l,2}^{4n} \cdot \nrm{h}_{l,2}^{2n}\,\d t
                    < C^{6n}.
    \end{aligned}
  \]
  Clearly, we then have a common bound (e.g. $C' = 2(C+1)^8$ suffices) and it follows that $K^t + Q^t\d t$ is an analytic homotopy, which shows that $i_! \diamond K_c = K \sim_\an \id_{\frg_\cE}$.
  The forms $Q^t_n(s\xi_1,\ldots,s\xi_n)$ and $K^t_n(s\xi_1,\ldots,s\xi_n)$ have compact support on $U$
  whenever at least on of the $\xi_i$ does.
  Writing $K_c^t$ and $Q_c^t$ to denote the composition with the restriction as before, we obtain an analytic homotopy
  \[
    K_c^t \diamond i_! + Q_c^t \diamond i_! \in \Hom_\AinfAlg(\frg_{\cE|_U,c},\Omega^\bullet_{[0,1]}\otimes \frg_{\cE|_U,c}),
  \]
  between $\id_{\frg_{\cE|_U,c}}$ and $K_c \diamond i_!$, which shows the result.
\end{proof}

We conclude that $\frg_{\cE|_U,c}$ is analytically homotopy equivalent to $\frg_\cE$, and therefore also analytically homotopy equivalent to the minimal model $\cH_\cE$.
\Cref{existsminmod} then shows that there exists some automorphism $T\in \Hom_\AinfAlg^\an(\frg_\cE,\frg_\cE)$ such that the compositions in the diagram
\begin{equation}\label{eq:minmodcs}
  \begin{tikzcd}[column sep=large]
    \cH_\cE \ar[r,"I",shift left=3pt] & \ar[l,"P",shift left=3pt] \frg_\cE \ar[r,"K_c",shift left=3pt] & \ar[l,"T\diamond i_!",shift left=3pt] \frg_{\cE|_U,c}.
  \end{tikzcd}    
\end{equation}
make $\cH_E$ into an analytic minimal model for $\frg_{\cE|_U,c}$.
Hence, $\frg_{\cE|_U,c}$ is again one of the well-behaved analytic $A_\infty$-algebras, admitting a compact analytic minimal model.

Finally, we wish to compare the situation where $\cE|_U$ is replaced by an arbitrary hermitian perfect complex on $U$.
It turns out that we can again compare the corresponding Dolbeault DG algebras via a quasi-isomorphism, as the following lemma shows.

\begin{lem}\label{arbcompqiso}
  Suppose $\cF$ is a hermitian perfect complex quasi-isomorphic to $\cE|_U$, such that $\frg_{\cF,c}$ is analytic with respect to the Sobolev norm.
  Then there exists a triple $(r,r^{-1},h_{rr^{-1}})$ with
  \[
    r \in \Upgamma_c(U,\sA_X^{0,\bullet}(\cH om(\cF,\cE|_U))),\quad r^{-1} \in \Upgamma_c(U,\sA_X^{0,\bullet}(\cH om(\cE|_U,\cF)))
  \]
  degree $0$ cocycles and $h_{rr^{-1}} \in \frg_\cE^{-1}$ satisfying $\Dol h_{rr^{-1}} = i_!(r\wedge r^{-1}) - \id_\cE$.
  Such that the maps
  \[
    F_n(s\upxi_1,\ldots,s\upxi_n) \colonequals s(r^{-1} \wedge \upxi_1 \wedge h_{rr^{-1}} \wedge \cdots \wedge h_{rr^{-1}} \wedge \upxi_n \wedge r)
  \]
  determine an analytic $A_\infty$-quasi-isomorphism $F\in \Hom_{\AinfAlg}(\frg_{\cE|_U,c},\frg_{\cF,c})$.
\end{lem}
\begin{proof}
  Because the cohomologies of $\cF$ and $\cE|_U$ are compactly supported on $U$, the morphism spaces between $\cE|_U$ and $\cF$ in the derived category can be computed as
  \[
    \Hom_{\D(U)}(\cE|_U,\cF) \cong \bR^0\Upgamma_c(U,\cH om(\cE|_U,\cF)) \cong \hh^0(\Upgamma_c(U,\sA_U^{0,\bullet}(\cH om(\cE|_U,\cF))),\Dol),
  \]
  and similarly $\Hom_{\D(U)}(\cF,\cE|_U) \cong \hh^0(\Upgamma_c(U,\sA_U^{0,\bullet}(\cH om(\cE|_U,\cF)),\Dol)$.
  Hence, there exists degree $0$ cocycles $r^{-1} \in \Upgamma_c(U,\sA_U^{0,\bullet}(\cH om(\cE|_U,\cF)))$ and $r \in \Upgamma_c(U,\sA_U^{0,\bullet}(\cH om(\cE|_U,\cF)))$ which induces the isomorphism between $\cF$ and $\cE|_U$ in the derived category.
  Because the composition $r\wedge r^{-1}$ induces the identity on $\cE|_U$ in the derived category, there exists some $h_{rr^{-1}}\in \frg_\cE^{-1}$ such that
  \[
    i_!(r\wedge r^{-1}) = \id_{\cE} + \Dol h_{rr^{-1}}.
  \]
  Hence, $F$ is well-defined and one can check easily that it is an $A_\infty$-morphism in an analogous fashion to the proof of \cref{qisotheotherway}.
  To see that it is a quasi-isomorphism we note that that $F_1(s\xi) = s(r^{-1}\wedge \xi \wedge r)$ has a quasi-inverse $F_1^{-1}(s\xi) \colonequals s(r \wedge \xi \wedge r^{-1})$.
  Indeed, we have
  \[
    \begin{aligned}
      [i_!(F_1^{-1}(F_1(s\upxi)))] &= [s(i_!(r\wedge r^{-1}) \wedge i_!\upxi \wedge i_!(r \wedge r^{-1}))]
      \\&= [s((\id_\cE + \Dol h_{rr^{-1}}) \wedge i_!\upxi \wedge (\id_\cE+\Dol h_{rr^{-1}}))]
      \\&= [si_!\upxi] \in \hh^0\frg_\cE.
    \end{aligned}
  \]
  which implies that $F_1^{-1}\circ F_1$ induces the identity on $\hh^\bullet\frg_{\cE|_U,c}$, since $i_!\colon \frg_{\cE|_U,c}\to \frg_\cE$ is a quasi-isomorphism.
  To check that $F$ is analytic let $C> \max\{\nrm{r^{-1}}_{l,2},\nrm{r}_{l,2},\nrm{h}_{l,2},\nrm{\wedge}_{l,2}\}$, so that
  \[
    \nrm{F}_{l,2} \leq \nrm{\wedge}_{l,2}^{2n}\cdot \nrm{r}_{l,2}\cdot \nrm{r^{-1}}_{l,2}\cdot \nrm{h_{rr^{-1}}}_{l,2}^{n-1} \leq C^{3n+1} \leq C^{4n},
  \]
  as in the proof of \cref{qisotheotherway}. It follows that $F$ is analytic when $\frg_{\cF,c}$ is an analytic $A_\infty$-algebra.
\end{proof}

If $\cF$ is an arbitrary perfect complex on $U$ with compactly supported cohomology then $\cF$ is quasi-isomorphic to $i^*\cE = \cE|_U$, where $\cE \to \bR i_! \cF$ is any resolution of the exceptional direct image.
With the above lemma, we find maps on analytic negative cyclic cohomology
\[
  \begin{tikzcd}[column sep=large]
    \HC_\lambda^\bullet(\frg_{\cF,c}) \ar[r,"F^*"] &
    \HC_\lambda^\bullet(\frg_{\cE|_U,c}) \arrow{r}{K_c^*}[swap]{\sim} &
    \HC_\lambda^\bullet(\frg_\cE)  \arrow{r}{I^*}[swap]{\sim}&
    \HC_\lambda^\bullet(\cH_\cE),
  \end{tikzcd}    
\]
along which right Calabi--Yau structures can be pulled back.
With these comparison maps, we are now ready to prove the main theorems of the paper.

\subsection{Cyclic minimal models from local volumes}

We now combine all results to obtain new cyclic minimal models associated to local holomorphic volume forms.
It turns out that the cyclic minimal model one obtains does not depend on the neighbourhood chosen, so we can take an agnostic approach: given a closed subset $Z$ we consider the space of germs along $Z$ 
\[
  (\Omega^n_X)_Z \colonequals \{ \upnu \in \Upgamma(V,\Omega^n_X) \mid V\supset Z \text{ open in } X\}/\sim
\]
where $\upnu \sim \upnu'$ for two sections $\upnu \in \Upgamma(V,\Omega^n_X)$ and $\upnu' \in \Upgamma(V',\Omega^n_X)$ if there exists $V'' \subset V\cap V'$ such that $Z \subset V''$ and $\upnu|_{V''} = \upnu'|_{V''}$.
We consider germs which locally act as a volume.

\begin{defn}
  An element $\upnu \in (\Omega^n_X)_Z$ is a \emph{volume germ} along $Z$ if there is some open neighbourhood $V\supset Z$ such that $\upnu|_V$ is nonzero at every point in $V$.
\end{defn}

The main result of this section is the following theorem.

\begin{thm}\label{cycminmodcan}
  Let $X$ be a smooth projective variety of dimension $n$, and $\cE$ a perfect complex on $X$ with cohomology supported on $Z \subset X$.
  Then any volume germ $\upnu \in (\Omega^n_X)_Z$ determines a canonical class $[\upphi^\upnu]_\an \in \HC_\lambda^{\bullet,\an}(\frg_\cE)$ of an analytic right Calabi--Yau structure,
  which determines a corresponding \emph{analytic cyclic minimal model} of $\frg_\cE$
  \[
    (\cH_\cE^\upnu,\upsigma^\upnu)  = (\Ext_X^\bullet(\cE,\cE),\upmu^\nu,\upsigma^\nu),
  \]
  which is well-defined up to cyclic analytic $A_\infty$-isomorphism.
\end{thm}
\begin{proof}
  We fix a hermitian structure on $X$ and $\cE$ as before, so that $\frg_\cE$ is an analytic DG algebra which admits a strong analytic minimal model by \cref{minmodTT} with underlying space
  \[
    \cH_\cE \cong \hh^\bullet\frg_\cE \cong \Ext^\bullet_X(\cE,\cE).
  \]
  We claim that every volume germ $\upnu$ along $Z$ has a representative which is $L^2$-integrable.
  To construct this, pick any nonvanishing representative $\upnu^0 \in \Upgamma(V,\Omega^n_X)$, and note that the pointwise squared-norm is a continuous function $p \mapsto |\upnu^0(p)|_X^2$ on $V$.
  Because $Z$ is compact, this function is bounded on $Z$, so picking any $\epsilon>0$ we obtain an open neighbourhood
  \[
    U = \left\{\ p\in Y \ \middle|\ |\upnu_p^0|_X^2 < \epsilon + \max_{q\in Z} |\upnu_X^0(q)|^2\ \right\},
  \]
  on which $|\upnu_p^0|^2$ is bounded.
  Then $\upnu = \upnu^0|_U$ is square-integrable, yielding the required representative.
 
  Fixing such a $\upnu \in \Upgamma(U,\Omega^n_X)$, the trace $\uptau_\upnu \in \hom_\elle^\cont(\ell,\frg_{\cE|_U,c}')$ from \cref{traceisCY} is an analytic right CY structure on $\frg_{\cE|_U,c}$ by \cref{volumeCYbounded}.
  Writing $i\colon U\to X$ for the embedding of $U$, it follows by \cref{unithompair} that there is a unit/coboundary pair $(u,h) \in \frg_{\cE|_U,c}^0\times \frg_\cE^{-1}$, which induces an analytic homotopy equivalence
  $K_c\in \Hom_{\AinfAlg}^\an(\frg_\cE,\frg_{\cE|_U,c})$ by \cref{qisotheotherway}.
  We obtain the analytic right CY structure
  \[
    \upphi^\upnu \colonequals K_c^*\uptau_\upnu \in \barC^{\an,\bullet}(\frg_\cE).
  \]
  Because the minimal model $\cH_\cE$ is unital, it follows from \cref{mainthm}(1) that $\upphi^\upnu$ determines an analytic cyclic minimal model
  \[
    (\cH_\cE^\upnu\kern-.5pt, \upsigma^\upnu) \colonequals (\cH_\cE^{\upphi^\upnu}\kern-5pt, \upsigma^{\upphi^\upnu}),
  \]
  so that $\upphi^\upnu$ pulls back to a class in $\barC_\lambda^{\bullet,\an}(\cH_\cE^\upnu)$ which is homotopic to the image of $\upsigma^\upnu$ along the map in \cref{analCYlifts}.
  It now remains to show that the right CY structure and the cyclic minimal model are independent of the choice of neighbourhood $U$ or the unit/coboundary pair $(u,h)$.
  
  Suppose $V\subset U$ is another open subset with inclusion $j\colon V\hookrightarrow U$, and let $(u_V,h_V)$ be a choice of unit/coboundary pair such that $(ij)_!u_V = \id_\cE + \Dol h_V$, and let $K_{c,V}\in \Hom_{\AinfAlg}^\an(\frg_\cE,\frg_{\cE|_V,c})$ be the induced analytic homotopy equivalence.
  Then $\upnu|_V \sim \upnu$ is another representative of the germ and induces the analytic right CY structure $K_{c,V}^*\uptau_{\upnu|_V}$ as before.
  Since $j_!$ is a strict morphism from $\frg_{\cE|_V,c}$ to $\frg_{\cE|_U,c}$ satisfying
  \[
    (j_!)^*\uptau_\upnu(\upxi) = \int_U \upnu \wedge \tr(j_!\upxi) = \int_V \upnu|_V \wedge \tr(\upxi) = \uptau_{\upnu|_V}(\upxi).
  \]
  Now since $K_c\diamond i_! \sim_\an \id_{\frg_{\cE|_U,c}}$ and $i_!\diamond j_!\diamond K_{c,V} \sim_\an \id_{\frg_\cE}$ it follows by \cref{analhochschildhomot} that
  \[
    \begin{aligned}
      [K_{c,V}^*\uptau_{\upnu|_V}]_\an
      = [(j_!\diamond K_{c,V})^*\uptau_\upnu]_\an
      = [(K_c \diamond i_! \diamond j_!\diamond K_{c,V})^*\uptau_\upnu]_\an
      = [K_c^*\uptau_\upnu]_\an = [\upphi^\upnu]_\an.
    \end{aligned}
  \]
  By \cref{mainthm}(2) the cyclic minimal model induced by $K_{c,V}^*\uptau_{\upnu|_V}$ is then cyclic analytic $A_\infty$-isomorphic to $\cH^\upnu_\cE$.
\end{proof}

In the 3-Calabi--Yau case, a potential can be constructed using \cref{potentialsmaincor}.

\begin{cor}\label{potentialsvolumecor}
  Suppose $X$ is a threefold, then any volume germ $\upnu$ induces an analytic potential
  \[
    W^\upnu \in \widetilde\barT_{\kern-1pt\ell} \Ext^1(\cE,\cE)^\vee,
  \]
  which is well-defined up to an analytic change of coordinates $\widetilde\barT_{\kern-1pt\ell} \Ext^1(\cE,\cE)^\vee \xrightarrow{\sim} \widetilde\barT_{\kern-1pt\ell} \Ext^1(\cE,\cE)^\vee$.
\end{cor}

\begin{rem}
  If $X$ is itself Calabi--Yau then the global volume form $\upnu \in \Upgamma(X,\Omega^n_X)$ is  $L^2$-integrable and $(u,h) = (\id_\cE,0)$ is a valid unit/coboundary inducing the identity maps $K_c = \id$ on $\frg_{\cE,c} = \frg_\cE$.
  The trace $K_c^*\uptau_\upnu = \uptau_\upnu$ induces the Serre pairing, and $(\cH^\upnu,\upsigma^\upnu)$ is simply the cyclic analytic minimal model of \cite{Pol01,Tu14,Tod18} as given in \cref{policyc}.
\end{rem}

The cyclic minimal models obtained from the above theorem express the local geometry on an open subset, but are always defined with respect to a global choice of projective variety.
We claim however, that the cyclic minimal model really only depends on the local geometry.
To substantiate this claim, we will consider diagrams of the form
\begin{equation}\label{eq:embeddings}
  \begin{tikzcd}
    X & Y\ar[l,hook',"i",swap] \ar[r,"f"] & X'
  \end{tikzcd}
\end{equation}
where $Y$ is an open analytic subvariety in a smooth projective varieties $X$, and $f$ is an open embedding into a second smooth projective variety $X'$.
We claim that the cyclic minimal model for a perfect complex with support on $Y$ can be computed equivalently on $X'$ or $X$.

\begin{thm}\label{embeddingminmod}
  In the situation of \cref{eq:embeddings}, let $\cE' \in \D^\perf(X')$ have support $f(Z) \subset f(Y)$ for some compact $Z\subset Y$, let $\upnu \in (\Omega^n_{X'})_Z$ be a volume form germ, and let $\cE\in \D^\perf(X)$ be such that $\cE|_Y \simeq f^*\cE'$.
  Then there exists an analytic homotopy equivalence
  \[
    \frg_{\cE} \simeq_\an \frg_{\cE'}
  \]
  along which $\upphi^\upnu$ pulls back to a an analytic right CY structure analytically homotopic to $\upphi^{f^*\upphi}$.
  In particular, there is an analytic cyclic $A_\infty$-isomorphism between the cyclic analytic minimal models
  \[
    (\cH_\cE^{f^*\upnu},\upsigma^{f^*\upnu}) \cong_{\an,\cyc} (\cH_{\cE'}^\upnu,\upsigma^\upnu).
  \]
\end{thm}
\begin{proof}
  Given a germ in $(\Omega^n_{X'})_{f(Z)}$ we can find a representative $\upnu \in \Upgamma(f(U),\Omega^n_{X'})$ defined on an the image of some open neighbourhood $U\supset Z$.
  As before, this neighbourhood can be chosen so that $\upnu$ and $f^*\upnu$ are $L^2$-integrable forms on $f(U)$ and $U$ respectively, yielding $\uptau_\upnu\in \barC_\lambda^{\bullet,\an}(\frg_{\cE|_{f(U)},c})$ and $\uptau_{f^*\upnu}\in \barC_\lambda^{\bullet,\an}(\frg_{\cE|_U,c})$.
  Now for any choice of complexes $\cE$ and $\cE'$ as above, we can consider the chain of quasi-isomorphisms
  \begin{equation}\label{eq:chainofqisos}
    \begin{tikzcd}
      \frg_\cE \ar[r,"K_c"] & \frg_{\cE|_U,c} \ar[r,"F"] & \frg_{f^*\cE'|_U,c} \ar[r,"(f^*)^{-1}"]& \frg_{\cE'|_{f(U)},c},
    \end{tikzcd}
  \end{equation}
  where $F\colon \frg_{\cE|_U,c} \to \frg_{f^*\cE'|_U,c}$ is the map in \cref{arbcompqiso} associated to a triple $(r,r^{-1},h_{rr^{-1}})$ of mutually inverse quasi-isomorphisms $r^{-1}\colon\cE|_U\to f^*\cE'|_U$ and $r^{-1}\colon f^*\cE'|_U\to \cE|_U$ with coboundary $h_{rr^{-1}}$ for the composition $r\wedge r^{-1}$, and $(f^*)^{-1}$ denotes the DG algebra morphism given by the inverse of the pullback map
  \[
    f^*\colon \Upgamma_c(f(U),\sA^{0,\bullet}_X(\cE nd(\cE'))) \to \Upgamma_c(U,\sA^{0,\bullet}_Y(\cE nd(f^*\cE'))),
  \]
  which is a DG algebra isomorphism because $f$ is a diffeomorphism onto its image.
  Choosing hermitian metrics on $\cE$ and $\cE'$, there is an induced hermitian metric on $f^*\cE'$, and the maps $K_c$ and $F$ are analytic.
  The strict map $f^*$ is again bounded provided that $U$ is chosen sufficiently small, as the ratio between the metric on $U$ and the metric pulled back from $f(U)$ is again bounded on a neighbourhood of the compact subset $Z$.
  
  Now we observe that the functional $\uptau_\upnu$ pulls back along the quasi-isomorphism $(f^*)^{-1} \diamond F$  to an analytic negative cyclic cocycle of the form  
  \[
  \begin{aligned}
    (((f^*)^{-1} \diamond F)^*\uptau_\upnu)_n(s\upxi_1,\ldots,s\upxi_n)
    &= \int_{f(U)} \upnu \wedge \tr((f^*)^{-1}(r^{-1} \wedge \upxi_1 \wedge h_{rr^{-1}} \wedge \cdots \wedge \upxi_n \wedge r))
    \\&= \int_{U} f^*\upnu \wedge \tr(r^{-1} \wedge \upxi_1 \wedge h_{rr^{-1}} \wedge \cdots \wedge \upxi_n \wedge r)
    \\&= \int_{U} f^*\upnu \wedge \tr(r\wedge r^{-1} \wedge\upxi_1 \wedge h_{rr^{-1}} \wedge \cdots \wedge \upxi_n)
    \\&= ((K_c\diamond i_!)^*\uptau_{f^*\upnu})_n(s\upxi_1,\ldots,s\upxi_n),
  \end{aligned}
\]
where $K_c$ is the map associated to the unit/coboundary pair $(u,h) = (r\wedge r^{-1}, h_{rr^{-1}})$ constructed in \cref{qisotheotherway}.
Since $K_c\diamond i_!\sim_\an \id_{\frg_{\cE|_U,c}}$ it follows that $((f^*)^{-1}\diamond F)^*\uptau_\upnu \sim_\an \uptau_{f^*\upnu}$.
Now let $j\colon f(U) \hookrightarrow X'$ denote the inclusion of $f(U)$, and let $K_c'$ be the homotopy inverse of $j_!$ defining the analytic right CY structure $\upphi^\upnu = (K_c')^*\uptau_\upnu$.
Then
\[
  j_! \diamond (f^*)^{-1} \diamond F \diamond K_c \in \Hom_{\AinfAlgan}^\an(\frg_\cE,\frg_{\cE'})
\]
is an analytic quasi-isomorphism for which the pullback of the cocycle $\upphi^\upnu$ satisfies
\[
  \begin{aligned}
    (j_! \diamond (f^*)^{-1} \diamond F \diamond K_c)^*\upphi^\upnu
    &= (K_c' \diamond j_! \diamond (f^*)^{-1} \diamond F \diamond K_c)^*\uptau_\upnu
    \\&\sim_\an ((f^*)^{-1} \diamond F \diamond K_c)^*\uptau_\upnu
    \\&= K_c^*((f^*)^{-1} \diamond F)^*\uptau_\upnu
    \\&\sim_\an K_c^*\uptau_{f^*\upnu} = \upphi^{f^*\upnu}.
  \end{aligned}
\]
It now follows by \cref{mainthm}(3) that there is a cyclic analytic $A_\infty$-isomorphism between the cyclic analytic minimal models, which finishes the proof.
\end{proof}

In the threefold case, the second part of \cref{potentialsmaincor} now directly implies the following.

\begin{cor}\label{potentialsembedcor}
  Suppose $X$ and $X'$ are threefolds with complexes $\cE$ and $\cE'$ as above.
  Then for any volume germ $\upnu \in (\Omega_{X'}^3)_Z$ there is an algebra isomorphism
  \[
    g\colon \widetilde\barT_{\kern-1pt\ell} \Ext^1(\cE',\cE')^\vee \xrightarrow{\ \sim\ } \widetilde\barT_{\kern-1pt\ell} \Ext^1(\cE,\cE)^\vee.
  \]
  which relates the two potentials via $g(W^\upnu) = W^{f^*\upnu}$
\end{cor}

\subsection{The example of a point}
We will give an explicit computation of the cyclic minimal model associated to the point sheaf of the origin in $\A^n$, which is a quasi-projective Calabi--Yau with standard volume form $\omega = \d z_1 \wedge \cdots \wedge \d z_n$ in standard coordinates.

Let $\P^n \subset \A^n$ be the projective compactification with coordinates $z_0,z_1,\ldots,z_n$ and write $o \in \A^n \subset \P^n$ for the origin.
Writing $V = (T_o\A^n)^* \cong \C^n$, the point sheaf $\O_o$ is resolved by the Koszul complex
\[
  \cE \colonequals \quad \wedge^n (V \otimes_\C \O_{\bbP^n}(-1)) \xrightarrow{\ \updelta\ } \cdots \xrightarrow{\ \updelta\ } \wedge^2 (V \otimes_\C \O_{\bbP^n}(-1)) \xrightarrow{\ \updelta\ } V \otimes_\C \O_{\bbP^n}(-1) \xrightarrow{\ \updelta\ } \O_{\bbP^n},
\]
where the differential acts as $v \otimes f \mapsto \sum_{i=1}^n v(\partial_{z_i}) z_i \cdot f$ on $\cE^{-1} = V\otimes_\C \O_{\bbP^n}(-1)$ and acts on the other terms by extension over the wedge product.
We endow $\cE$ with a hermitian structure from the Fubini-Study metric on $\bbP^n$, yielding a Dolbeault DG algebra $\frg_\cE$ as a model for $\Ext^\bullet(\O_o,\O_o)$.
It is known (see e.g. \cite[Lemma 3.6]{RS21} for the threefold case) that the algebra of polyvectors
\[
  \frh \colonequals \left(\bigoplus_{i=0}^n (T_o\A^n)^{\wedge i}, \wedge \right),
\]
at $o$ forms an $A_\infty$-minimal model for $\frg_\cE$.
Here we include a proof for the analytic setting.

\begin{lem}
  There is an analytic DG algebra quasi-isomorphism $I \colon \frh \xrightarrow{\sim} \frg_\cE$, which makes $\frh$ into an analytic minimal model for $\frg_\cE$.
\end{lem}
\begin{proof}
  It suffices to define $I(\xi)$ for $\xi \in T_o\A^n$ and check that $I(\xi)\wedge I(\xi) = 0$ and $\Dol I(\xi) = 0$ hold.
  For a vector $\xi\in T_o\A^n$ the element $I(\xi)$ is defined as the holomorphic endomorphism acting by $v\otimes f \mapsto v(\xi)\cdot f \in \cE^0$ on $\cE^{-1}$ and extended to maps $\cE^{-k} \mapsto \cE^{1-k}$ over the wedge product.
  A standard computation shows that $I(\xi) \wedge I(\xi) = I(\xi) \circ I(\xi)$ acts by
  \[
    \begin{aligned}
      (I(\xi)\circ I(\xi))(v_1 \otimes f_1 \wedge \cdots \wedge v_k \otimes f_k)
      &= \sum_{i=1}^n (-1)^i v_i(\xi)f_i \cdot I(\xi)(v_1\otimes f_1 \wedge \cdots \widehat{v_i\otimes f_i}\cdots \wedge v_k \otimes f_k)\\
      &= \sum_{i<j} ((-1)^{i+j} v_i(\xi)v_j(\xi)f_if_j + (-1)^{i+j-1} v_i(\xi)v_j(\xi)f_if_j) \\
      &\quad\quad\quad\cdot (v_1\otimes f_1 \wedge \cdots  \cdots \widehat{v_i\otimes f_i} \cdots \widehat{v_j\otimes f_j}\cdots \wedge v_k \otimes f_k))\\
      &=0,
    \end{aligned}
  \]
  where the hats indicate omission.
  It follows that $I$ is an algebra morphism, and a similar argument shows that $\Dol(I(\xi)) = 0$, making $I$ into an analytic DG algebra morphism.
  It is moreover a quasi-isomorphism because the dimension of $\frh$ is $\dim_\C \Ext^i(\O_o,\O_o) = \dim_\C (T_o\A^n)^{\wedge i}$.
  By \cref{uniqanalminmod} it follows that $\frh$ is equivalent to the analytic minimal model $\cH_\cE$ of $\frg_\cE$.
\end{proof}

It follows easily from \cite[Lemma 11.2]{VdBer15} that $\frh$ admits a cyclic structure: for every choice of trace $\uplambda\colon \frh^n \to \C$ there is a cyclic structure of the form
\[
  \upsigma(s\upxi_1)(s\upxi_2) \colonequals \uplambda(\upxi_1\wedge \upxi_2).
\]
Viewing $\uplambda$ as a cocycle in $\barC_\lambda^{\bullet,\an}(\frh)$, one checks that it corresponds to $\upsigma$ along the map in \cref{analCYlifts}.
The following lemma relates $\uplambda$ to an analytic right CY structure on $\frg_{\cE|_U,c}$ induced by the volume $\omega$.

\begin{lem}\label{uhpairfortrace}
  There exists an open $U\subset \A^n$ and unit/coboundary pair $(u,h)$ for $\frg_{\cE|_U,c} \subset \frg_\cE$ with associated quasi-isomorphism $K_c$,
  such that $(K_c\diamond I)^*\uptau_{\omega|_U}$ is analytically homotopic to a trace $\uplambda$.
\end{lem}
\begin{proof}
  Let $p\colon \R \to \R$ be the smooth function with compact support in $[-1,1]$ defined by
  \[
    p(t) = \begin{cases}
      e^{\frac{-t^2}{1-t^2}} & \text{if } t \in [-1,1],\\
      0 & \text{otherwise}.
    \end{cases}
  \]
  Then $q(t) = t^{-1}(p(t) - 1)$ is another smooth function $q\colon \R\to\R$ and we define for each $i = 1,\ldots,n$ the smooth sections $u_i \in \frg_{\cE|_{\A^n}}^0$ and $h_i\in \frg_{\cE|_{\A^n}}^{-1}$ by the formulas
  \[
    \begin{aligned}
      u_i(z) &= p(|z_i|^2) \cdot \id_{\cE|_{\A^n}} - p'(|z_i|^2)(\partial_{z_i}^* \wedge -) \cdot \d \overline z_i,\\
      h_i(z) &= q(|z_i|^2)\overline z_i (\partial_{z_i}^*\wedge -),
    \end{aligned}
  \]
  where $\partial_{z_i}^* \in (T_o\A^n)^* = V$ is the dual of the standard basis vector $\partial_{z_i}$.
  By inspection, we have
  \[
    \begin{aligned}
      \Dol h_i(z) &= q(|z_i|^2)|z_i|^2 \cdot \id_{\cE|_{\A^n}} + (q(|z_i|^2) + q'(|z_i|^2)|z_i|^2)(\partial^*_{z_i}\wedge-) \d \overline z_i\\
      &= (p(|z_i|^2) - 1) \cdot \id_{\cE|_{\A^n}} - p'(|z_i|^2)(\partial^*_{z_i}\wedge-) \d \overline z_i\\
      &= u_i(z) - \id_{\cE|_{\A^n}}.
    \end{aligned}
  \]
  The element $u = u_1 \wedge \cdots \wedge u_n$ has the closed polydisc $\{z \mid |z_i| \leq 1 \ \forall i\}$ as its support and can therefore be interpreted as an element of $\frg_{\cE|_U,c}$ for a bounded open neighbourhood $U$ of said polydisc.
  It moreover satisfies $u = \id_{\cE|_{\A^n}} + \Dol \widetilde h$ with respect to
  \[
    \widetilde h = \sum_{k=1}^n h_k + \sum_{i < j} h_i \wedge \Dol h_j + \ldots + \sum_{j_1 < \cdots < j_{n-1}} h_{j_1} \wedge \Dol h_{j_2} \wedge \cdots \wedge \Dol h_{j_n} +  h_1 \wedge \Dol h_2 \wedge \cdots\wedge \Dol h_n.
  \]
  We want to extend $(u,\widetilde h)$ to a unit/coboundary pair on all of $\P^n$, but note that $\widetilde h$ blows up at infinity.
  To remedy this, we can pick any smooth bump function $\phi$ on $\A^n$ with $\phi|_U = 1$.
  Then $\phi\cdot\widetilde h$ has compact support, and hence extends to a form $i_!(\phi\cdot \widetilde h)$ on $\P^n$ which satisfies
  \[
    \Dol(i_!(\phi\cdot \widetilde h)) = \phi(i_!u - \id_\cE) + i_!\overline\partial \phi \cdot \widetilde h = i_!u - \id_\cE + ((1-\phi) \id_\cE + i_!(\overline\partial \phi\cdot \widetilde h)).
  \]
  Now the element $(1-\phi) \id_\cE + i_!(\overline\partial \phi\cdot\widetilde h) \in \frg_\cE^0$ has support in $\P^n \setminus U$, and is therefore given by $\Dol\upbeta$ for some $\upbeta\in \frg_\cE^{-1}$ with compact support in $\P^n\setminus U$.
  Taking $h = i_!(\phi\cdot \widetilde h) - \upbeta$ then yields the required extension, making $(u,h)$ into a unit/coboundary pair with $h|_U = \widetilde h$.
  Let $K_c \colon \frg_\cE \to \frg_{\cE|_U,c}$ be the morphism obtained from \ref{qisotheotherway}, so that $I^*K_c^*\uptau_{\omega|_U}$ is an analytic negative cyclic cocycle for $\frh$.
  A computation in polar coordinates shows that the zeroth component maps the shift of $\xi = \partial_{z_1} \wedge \cdots \wedge \partial_{z_n} \in \frh^n$ to
\[
  \begin{aligned}
    (I^*K_c^*\uptau_{\omega|_U})_0(s\xi)
  &= \int_U p'(|z_1|^2)\cdots p'(|z_n|^2) \cdot \d z_1\cdots \d z_n \d \overline z_1\cdots\d \overline z_n.
    \\&= (-1)^n \int_0^{2\pi}\int_0^1 p'(r_1^2) r_1 \d r_1\d \theta_1 \cdots \int_0^{2\pi}\int_0^1 p'(r_n)^2 r_n \d r_n\d \theta_n
    \\&= (-1)^n (\pi p(1) - \pi \cdot p(0))^n = \pi^n.
  \end{aligned}
\]
A similar computation shows that the higher components of $I^*K_c^*\uptau_{\omega|_U}$ vanish
\[
  \begin{aligned}
    (I^*K_c^*\uptau_{\omega|_U})_k(s\xi_1,\ldots,s\xi_k)(\xi_0)
    &= \int_{\A^n} \d z_1 \cdots  \d z_n  \tr(u \wedge I(\xi_0) \wedge h \wedge \cdots \wedge h \wedge I(\xi_k))
    \\&= \sum_{\substack{\alpha \in \N^n\\ |\alpha| = k}} \int_{\A^n} F_{\xi,\alpha}(|z_1|^2,\ldots,|z_n|^2) \overline z_1^{\alpha_1} \cdots \overline z_n^{\alpha_n} \d z_1 \cdots  \d z_n \d \overline z_1  \cdots  \d \overline z_n
    \\&=
    \sum_{\substack{\alpha\in\N^n\\|\alpha|=k}} (-1)^n\int_{[0,1]^n} F_{\xi,\alpha}(r_1^2,\ldots,r_n^2)r^{\alpha+1} \d r \cdot \prod_{j=1}^n \int_0^{2\pi} e^{-\alpha_j \theta_j \cdot i} \d \theta_j.
  \end{aligned}
\]
for some smooth functions $F_{\xi,\alpha}\colon \R^n \to \R$ supported on $[-1,1]^n$, which depend on $\xi$ and a multi-index $\alpha$ with $|\alpha| = \alpha_1+ \ldots + \alpha_n = k$.
For $k\geq 1$ there is at least one $\alpha_j \neq 0$, contributing a factor $\int_0^{2\pi} e^{-\alpha_j \theta_j \cdot i} \d\theta_j = 0$ which makes the term in the sum vanish. 
It follows that $(I^*K_c^*\uptau_{\omega|_U})_k = 0$ for $k\geq 1$, and therefore $\uplambda = I^*K_c^*\uptau_{\omega|_U} = (I^*K_c^*\uptau_{\omega|_U})_0$ is a linear functional.
\end{proof}

\begin{cor}
  The pair $(\frh,\upsigma)$ is equivalent to the cyclic analytic minimal model $(\cH_\cE^{\omega},\upsigma^\omega)$.
\end{cor}
\begin{proof}
  Letting $U$ and $K_c$ be as in \cref{uhpairfortrace} above, the map $K_c\diamond I$ makes $(\frh,\upsigma)$ into the cyclic analytic minimal model of $\frg_{\cE|_U,c}$ corresponding to the analytic right CY structure $\uptau_{\omega|_U}$ as in \cref{mainthm}.
  It follows that $K_c \in \Hom_{\AinfAlg}^\an(\frg_\cE,\frg_{\cE|_U,c})$ is a quasi-isomorphism such that
  \[
    [K_c^*\uptau_{\omega|_U}]_\an = [\upphi^\omega]_\an \in \HC_\lambda^{\an,\bullet}(\frg_\cE),
  \]
  is the canonical analytic right CY structure of \cref{cycminmodcan}.
  But then it follows from \cref{mainthm}(3) that there is a cyclic analytic $A_\infty$-isomorphism $(\frh,\upsigma) \cong_{\cyc,\an} (\cH_\cE^\omega,\upsigma^\omega)$.
\end{proof}

Because any point of a smooth variety has an analytic neighbourhood biholomorphic to an open analytic subset of $\A^n \subset \P^n$, we now find a similar result for the cyclic analytic minimal models of arbitrary points.

\begin{prop}\label{cyclicstructonpoint}
  Let $p\in X'$ be a closed point in a smooth projective variety $X'$ of dimension $n$, and let $\cE' \to \O_p$ be any resolution by a perfect complex.
  Then for \emph{every} holomorphic volume germ $\upnu \in (\Omega^d_X)_Z$ the cyclic analytic minimal model is cyclic-analytic $A_\infty$-isomorphic to $(\frh,\upsigma)$.
\end{prop}
\begin{proof}
  Let $Y \subset \A^n \subset \P^n$ be a sufficiently small neighbourhood of $o \in \A^n$ so that there is an open embedding $f\colon Y \to X'$ with $f(o) = p$.
  Then along this map $f^*\O_p \cong \O_o$ and the volume $\upnu$ pulls back to a form
  \[
    f^*\upnu = g(z_1,\ldots,z_n) \d z_1 \wedge \cdots \wedge \d z_n,
  \]
  for some nonvanishing holomorphic function $g$ on $Y$.  
  Let $G(z_1,\ldots,z_n)$ be a holomorphic function such that $\frac{\partial G}{\partial z_1} = g$ and $G(0,\ldots,0) = 0$.
  Then the map $\varphi\colon Y \to \A^n$ given by
  \[
    \varphi(z_1,\ldots,z_n) = (G(z_1,\ldots,z_n),z_2,\ldots,z_n)
  \]
  has a Jacobian $J(\varphi) = \det \frac{\partial G_i}{\partial z_j} = g$ which does not vanish on $Y$.
  Shrinking $Y$ if necessary, we obtain an open embedding $\varphi \colon Y \to \A^n$ such that $\varphi^*\O_o \cong \O_o$ and 
  \[
    \varphi^*\omega = \d G(z_1,\ldots,z_n) \wedge \cdots \wedge\d z_n = g(z_1,\ldots,z_n) \d z_1 \wedge \cdots \wedge \d z_n = f^*\upnu.
  \]
  Applying \cref{embeddingminmod} twice then yields cyclic-analytic $A_\infty$-isomorphisms
  \[
    (\frh,\upsigma) = \cH^\omega_\cE \cong_{\an,\cyc} \cH^{f^*\upnu}_{\cE} \cong_{\an,\cyc} \cH^\upnu_{\cE'}
    \qedhere
  \]
\end{proof}

\printbibliography%

\end{document}